\documentclass[a4paper,11pt]{article}

\usepackage{amsmath,amssymb,amsthm,mathrsfs,url,enumerate}

\usepackage[top=3cm, bottom=3cm, left=2.3cm, right=2.3cm]{geometry}

\newcommand{\TeXmacs}{T\kern-.1667em\lower.5ex\hbox{E}\kern-.125emX\kern-.1em\lower.5ex\hbox{\textsc{m\kern-.05ema\kern-.125emc\kern-.05ems}}}
\newcommand{\cdummy}{\cdot}
\newcommand{\mathD}{\mathrm{D}}
\newcommand{\mathd}{\mathrm{d}}

\newcommand{\nocomma}{}
\newcommand{\nosymbol}{}

\newcommand{\tmem}[1]{{\em #1\/}}

\newcommand{\tmop}[1]{\ensuremath{\operatorname{#1}}}

\newcommand{\tmtextit}[1]{{\itshape{#1}}}

\newenvironment{enumeratenumeric}{\begin{enumerate}[1.] }{\end{enumerate}}
\newenvironment{enumerateroman}{\begin{enumerate}[i.] }{\end{enumerate}}
\newtheorem{theorem}{Theorem}[section]
\newtheorem{corollary}[theorem]{Corollary}
\newtheorem{lemma}[theorem]{Lemma}
\newtheorem{definition}[theorem]{Definition}
\theoremstyle{remark}
\newtheorem{remark}[theorem]{Remark}

\newcommand{\R}{\mathbb{R}}
\newcommand{\N}{\mathbb{N}}
\newcommand{\Z}{\mathbb{Z}}
\newcommand{\E}{\mathbb{E}}
\renewcommand{\P}{\mathbb{P}}

\newcommand{\Lcal}{\mathcal{L}}
\renewcommand{\le}{\leqslant}

\newcommand{\CA}{\ensuremath{\mathscr{A}}}
\newcommand{\CB}{\ensuremath{\mathscr{B}}}
\newcommand{\CC}{\ensuremath{\mathscr{C}}}
\newcommand{\CD}{\ensuremath{\mathscr{D}}}

\newcommand{\CF}{\ensuremath{\mathscr{F}}}

\newcommand{\CS}{\ensuremath{\mathscr{S}}}

\newcommand{\lpara}{\,\mathord{\prec}\,}
\newcommand{\rpara}{\,\mathord{\succ}\,}
\newcommand{\reso}{\,\mathord{\circ}\,}

\newcommand{\rparacurl}{\, \mathord{\succcurlyeq} \,}
\newcommand{\mpara}{\,\mathord{\prec\!\!\!\prec}\,}

\begin{document}



\title{Paracontrolled distributions and singular PDEs}

\author{Massimiliano Gubinelli \\ CEREMADE \& CNRS UMR 7534, \\ Universit{\'e} Paris-Dauphine \\ and Institut~Universitaire de France \and 
Peter Imkeller \\ Institut f\"ur Mathematik, \\ Humboldt-Universit\"at zu Berlin \and
Nicolas Perkowski \\ CEREMADE \& CNRS UMR 7534,\\ Universit{\'e} Paris-Dauphine}

\date{}


\maketitle

\begin{abstract}
   We introduce an approach to study certain singular PDEs which is based on  techniques from paradifferential calculus and on ideas from the theory of controlled rough paths. We illustrate its applicability on some model problems like differential equations driven by fractional Brownian motion, a fractional Burgers type SPDE driven by space-time white noise, and a non-linear version of the parabolic Anderson model with a white noise potential.
\end{abstract}




\section{Introduction}

In this paper we introduce the notion of \emph{paracontrolled distribution} and show how to use it to give a meaning to and solve partial differential equations involving non-linear operations on generalized functions.  More precisely, we combine the idea of \emph{controlled paths}, introduced  in~\cite{Gubinelli2004}, with the \emph{paraproduct} and the related paradifferential calculus introduced by Bony~\cite{Bony1981}, in order to develop a non-linear theory for a certain class of distributions.

The approach presented here works for generalized functions defined on an index set of arbitrary dimension and constitutes a flexible and lightweight generalization of Lyons' rough path theory~\cite{Lyons1998}. In particular it allows to handle problems involving singular stochastic PDEs which were substantially out of reach with previously known methods.

In order to set the stage for our analysis let us list some of the problems which are amenable to be analyzed in the paracontrolled framework:
\begin{enumerate}
  \item The \emph{rough} differential equation (\textsc{rde}) driven by an $n$--dimensional Gaussian process $X$: 
  \[
     \partial_t u(t) = F(u(t)) \partial_t X(t),
  \]
  where $F\colon\R^{d} \rightarrow \Lcal(\R^n, \R^d)$ is a smooth vector-field . Typically, $X$ will be a Brownian motion or a fractional Brownian motion with Hurst exponent $H\in(0,1)$. The paracontrolled analysis works up to $H>1/3$.  While we do not have any substantial new results for this problem, it is a useful pedagogical example on which we can easily describe our approach.
  
  \item Generalizations of Hairer's Burgers-like SPDE (\textsc{burgers}):
  \[
     L u = G(u) \partial_x u + \xi.
  \]
  Here $u\colon \R_+ \times \mathbb{T} \rightarrow \R^n$, where $\mathbb{T} = (\R/2\pi \Z)$ denotes the torus, $L = \partial_t + (-\Delta)^{\sigma}$, where $-(-\Delta)^\sigma$ is the fractional Laplacian with periodic boundary conditions and we will take $\sigma > 5/6$, and $\xi$ is a space-time white noise with values in $\R^n$. Moreover, $G\colon \R^n \rightarrow \Lcal (\R^n, \R^n)$ is a smooth field of linear transformations.
  
  \item A non-linear generalization of the parabolic Anderson model (\textsc{pam}):
  \[
     L u = F(u) \diamond \xi,
  \]
  where $u \colon \R_+ \times \mathbb{T}^2 \rightarrow \R$, $L = \partial_t - \Delta$ is the parabolic operator corresponding to the heat equation, and where $\xi$ is a random potential which is sampled according to the law of the white noise on  $\mathbb{T}^2$ and is therefore independent of the time variable. We allow for a general smooth function $F\colon \R \rightarrow \R$, the linear case $F(u)=u$ corresponding to the standard parabolic Anderson model. The symbol $\diamond$ stands for a renormalized product which is necessary to have a well defined problem. 
  
  \item The one-dimensional periodic Kardar--Parisi--Zhang equation (\textsc{kpz}):
  \[
   L h = ``(\partial_x h)^2" + \xi,
  \]
    where $u \colon \R_+ \times \mathbb{T} \rightarrow \R$, $L = \partial_t - \Delta$, and where $\xi$ is a space-time white noise. Here $``(\partial_x h)^2"$ denotes the necessity of an additive renormalization in the definition of the square of the distribution $\partial_x h$.

  \item The three-dimensional, periodic,  stochastic quantization equation for the $(\phi)^4_3$ euclidean quantum field (\textsc{sq}):
  \[
   L \phi = ``\frac{\lambda}{4!} (\phi)^3" + \xi,
  \]
    where $\phi \colon \R_+ \times \mathbb{T}^3 \rightarrow \R$, $L = \partial_t - \Delta$, $\xi$ is a space-time white noise, and where $``(\phi)^3"$ denotes a suitable renormalization of a cubic polynomial of $\phi$ and $\lambda$ is the coupling constant of the scalar theory.
  
\end{enumerate}

In this paper we will consider in detail the three cases~\textsc{rde}, \textsc{burgers}, \textsc{pam}. In all cases we will exhibit a space of paracontrolled distributions where the equations are well posed (in a suitable sense), and admit at least a local in time solution which is unique.
The three-dimensional stochastic quantization equation \textsc{sq} is studied by R.~Catellier and K.~Chouk in~\cite{Catellier2013} by applying the paracontrolled technique. The paracontrolled analysis of \textsc{kpz}  will be presented elsewhere~\cite{Gubinelli-KPZ-2013}.

\bigskip

The kind of results which will be obtained below can be exemplified by the following statement for \textsc{rde}s. Below $\CC^\alpha = B^\alpha_{\infty, \infty}$ stands for the H\"older-Besov space of index $\alpha$ on $\mathbb{R}$. Given two distributions $f\in \CC^\alpha$ and $g\in \CC^\beta$ with $\alpha+\beta >0$ we can always consider a certain distribution $f \reso g$ which is obtained via a bilinear operation of $f,g$ and which belongs to $\CC^{\alpha+\beta}$.
\begin{theorem}
Let $\xi:[0,1]\to \mathbb{R}^n$ be a continuous function and $F:\mathbb{R}^d\to \mathcal{L}(\mathbb{R}^n, \mathbb{R}^d)$ be a family of  smooth vector-fields. Let $u:[0,1]\to \mathbb{R}^d$ be a 
 solution of the Cauchy problem
\[
\partial_t u(t) = F(u(t)) \xi(t),\qquad
u(0)=u_0,
\]
where $u_0 \in \R^d$. Let $\vartheta$ be a solution to $\partial_t \vartheta = \xi$ and let $R\xi=(\xi,\vartheta\reso \xi)$. Then for all $\alpha \in (1/3,1)$ there exists a continuous map $\Psi:\mathbb{R}^d\times\CC^{\alpha-1}\times\CC^{2\alpha-1}\to \CC^{\alpha}$ such that $u=\Psi(u_0,R\xi)$ for all $\xi\in C([0,1];\mathbb{R}^d)$.
\end{theorem}

In particular, this theorem provides a natural way of extending the solution map to data $\xi$ which are merely distributions in $\CC^{\alpha-1}$. It suffices to approximate $\xi$ by a sequence of smooth functions $(\xi^n)$ converging to $\xi$ in $\CC^{\alpha-1}$, and to prove that the ``lifted" sequence $(R\xi^n)$ converges to some limit in $ \CC^{\alpha-1}\times\CC^{2\alpha-1}$. The uniqueness of this limit is not guaranteed however, and each possible limit will give rise to a different notion of solution to the \textsc{rde}, just like in standard rough path theory. 

The space $\mathcal{X}$ obtained by taking the closure in  $\CC^{\alpha-1}\times\CC^{2\alpha-1}$ of the set of all elements of the form $R\xi$ for smooth $\xi$ replaces the space of (geometric) rough paths, and the above theorem is a partial restatement of Lyons' continuity result: namely that the (It\^o) solution map $\Psi$, going from data to solution of the differential equation, is a continuous map from the rough path space $\mathcal{X}$ to $\CC^\alpha$. The space $\mathcal{X}$ is fibered over $\CC^{\alpha-1}$. It allows us to equip the driving distribution with enough information to control the continuity of the solution map to our \textsc{rde} problem -- and as we will see below, also the continuity of the solution maps to suitable PDEs. In various contexts the space $\mathcal{X}$ can take different forms, and in general it does not seem to have the rich geometrical and algebraic structure of standard rough paths. 

The verification that suitable approximations $(\xi^n)$ are such that their lifts $(R\xi^n)$  converge in $ \CC^{\alpha-1}\times\CC^{2\alpha-1}$ depends on the particular form of $\xi$. In the case of $\xi$ being a Gaussian stochastic process (like in all our examples above), this verification is the result of almost sure convergence results for elements in a fixed chaos of an underlying Gaussian process, and the proofs rely on elementary arguments on Gaussian random variables.   

Even in the case of \textsc{rde}s, the paracontrolled analysis leads to some interesting insights. For example, we have that a more general equation of the form
$$
\partial_t u(t) = F(u(t)) \xi(t) + F'(u(t))F(u(t))\eta(t) ,\qquad
u(0)=u_0,
$$
where $\eta \in C([0,1];\mathbb{R}^n \times \mathbb{R}^n)$,
has a solution map which depends continuously on $(\xi, \vartheta\reso \xi +\eta)\in\CC^{\alpha-1}\times \CC^{2\alpha-1}$. The remarkable fact here is that the solution map depends only on the combination $\vartheta\reso \xi +\eta$ and not on each term separately. Such structural features of the solution map, which can be easily seen using the paracontrolled analysis, are very important in situations where renormalizations are needed, as for example in the \textsc{pam} model. In the \textsc{rde} context we can simply remark that setting $\eta = -\vartheta\reso \xi$, the solution map becomes a continuous function of $\xi \in \CC^{\alpha-1}$, without any further requirement on the bilinear object $\vartheta\reso \xi$. Thus, the equation
$$
\partial_t u(t) = F(u(t)) \xi(t) - F'(u(t))F(u(t))(\vartheta\reso \xi)(t) ,\qquad
u(0)=u_0,
$$
can be readily extended to any $\xi\in\CC^{\alpha-1}$ by continuity. 

We should however point out a limitation of our approach: While in rough path theory one can deal with more irregular paths than $\vartheta \in \CC^{\alpha}$ for $\alpha > 1/3$ and in fact $\alpha>0$ can be chosen arbitrarily close to $0$ as long as sufficiently many iterated integrals of $\vartheta$ are given, with paracontrolled distributions we are currently only able to perform ``first order expansions'' and are therefore restricted to the case $\alpha>1/3$.

\bigskip

We remark that, even if only quite implicitly, paraproducts  have been already exploited in the rough path context in the work of Unterberger on the renormalization of rough paths~\cite{Unterberger2010a, Unterberger2010}, where it is referred to as ``Fourier normal-ordering'', and in the related work of Nualart and Tindel~{\cite{Nualart2011}}.

In this paper we construct weak solutions for the SPDEs under consideration. For an approach using mild solutions see~\cite{Perkowski2014Thesis}. See also~\cite{Gubinelli2013}, where we use the decomposition of continuous functions in a certain wavelet series and similar ideas as developed below, in order to give a new and relatively elementary approach to rough path integration.

\begin{remark}
Various versions of this paper have been available as online preprints since October 2012. Since also the content changed slightly from iteration to iteration, this might cause some confusion. We therefore point out the main differences between the first and the current version:
\begin{itemize}
   \item[--] We changed the notation, writing $f \lpara g$ rather than $\pi_<(f,g)$ and similarly for $\rpara$ and $\reso$. In the first version, we defined ``controlled distributions'', while by now we prefer the terminology ``paracontrolled distributions''.
   \item[--] We now work with weak solutions, rather than mild solutions as in the first version. In particular, the paracontrolled ansatz (see e.g.~\eqref{eq:paracontrolled ansatz}) is new. This has the advantage that we no longer need to control the commutator between heat kernel and paraproduct, but the disadvantage that we need to consider a modified paraproduct when solving \textsc{pam} (see~\eqref{eq:modpara}).
   \item[--] Section~\ref{sec:ODE} on \textsc{rde}s is new. 
   \item[--] The ``conditional global existence result'' for \textsc{pam} (see Theorem~\ref{theorem:pam existence uniqueness}) is new.
   \item[--] We have included Section~\ref{sec:regularity} which is a first attempt at creating a link between paracontrolled distributions and Hairer's regularity structures.
\end{itemize}
\end{remark}

\paragraph{Relevant literature.}

Before going into the details, let us describe the context of our work. Consider for example the \textsc{rde} problem above. 
Schwartz' theory of distributions gives a robust framework for defining linear operations on irregular generalized functions. But when trying to handle non-linear operations, we quickly run into problems. For example, in Schwartz' theory, it is not possible to define the product $F(u) \partial_t X(t)$ in the case where $X$ is the sample path of a Brownian motion. The standard analysis of this difficulty goes as follows: $X$ is an $\alpha$--H\"older continuous process for any $\alpha<1/2$, but not better. The solution $u$ has to have the same regularity, which is transferred to $F(u)$ if $F$ is smooth. In this situation, the product $F(u)\partial_t X$ corresponds to the product of an $\alpha$--H\"older continuous function with the distribution $\partial_t X$ which is of order $\alpha-1$. A well known result of analysis (see Section~\ref{sec:bony} below) tells us that a necessary condition for this product to be well defined is that the sum of the orders is positive, that is $2\alpha-1>0$, which is barely violated in the Brownian setting.
This is the classical problem which motivated It\^o's theory of stochastic integrals.

It\^{o}'s integral has however quite stringent structural requirements: an ``arrow of time'' (i.e. a filtration and adapted integrands), a probability measure (it is defined as $L^2$--limit), and $L^2$--orthogonal increments of the integrator (the integrator needs to be a (semi-) martingale).

If one or several of these assumptions are violated, then Lyons' rough path integral~\cite{Lyons1998, Lyons2002, Lyons2007, Friz2010} can be an effective alternative. For example, it allows to construct pathwise  integrals for, among other processes, fractional Brownian motion, which is not a semimartingale.


In the last years, several other works applied rough path techniques to SPDEs. 
But they all relied on special features of the problem at hand in order to apply the integration theory provided by the rough path machinery. 

A first series of works attempts to deal with ``time"-like irregularities by adapting the standard rough path approach:

\begin{itemize}
\item[--]
Deya, Gubinelli, Lejay, and Tindel~\cite{Gubinelli2006,Gubinelli2012a,Deya2012} deal with SPDEs of the form
\begin{align*}
   L u (t, x) = \sigma ( u ( t, x ) ) \eta ( t,x),
\end{align*}
where $x \in \mathbb{T}$, $L = \partial_t - \Delta$, the noise $\eta$ is a space-time Gaussian distribution (for example white in time and colored in space), and $\sigma$ is some non-linear coefficient. They interpret this as an evolution equation (in time), taking values in a space of functions (with respect to the space variable). They extend the rough path machinery to handle the convolution integrals that appear when applying the heat flow to the noise.

\item[--]
Friz, Caruana, Diehl, and Oberhauser~\cite{Caruana2009,Caruana2011,Friz2011b,Diehl2012} deal with fully non-linear stochastic PDEs with a special structure. Among others, of the form
\begin{align*}
   \partial_t u ( t, x ) = F ( u, \partial_{x} u, \partial_{x}^2 u) + \sigma( t, x) \partial_{x} u( t, x) \eta(t),
\end{align*}
where the spatial index $x$ can be multidimensional, but the noise $\eta$ only depends on time. Such an SPDE can be reinterpreted as a standard PDE with random coefficients via a change of variables involving the flow of the stochastic characteristics associated to $\sigma$. This flow is handled using usual rough path results for \textsc{rde}s.

\item[--]
Teichmann~\cite{Teichmann2011} studies semilinear SPDEs of the form
\begin{align*}
   (\partial_t - A) u(t,x) = \sigma(u)(t,x) \eta(t,x),
\end{align*}
where $A$ is a suitable linear operator, in general unbounded, and $\sigma$ is a general non-linear operation on the unknown $u$ which however should satisfy some restrictive conditions. The SPDE is transformed into an SDE with bounded coefficients by applying a  transformation based on the group generated by $A$ on a suitable space.
\end{itemize}

The ``arrow of time'' condition of It\^o's integral is typically violated if the index is a spatial variable and not a temporal variable. Another series of works applied rough path integrals to deal with situations involving irregularities in the ``space" directions:

\begin{itemize}
\item[--]
Bessaih, Gubinelli, and Russo~\cite{Bessaih2005} and Brzezniak, Gubinelli, and Neklyudov \cite{Brzezniak2010} consider the vortex filament equation which describes the (approximate) motion of a closed vortex line $x(t,\cdot)\in C(\mathbb{T},\mathbb{R}^3)$ in an incompressible three-dimensional fluid:
\[
\partial_t x(t,\sigma) = u^{x(t,\cdot)}(x(t,\sigma)),\qquad u^{x(t,\cdot)}(y) = \int_{\mathbb{T}} K(y-x(t,\sigma)) \partial_\sigma x(t,\sigma) \mathd \sigma,
\]
where $K:\mathbb{R}^3 \to \mathcal{L}(\mathbb{R}^3,\mathbb{R}^3)$ is a smooth antisymmetric field of linear transformations of $\mathbb{R}^3$. In the modeling of turbulence it is interesting to study this equation with initial condition $x(0,\cdot)$ sampled according to the law of the three-dimensional Brownian bridge. In this case, the regularity of $x(t,\sigma)$ with respect to $\sigma$ is no better than Brownian for any positive time, and thus the integral in the definition of the velocity field $u^{x(t,\cdot)}$ is not well defined. Rough path theory allows to make sense of this integral and then of the equation.

\item[--]
Hairer, Maas, and Weber~\cite{Hairer2011, Hairer2013, Hairer2013b, Hairer2012} build on the insight of Hairer that rough path theory allows to make sense of SPDEs which are ill-defined in standard function spaces due to spatial irregularities. Hairer and Weber~\cite{Hairer2013} extend the \textsc{burgers} type SPDE that we presented above to the case of multiplicative noise. Hairer, Maas, and Weber~\cite{Hairer2012} study approximations to this equation, where they discretize the spatial derivative as $\partial_x u(t,x) \simeq 1/\varepsilon (u(t,x+\varepsilon) - u(t,x))$. They show that in the limit $\varepsilon \rightarrow 0$, the approximation may introduce a Stratonovich type correction term to the equation. Finally, Hairer~\cite{Hairer2013b} uses this approach to  define and solve for the first time the Kardar--Parisi--Zhang (KPZ) equation, an SPDE of one spatial index variable that describes the random growth of an interface. The KPZ equation was introduced by Kardar, Parisi, and Zhang~\cite{Kardar1986}, and prior to Hairer's work it could only be solved by applying a spatial transform (the Cole-Hopf transform) which had the effect of linearizing the equation.
\end{itemize}

\paragraph{Alternative approaches.}

In all the papers cited above, the intrinsic one-dimen\-sional nature of rough path theory severely limits possible improvements or applications to other contexts. To the best of our knowledge, the first attempt to remove these limitations is the still unpublished work by Chouk and Gubinelli~\cite{Chouk2013}, extending rough path theory to handle (fractional) Brownian sheets (Gaussian two-parameter stochastic processes akin to (fractional) Brownian motion).

In the recent paper~\cite{Hairer2013a}, Hairer has introduced a theory of \emph{regularity structures} with the aim of giving a more general and versatile notion of regularity. Hairer's theory is also inspired by the theory of controlled rough paths, and it can also be considered a generalization of it to functions of a multidimensional index variable. The crucial insight is that the regularity of the solution to an equation driven by -- say -- Gaussian space-time white noise should not be described in the classical way. Usually we say that a function is smooth if it can be approximated around every point by a polynomial of a given degree (the Taylor polynomial). Since the solution to an SPDE does not look like a polynomial at all, this is not the correct way of describing its regularity. We rather expect that the solution locally looks like the driving noise (more precisely like the noise convoluted with the Green kernel of the linear part of the equation; so in the case of \textsc{rde}s the time integral of the white noise, i.e. Brownian motion). Therefore, in Hairer's theory a function is called smooth if it can locally be well approximated by this convolution (and higher order terms depending on the noise). Hairer's notion of smoothness induces a natural topology in which the solutions to semilinear SPDEs depend continuously on the driving signal. This approach is very general, and allows to handle more complicated problems than the ones we are currently able to treat in the paracontrolled approach. If there is a merit in our approach, then its relative simplicity, the fact that it seems to be very adaptable so that it can be easily modified to treat problems with a different structure, and that we make the connection between harmonic analysis and rough paths.

\paragraph{Plan of the paper.}

Section~\ref{sec:paracontrolled} develops the calculus of paracontrolled distributions. In Section~\ref{sec:ODE} we solve ordinary differential equations driven by suitable Gaussian processes such as the fractional Brownian motion with Hurst index $H>1/3$. In Section~\ref{sec:burgers} we solve a fractional Burgers type equation driven by white noise, and in Section~\ref{sec:PAM} we study a non-linear version of the parabolic Anderson model. 
In \ref{sec:para} we recall the main concepts of Littlewood-Paley theory and of Bony's paraproduct,\ref{sec:convolution commutators} contains a commutator estimate between paraproduct and time integral, and in\ref{sec:paralin mod} we prove a modified version of the paralinearization theorem. We stress the fact that this paper is mostly self-contained, and in particular we will not need any results from rough path theory and just basic elements of the theory of Besov spaces.

\paragraph{Notation and conventions.}

Throughout the paper, we use the notation $a \lesssim b$ if there exists a constant $c>0$, independent of the variables under consideration, such that $a \leqslant c \cdot b$, and we write $a \simeq b$ if $a \lesssim b$ and $b \lesssim a$. If we want to emphasize the dependence of $c$ on the variable $x$, then we write $a(x) \lesssim_{x} b(x)$. For index variables $i$ and $j$ of Littlewood-Paley decompositions (see below) we write $i \lesssim j$ if $2^i \lesssim 2^j$, so in other words if there exists $N \in \mathbb{N}$, independent of $j$, such that $i \leqslant j +N$, and we write $i \sim j$ if $i \lesssim j$ and $j \lesssim i$.

An \tmtextit{annulus} is a set of the form $\CA = \{ x \in \mathbb{R}^{d} :a \leqslant | x | \leqslant b \}$ for some $0<a<b$. A {\tmem{ball}} is a set of the form $\CB = \{ x \in \mathbb{R}^{d} : | x | \leqslant b \}$. $\mathbb{T}= \R / (2\pi \Z)$ denotes the torus.

The H\"older-Besov space $B^\alpha_{\infty, \infty}(\R^d, \R^n)$ for $\alpha \in \R$ will be denoted by $\CC^\alpha$, equipped with the norm $\lVert \cdot \rVert_{\alpha}= \lVert \cdot \rVert_{B^\alpha_{\infty,\infty}}$. The local space $\CC^\alpha_\mathrm{loc}$ consists of all $u$ which satisfy $\varphi u \in \CC^\alpha$ for every infinitely differentiable $\varphi$ of compact support. 
Given $k \in \mathbb{N}$ and Banach spaces $X_{1} , \ldots ,X_{k}$ and $Y$, we write $\Lcal^{k} ( X_{1}\times \ldots \times X_{k} ,Y )$ for the space of $k$-linear maps from $X_{1} \times \ldots \times X_{k}$ to $Y$. 
For $T>0$ we write $C_{T} Y=C ( [ 0,T ] ,Y )$ for the space of continuous maps from $[ 0,T ]$ to $Y$, equipped with the supremum norm $\lVert \cdummy \rVert_{C_{T} Y}$. If $\alpha \in ( 0,1 )$, then we also define $C^{\alpha}_{T} Y$ as the space of $\alpha$-H{\"o}lder continuous functions from $[ 0,T ]$ to $Y$, endowed with the seminorm
\[
   \| f \|_{C^{\alpha}_{T} Y} = \sup_{0 \leqslant s<t \leqslant T} \frac{\| f( t ) -f ( s ) \|_{Y}}{| t-s |^{\alpha}}.
\]
If $f$ is a map from $A \subset \R$ to the linear space $Y$, then we write $f_{s,t} = f(t) - f(s)$. 
For  $f \in L^p(\mathbb{T})$ we write $\Vert  f(x) \Vert ^p_{L^p_x (\mathbb{T} )} = \int_{\mathbb{T}} | f ( x )|^p \mathd x$.

For a multi-index $\mu = ( \mu_{1} , \ldots , \mu_{d} ) \in \mathbb{N}^{d}$ we write $| \mu | = \mu_{1} + \ldots + \mu_{d}$ and $\partial^{\mu} = \partial^{| \mu |} / \partial_{x_{1}}^{\mu_{1}} \cdots \partial_{x_{d}}^{\mu_{d}}$. 
$\mathD F$ or $F'$ denote the total derivative of $F$. For $k \in \mathbb{N}$ we denote by $\mathD^{k} F$ the $k$-th order derivative of $F$. For $\alpha>0$, $C^{\alpha}_{b} =C^{\alpha}_{b} (\mathbb{R}^{d} ,\mathbb{R}^{n} )$ is the space of $\lfloor \alpha \rfloor$ times continuously differentiable functions, bounded with bounded partial derivatives, and with $(\alpha - \lfloor \alpha \rfloor)$--H\"older continuous partial derivatives of order $\lfloor \alpha \rfloor$, equipped with its usual norm $\lVert \cdot \rVert_{C^\alpha_b}$. 
We also write $\partial_{x}$ for the partial derivative in direction $x$, and if $F\colon \mathbb{R} \times\mathbb{R}^{d} \rightarrow \mathbb{R}^{n}$, then we write $\mathD_{x} F(t,x)$ for its spatial derivative in the point $( t,x ) \in \mathbb{R} \times \mathbb{R}^{d}$.

The space of real valued infinitely differentiable functions of compact support is denoted by $\CD ( \mathbb{R}^{d} )$ or $\CD$. The space of Schwartz functions is denoted by $\CS ( \mathbb{R}^{d} )$ or $\CS$. Its dual, the space of tempered distributions, is $\CS' ( \mathbb{R}^{d} )$ or $\CS'$. If $u$ is a vector of $n$ tempered distributions on $\mathbb{R}^{d}$, then we write $u \in \CS' (\mathbb{R}^{d} ,\mathbb{R}^{n} )$. The Fourier transform is defined with the normalization
\[
   \CF u ( z ) = \hat{u} ( z ) = \int_{\mathbb{R}^{d}} e^{- \iota \langle z,x \rangle} u ( x ) \mathd x,
\]
so that the inverse Fourier transform is given by $\CF^{-1} u ( z ) = ( 2 \pi)^{-d} \CF u ( -z )$. If $\varphi$ is a smooth function, such that $\varphi$ and all its partial derivatives are at most of polynomial growth at infinity, then we define the Fourier multiplier $\varphi ( \mathD )$ by $\varphi (\mathD ) u= \CF^{-1} ( \varphi \CF u)$ for any $u \in \CS'$. More generally, we define $\varphi ( \mathD ) u$ by this formula whenever the right hand side makes sense. The scaling operator $\Lambda$ on $\CS'$ is defined for $\lambda >0$ by $\Lambda_\lambda u = u(\lambda\cdot)$.

Throughout the paper, $(\chi, \rho)$ will denote a dyadic partition of unity, and $(\Delta_j)_{j \geqslant -1}$ will denote the Littlewood-Paley blocks associated to this partition of unity, i.e. $\Delta_{-1}= \chi(\mathD)$ and $\Delta_j = \rho(2^{-j} \mathD)$ for $j \geqslant 0$. We will often write $\rho_j$, by which we mean $\chi$ if $j=-1$, and we mean $\rho(2^{-j} \cdot)$ if $j \geqslant 0$. We also use the notation $S_j = \sum_{i < j} \Delta_i$.

\section{Paracontrolled calculus}\label{sec:paracontrolled}

\subsection{Bony's paraproduct}\label{sec:bony}

Paraproducts are bilinear operations introduced by Bony~\cite{Bony1981} in order to linearize a class of non-linear PDE problems. In this section we will introduce paraproducts to the extent of our needs. We will be using the Littlewood-Paley theory of Besov spaces. The reader can peruse \ref{sec:para}, where we summarize the basic elements of Besov space theory and Littlewood-Paley decompositions which will be needed in the remainder of the paper.

 One of the simplest situations where paraproducts appear naturally is in the analysis of the product of two Besov distributions. 
In general, the product $f g$ of two distributions $f \in \CC^{\alpha}$ and $g \in \CC^{\beta}$ is not well defined unless $\alpha + \beta >0$. In terms of Littlewood--Paley blocks, the product $f g$ can be (at least formally) decomposed as
\[
   fg= \sum_{j \geqslant -1} \sum_{i \geqslant -1} \Delta_{i} f \Delta_{j} g=f \lpara g+f \rpara g+f \reso g.
\]
Here $f \lpara g$ is the part of the double sum with $i<j-1$, and $f \rpara g$ is the part with $i>j+1$, and $f \reso g$ is the ``diagonal'' part, where $|i-j | \leqslant 1$. More precisely, we define
\[
   f \lpara g= g\rpara f = \sum_{j \geqslant -1} \sum_{i=-1}^{j-2} \Delta_{i} f \Delta_{j} g \qquad \text{and}\qquad
   f \reso g= \sum_{| i-j | \leqslant 1} \Delta_{i} f \Delta_{j} g.
\]
We also introduce the notation
\[
   f \rparacurl g=f \rpara g+f \reso g.
\]
This decomposition  behaves nicely with respect to Littlewood--Paley theory. Of course, it depends on the dyadic partition of unity used to define the blocks $\Delta_{j}$, and also on the particular choice of the pairs $( i,j )$ in the diagonal part. Our choice of taking all $(i,j)$ with $| i-j| \leqslant 1$ into the diagonal part corresponds to property \ref{def:dyadic partition3}. in the definition of dyadic partition of unity in \ref{sec:para},  where we assumed that $\tmop{supp} ( \rho ( 2^{-i} \cdummy ) ) \cap \tmop{supp} ( \rho ( 2^{-j} \cdummy ) ) = \emptyset$ for $| i-j | >1$. This means that every term in the series
\[
   f \lpara g = \sum_{j \geqslant -1} \sum_{i=-1}^{j-2} \Delta_{i} f \Delta_{j} g= \sum_{j \geqslant -1} S_{j-1} f \Delta_{j} g
\]
has a Fourier transform which is supported in a suitable annulus, and of course the same holds true for $f \rpara g$. On the other side, every term in the diagonal part $f \reso g$ has a Fourier transform that is supported in a ball. We call $f \lpara g$ and $f \rpara g$ \tmtextit{paraproducts}, and $f \reso g$ the \tmtextit{resonant} term.

Bony's crucial observation is that $f \lpara g$ (and thus $f \rpara g$) is always a well-defined distribution. In particular, if $\alpha >0$ and $\beta \in \mathbb{R}$, then $( f,g ) \mapsto f \lpara g$ is a bounded bilinear operator from $\CC^{\alpha} \times \CC^{\beta}$ to $\CC^{\beta}$. Heuristically, $f \lpara g$ behaves at large frequencies like $g$ (and thus retains the same regularity), and $f$ provides only a modulation of $g$ at larger scales. The only difficulty in defining $fg$ for arbitrary distributions lies in handling the diagonal term $f \reso g$. The basic result about these bilinear operations is given by the following estimates.

\begin{lemma}[Paraproduct estimates, \cite{Bony1981}]\label{thm:paraproduct}
   For any $\beta \in \mathbb{R}$ we have
   \begin{equation} \label{eq:para-1}
       \| f \lpara g \|_{\beta} \lesssim_{\beta} \| f \|_{L^{\infty}} \| g \|_{\beta} ,
   \end{equation}
   and for $\alpha <0$ furthermore
   \begin{equation}\label{eq:para-2}
      \| f \lpara g \|_{\alpha + \beta} \lesssim_{\alpha , \beta} \| f \|_{\alpha} \| g \|_{\beta}.
   \end{equation}
   For $\alpha + \beta >0$ we have
   \begin{equation}\label{eq:para-3}
      \| f \reso g \|_{\alpha + \beta} \lesssim_{\alpha , \beta} \| f \|_{\alpha} \| g \|_{\beta}.
    \end{equation}
\end{lemma}

\begin{proof}
  Observe that there exists an annulus $\CA$ such that $S_{j-1} f \Delta_{j} g$ has Fourier transform supported in $2^{j} \CA$, and that for $f \in L^{\infty}$ we have
  \[
     \| S_{j-1} f \Delta_{j} g \|_{L^{\infty}} \leqslant \| S_{j-1} f  \|_{L^{\infty}} \| \Delta_{j} g \|_{L^{\infty}} \leqslant \| f \|_{L^{\infty}} 2^{-j \beta} \| g \|_{\beta}.
  \]
  On the other side, if $\alpha <0$ and $f \in \CC^{\alpha}$, then
  \begin{align*}
     \| S_{j-1} f \Delta_{j} g \|_{L^{\infty}} & \leqslant \sum_{i \leqslant j-2} \| \Delta_{i} f \|_{L^{\infty}} \| \Delta_{j} g \|_{L^{\infty}} \lesssim \| f \|_{\alpha} \| g \|_{\beta} \sum_{i \leqslant j-2} 2^{-i\alpha -j \beta} \\
     & \lesssim \| f \|_{\alpha} \| g \|_{\beta} 2^{-j ( \alpha + \beta )} .
  \end{align*}
  By Lemma~\ref{lem: Besov regularity of series}, we thus obtain~{\eqref{eq:para-1}} and~{\eqref{eq:para-2}}. To estimate $f \reso g$, observe that the term $u_{j} = \Delta_{j} f \sum_{i: | i-j | \leqslant 1} \Delta_{i} g$ has Fourier transform supported in a ball $2^{j} \CB$, and that
  \[
     \| u_{j} \|_{L^{\infty}} \lesssim \| \Delta_{j} f \|_{L^{\infty}} \sum_{i: | i-j | \leqslant 1} \| \Delta_{i} g \|_{L^{\infty}} \lesssim \| f \|_{\alpha} \| g \|_{\beta} 2^{- ( \alpha + \beta ) j}.
  \]
  So if $\alpha + \beta >0$, then we can apply the second part of Lemma~\ref{lem: Besov regularity of series} to obtain that $f \reso g= \sum_{j \geqslant -1} u_{j}$ is an element of $\CC^{\alpha + \beta}$ and that equation~{\eqref{eq:para-3}} holds.
\end{proof}

A natural corollary is that the product $fg$ of two elements $f \in \CC^{\alpha}$ and $g \in \CC^{\beta}$ is well defined as soon as $\alpha + \beta >0$, and that it belongs to $\CC^{\gamma}$, where $\gamma = \min \{\alpha , \beta\}$.

\subsection{Paracontrolled distributions and RDEs}\label{sec:paracontrolled ode}

Consider the \textsc{rde}
\begin{equation}\label{eq:rde-intro} 
  \partial_{t} u=F ( u ) \xi, \qquad u(0) = u_0,
\end{equation}
where $u_0 \in \R^d$, $u\colon \mathbb{R} \rightarrow \mathbb{R}^{d}$ is a continuous vector
valued function, $\partial_{t}$ is the time derivative, $\xi\colon \mathbb{R} \rightarrow \mathbb{R}^{n}$ is a vector valued distribution with values in
$\CC^{\alpha -1}$ for some $\alpha \in ( 1/3,1 )$, and
$F\colon \mathbb{R}^{d} \rightarrow \Lcal ( \mathbb{R}^{n} ,\mathbb{R}^{d})$ is a family of vector fields on $\mathbb{R}^{d}$. A natural approach is to
understand this equation as limit of the classical ODEs
\begin{equation}\label{eq:rde-reg-intro}
  \partial_{t} u^{\varepsilon} = F ( u^{\varepsilon} ) \xi^{\varepsilon}, \qquad u^\varepsilon(0) = u_0,
\end{equation}
for a family of smooth approximations $(\xi^{\varepsilon})$ of $\xi$ such that
$\xi^{\varepsilon} \rightarrow \xi$ in $\CC^{\alpha -1}$ as $\varepsilon \rightarrow 0$. In order to pass to the limit, we are looking for a
priori estimates on $u^{\varepsilon}$ which require only a control on the
$\CC^{\alpha -1}$ norm of $\xi$.

To avoid cumbersome notation, we will work at the level of
equation~{\eqref{eq:rde-intro}} for smooth $\xi$, where it should be understood that
our aim is to obtain a priori estimates for the solution, in order to safely pass
to the limit and extend the solution map to a larger class of data. The
natural regularity of $u$ is $\CC^{\alpha}$, since $u$ should gain one
derivative with respect to $F ( u ) \xi$, which will not behave better than
$\xi$, and will therefore be in $\CC^{\alpha -1}$.

We use the paraproduct decomposition to write the right hand side of {\eqref{eq:rde-intro}} as a
sum of the three terms
\begin{equation}\label{eq:para-rhs-intro}
  \underbrace{F ( u ) \lpara \xi}_{\alpha -1} + \underbrace{F ( u ) \reso  \xi}_{2 \alpha -1} + \underbrace{F ( u ) \rpara \xi}_{2 \alpha -1}
\end{equation}
(where the quantity indicated by the underbrace corresponds to the expected regularity of each term). Note however that unless $2 \alpha -1>0$, the resonant term
$F ( u ) \reso   \xi$ cannot be controlled using only the $\CC^{\alpha}$--norm
of $u$ and the $\CC^{\alpha -1}$--norm of $\xi$. If $F$ is at least in $C^2$, we can use a paralinearization result (see Lemma~\ref{lemma:para-taylor} below) to rewrite this term as
\begin{equation}\label{eq:rde expand Fu}
   F ( u ) \reso \xi =F' ( u )  ( u \reso \xi ) + \Pi_{F} ( u, \xi ),
\end{equation}
where the remainder $\Pi_{F} ( u, \xi )$ is well defined under the condition
$3 \alpha -1>0$, provided that $u \in \CC^{\alpha}$ and $\xi \in \CC^{\alpha -1}$. In this case it belongs to $\CC^{3 \alpha -1}$. The
difficulty is now localized in the linearized resonant product $u \reso \xi$. In order to control this term, we would like to exploit the fact that
the function $u$ is not a generic element of $\CC^{\alpha}$ but that it has a
specific structure, since its derivative $\partial_{t} u$ has to match the
paraproduct decomposition given in~\eqref{eq:para-rhs-intro}. Thus, we postulate
that the solution $u$ is given by the following \tmtextit{paracontrolled ansatz:}
\begin{equation}\label{eq:paracontrolled ansatz}
   u=u^{\vartheta} \lpara \vartheta +u^{\sharp},
\end{equation}
where $u^{\vartheta} , \vartheta \in \CC^{\alpha}$ and the remainder
$u^{\sharp}$ is in $\CC^{2 \alpha}$. This decomposition allows for a finer analysis
of the resonant term $u \reso \xi$. Indeed, we have
\begin{equation}\label{eq:rde expand resonant term}
   u \reso   \xi = ( u^{\vartheta} \lpara \vartheta ) \reso   \xi +u^{\sharp} \reso   \xi =u^{\vartheta} ( \vartheta \reso \xi ) +C ( u^{\vartheta} , \vartheta , \xi ) +u^{\sharp} \reso   \xi ,
\end{equation}
where the commutator is defined by $C ( u^{\vartheta} ,\vartheta , \xi ) = ( u^{\vartheta} \lpara \vartheta ) \reso \xi -u^{\vartheta}( \vartheta \reso \xi )$. 
Observe now that the term $u^{\sharp} \reso   \xi$
does not pose any further problem, as it is bounded in $\CC^{3 \alpha -1}$.
Moreover, we will show that the commutator is a bounded multilinear function of its arguments as
long as the sum of their regularities is strictly positive, see
Lemma~\ref{lemma:commutator} below. By assumption, we have $3 \alpha -1>0$,
and therefore $C ( u^{\vartheta} , \vartheta , \xi ) \in \CC^{3 \alpha -1}$.
The only problematic term which remains to be handled is thus $\vartheta \reso   \xi$. Here we need to make the assumption that $\vartheta \reso \xi \in \CC^{2 \alpha -1}$ in order for the product $u^{\vartheta} (\vartheta \reso   \xi )$ to be well defined. That assumption is not guaranteed
by the analytical estimates at hand, and it has to be added as a
further requirement. Granting this, we have obtained that the right hand side
of equation~{\eqref{eq:rde-intro}} is well defined and a continuous function of $(u,u^{\vartheta} ,u^{\sharp} , \vartheta , \xi , \vartheta \reso \xi )$.

The paracontrolled ansatz and the Leibniz rule for the paraproduct now imply
that~{\eqref{eq:rde-intro}} can be rewritten as
\[
   \partial_{t} u= \partial_{t} ( u^{\vartheta} \lpara \vartheta +u^{\sharp} ) = \partial_{t} u^{\vartheta} \lpara \vartheta +u^{\vartheta} \lpara \partial_{t} \vartheta + \partial_{t} u^{\sharp} =F ( u ) \lpara \xi +F ( u ) \reso   \xi +F ( u ) \rpara \xi.
\]
If we choose $\vartheta$ such that $\partial_{t} \vartheta = \xi$ 
and we set $u^{\vartheta} =F ( u )$, then we can
use~{\eqref{eq:rde expand Fu}} and~{\eqref{eq:rde expand resonant term}} to
obtain the following equation for the remainder $u^{\sharp}$:
\begin{align*}
  \partial_{t} u^{\sharp} & = F' ( u ) F ( u ) (\vartheta \reso \xi ) +F ( u ) \rpara \xi - ( \partial_{t} F ( u )) \lpara \vartheta \\
  &\quad +F' ( u ) C ( F ( u ) , \vartheta , \xi ) +F' ( u )  ( u^{\sharp} \reso \xi) + \Pi_{F} ( u, \xi ).
\end{align*}
Together with the equation $u=F ( u ) \lpara \vartheta +u^{\sharp}$, this completely describes the solution and allows us to obtain an a priori estimate on $u$ in terms of 
\[
   (u_0, \|\xi\|_{\alpha-1}, \|\vartheta \reso \xi\|_{2\alpha-1}).
\]
With this estimate at hand, it is now relatively straightforward to show that if $F \in C^3_b$, then $u$ depends continuously on the data $(u_0, \xi, \vartheta \reso \xi)$, so that we can pass to the limit in~\eqref{eq:rde-reg-intro} and make sense of the solution to~\eqref{eq:rde-intro} also for irregular $\xi \in \CC^{\alpha-1}$ as long as $\alpha > 1/3$.

\subsection{Commutator estimates and paralinearization}

In this section we prove some lemmas which will allow us to perform algebraic computations with the paraproduct and the resonant term, and thus to justify the analysis of the previous section.

\begin{lemma}[see also Lemma 2.97 of {\cite{Bahouri2011}}]\label{l:bahouri commutator}
  Let $f \in \CC^{\alpha}$ for $\alpha \in ( 0,1 )$, and let $g \in L^{\infty}$. For any $j \geqslant -1$ we have
  \[
     \| [ \Delta_{j} ,f ] g \|_{L^{\infty}} = \| \Delta_{j} ( f g ) -f \Delta_{j} g \|_{L^{\infty}} \lesssim 2^{- \alpha j} \|  f \|_{\alpha} \| g \|_{L^{\infty}} .
  \]
\end{lemma}

This commutator lemma is easily proven by writing $\Delta_{j} = \rho_{j} (\mathD )$ as a convolution operator, and by using the embedding of $\CC^{\alpha}$ in the space of H{\"o}lder continuous functions.

\begin{lemma}\label{lemma:paraprojection}
  Assume that $\alpha \in ( 0,1 )$ and $\beta \in \mathbb{R}$, and let $f \in \CC^{\alpha}$ and $g \in \CC^{\beta}$. Then
  \[
     \Delta_{j} ( f \lpara g ) =f \Delta_{j} g+R_{j} ( f,g ) ,
  \]
  for all $j \geqslant -1$, with a remainder $\| R_{j} ( f,g ) \|_{L^{\infty}} \lesssim 2^{-j ( \alpha + \beta )} \| f \|_{\alpha} \| g  \|_{\beta}$.
\end{lemma}

\begin{proof}
  Note that $f \lpara g= \sum_{i} f \lpara \Delta_{i} g$, and that there exists an annulus $\CA$ such that for all $i$ the Fourier  transform of $f \lpara \Delta_{i} g$ is supported in $2^{i} \CA$. Hence, we have $\Delta_{j} ( f \lpara \Delta_{i} g ) \neq 0$ only if $j \sim i$, which leads to
  \begin{align*}
     \Delta_{j} ( f \lpara g ) &= \sum_{i:i \sim j} \Delta_{j} ( f \lpara \Delta_{i} g ) = \sum_{i:i \sim j} \Delta_{j} ( f \Delta_{i} g ) -  \sum_{i:i \sim j} \Delta_{j} ( f \rparacurl \Delta_{i} g ) \\
     & = \sum_{i:i \sim j} f \Delta_{j} \Delta_{i} g- \sum_{i:i \sim j} [\Delta_{j} ,f ] \Delta_{i} g- \sum_{i:i \sim j} \Delta_{j} ( f \rparacurl \Delta_{i} g ),
  \end{align*}
  where we recall that $[ \Delta_{j} ,f ] \Delta_{i} g= \Delta_{j} ( f\Delta_{i} g ) -f \Delta_{j} \Delta_{i} g$ denotes the commutator. The sum over $i$ with $i \sim j$ can be chosen to encompass enough terms so that $\Delta_{j} g= \sum_{i:i \sim j} \Delta_{j} \Delta_{i} g$, and therefore we conclude that
  \[
     \| \Delta_{j} ( f \lpara g ) -f \Delta_{j} g \|_{L^{\infty}} \leqslant  \sum_{i:i \sim j} \| [ \Delta_{j} ,f ] \Delta_{i} g \|_{L^{\infty}} -  \sum_{i:i \sim j} \| \Delta_{j} ( f \rparacurl \Delta_{i} g )
     \|_{L^{\infty}}.
  \]
  We apply Lemma~\ref{l:bahouri commutator} to each term of the first sum, and the paraproduct estimates to each term of the second sum, to obtain
  \[
     \| \Delta_{j} ( f \lpara g ) -f \Delta_{j} g \|_{L^{\infty}} \lesssim  2^{-j ( \alpha + \beta )} \| f \|_{\alpha} \| g \|_{\beta} .
  \]
\end{proof}

Using this result, it is easy to prove our basic commutator lemma.

\begin{lemma}\label{lemma:commutator}
   Assume that $\alpha \in ( 0,1 )$ and $\beta , \gamma \in \mathbb{R}$ are such that $\alpha + \beta + \gamma >0$ and $\beta + \gamma <0$. Then there exists a bounded trilinear operator $C \in \Lcal^{3} \left( \CC^{\alpha}\! \times\! \CC^{\beta}\! \times\! \CC^{\alpha}, \CC^{\alpha + \beta + \gamma} \right)$ such that
  \[
     C ( f,g,h ) = ( ( f \lpara g ) \reso h ) -f ( g \reso h )
  \]
  whenever $f,g,h \in \CS$.
\end{lemma}

\begin{proof}
  Let $f,g,h \in \CS$ and write
  \[
     C ( f,g,h ) = ( ( f \lpara g ) \reso h ) -f ( g \reso h ) = \sum_{j,k  \geqslant -1} \sum_{i: | i-j | \leqslant 1} [ \Delta_{i} ( \Delta_{k} f  \lpara g ) \Delta_{j} h- \Delta_{k} f \Delta_{i} g \Delta_{j} h ].
  \]
  Observe that for fixed $k$, the term $\Delta_{k} f \lpara g$ has a Fourier transform supported outside of a ball $2^{k} \CB$. Thus, we have $\Delta_{i}( \Delta_{k} f \lpara g ) = \mathbf{1}_{i \gtrsim k} \Delta_{i} ( \Delta_{k}f \lpara g )$, and therefore we can apply Lemma~\ref{lemma:paraprojection} to obtain
  \begin{align}\label{eq:commutator pr1}\nonumber
     C ( f,g,h ) &= \sum_{j,k \geqslant -1} \sum_{i: | i-j | \leqslant 1} [\mathbf{1}_{i \gtrsim k} ( \Delta_{k} f \Delta_{i} g+R_{i} ( \Delta_{k} f,g ) ) \Delta_{j} h- \Delta_{k} f \Delta_{i} g \Delta_{j} h] \\
    &= \sum_{j,k \geqslant -1} \sum_{i: | i-j | \leqslant 1} [\mathbf{1}_{i \gtrsim k} R_{i} ( \Delta_{k} f,g ) \Delta_{j} h -  \mathbf{1}_{i \leqslant k-N} \Delta_{k} f \Delta_{i} g \Delta_{j} h]
  \end{align}
  for some fixed $N \in \mathbb{N}$. We treat the two sums separately. First
  observe that for fixed $k$, the term $\sum_{j \geqslant -1} \sum_{i: | i-j | \leqslant 1} \mathbf{1}_{i \leqslant k-N} \Delta_{k} f \Delta_{i} g \Delta_{j} h$ has a
  Fourier transform which is supported in a ball $2^{k} \CB$.
  Moreover,
  \begin{align*}
     \Bigg\| \sum_{j \geqslant -1} \sum_{i: | i-j | \leqslant 1} \mathbf{1}_{i \leqslant k-N} \Delta_{k} f \Delta_{i} g \Delta_{j} h \Bigg\|_{L^{\infty}} & \lesssim 2^{-k \alpha} \| f \|_{\alpha} \sum_{i=-1}^{k-N} 2^{-i ( \beta + \gamma)} \| g \|_{\beta} \| h \|_{\gamma} \\
     & \simeq 2^{-k ( \alpha + \beta + \gamma )} \| f \|_{\alpha} \| g \|_{\beta} \| h \|_{\gamma},
  \end{align*}
  where in the second step we used that $\beta + \gamma <0$. Since $\alpha +  \beta + \gamma >0$, the estimate for the second series
  in~{\eqref{eq:commutator pr1}} follows from Lemma~\ref{lem: Besov regularity of series}.
  
  For the first series, recall that $R_{i} ( \Delta_{k} f,g ) = \Delta_{i} ( \Delta_{k} f \lpara g ) - \Delta_{k} f \Delta_{i} g$. So for fixed $j$, the
  Fourier transform of $\sum_{k \geqslant -1} \sum_{i: | i-j | \leqslant 1} \mathbf{1}_{i \gtrsim k} R_{i} ( \Delta_{k} f,g ) \Delta_{j} h$ is
  supported in ball $2^{j} \CB$. Furthermore,
  Lemma~\ref{lemma:paraprojection} yields
  \begin{align*}
     \Bigg\| \sum_{k \geqslant -1} \sum_{i: | i-j | \leqslant 1} \mathbf{1}_{i \gtrsim k} R_{i} ( \Delta_{k} f,g ) \Delta_{j} h \Bigg\|_{L^{\infty}} & = \Bigg\| \sum_{i: | i-j | \leqslant 1} R_{i}  \Bigg( \sum_{k \lesssim i} \Delta_{k} f,g \Bigg) \Delta_{j} h  \Bigg\|_{L^{\infty}} \\
     & \lesssim \sum_{i: | i-j | \leqslant 1} 2^{-i ( \alpha + \beta )} \bigg\| \sum_{k \lesssim i} \Delta_{k} f \bigg\|_{\alpha} \| g \|_{\beta} 2^{-j \gamma} \| h \|_{\gamma} \\
     &\lesssim 2^{-j ( \alpha + \beta + \gamma )} \| f \|_{\alpha} \| g \|_{\beta} \| h \|_{\gamma},
  \end{align*}
  so that the claimed bound for $\| C ( f,g,h ) \|_{\alpha + \beta + \gamma}$
  follows from another application of Lemma~\ref{lem: Besov regularity of series}.
  
  Now we can extend $C$ to a bounded trilinear operator on the
  closure of the smooth functions in $\CC^{\alpha} \times \CC^{\beta} \times \CC^{\gamma}$. Unfortunately, this is a strict subset of
  $\CC^{\alpha} \times \CC^{\beta} \times \CC^{\gamma}$. But we obtain similar bounds for $C$ acting on $\CC^{\alpha'} \times \CC^{\beta'} \times \CC^{\gamma'}$ for $\alpha' \in ( 0,1 )$ and
  $\beta' , \gamma' \in \mathbb{R}$, such that $\alpha' < \alpha$, $\beta' < \beta$, $\gamma' < \gamma$, and $\alpha' + \beta'+ \gamma' >0$. Since
  $\CC^{\alpha} \times \CC^{\beta} \times \CC^{\gamma}$ is contained in the
  closure of the smooth functions in $\CC^{\alpha'} \times \CC^{\beta'} \times \CC^{\gamma'}$, we obtain the required extension of $C$ to $\CC^{\alpha} \times \CC^{\beta} \times \CC^{\gamma}$. Moreover, this argument also shows that for $(f,g,h) \in \CC^{\alpha} \times \CC^{\beta} \times \CC^{\gamma}$ we have
  \begin{align*}
    \| C(f,g,h)\|_{\alpha+\beta+\gamma} &= \limsup_{\alpha'\uparrow \alpha, \beta' \uparrow \beta, \gamma'\uparrow\gamma} \| C(f,g,h)\|_{\alpha'+\beta'+\gamma'} \lesssim \limsup_{\alpha'\uparrow \alpha, \beta' \uparrow \beta, \gamma'\uparrow\gamma} \|f\|_{\alpha'} \|g \|_{\beta'} \|h\|_{\gamma'}\\
    &= \|f\|_{\alpha} \|g \|_{\beta} \|h\|_{\gamma}.
  \end{align*}
  Alternatively, this last bound also follows from the Fatou property of Besov spaces, see Theorem~2.72 of \cite{Bahouri2011}.
\end{proof}

\begin{remark}
  The restriction $\beta + \gamma <0$ is not problematic. If $\beta + \gamma >0$, then $( f \lpara g ) \reso h$ can be treated with the usual paraproduct estimates, without the need of introducing the commutator. If $\beta + \gamma =0$, then we can apply the commutator estimate with $\gamma' < \gamma$
  sufficiently close to $\gamma$ such that $\alpha + \beta + \gamma' >0$.
  
  The restriction $\alpha < 1$ can be lifted, see~\cite{Gubinelli2015EBP}, but the price to pay is that then the commutator can only be controlled in $\CC^{\beta+\gamma}$ and not in $\CC^{\alpha+\beta+\gamma}$. Passing the threshold $\alpha = 1$ seems to be one of the key challenges in extending the paracontrolled approach to problems where one has to gain a lot of regularity, such as the three-dimensional version of \textsc{pam}, where the noise $\xi$ is in $\CC^{-3/2-\varepsilon}$, the solution $u$ is in $\CC^{1/2-\varepsilon}$, and thus the sum of the regularities of the factors $F(u)$ and $\xi$ is smaller than $-1$.
\end{remark}

Our next result is a simple paralinearization lemma for non-linear operators.

\begin{lemma}[see also~\cite{Bahouri2011}, Theorem 2.92]\label{lemma:paralinearization}
  Let $\alpha \in (0,1)$, $\beta \in (0,\alpha]$, and let $F \in C^{1+\beta/\alpha}_b$. There exists a locally bounded map $R_{F} \colon  \CC^{\alpha} \rightarrow \CC^{\alpha+\beta}$ such that
  \begin{equation}\label{eq:para-linearization}
    F ( f ) =F' ( f ) \lpara f+R_{F} ( f )
  \end{equation}
  for all $f \in \CC^{\alpha}$. More precisely, we have
  \[
     \| R_{F} ( f ) \|_{\alpha + \beta} \lesssim \|F\|_{C^{1+\beta/\alpha}_b} (1 + \| f \|_{\alpha}^{1+\beta/\alpha}).
  \]
  If $F \in C_b^{2+\beta/\alpha}$, then $R_{F}$ is locally Lipschitz continuous:
  \[
     \| R_{F} ( f ) -R_{F} ( g ) \|_{\alpha + \beta} \lesssim \|F\|_{C^{2+\beta/\alpha}_b} ( 1+ \| f \|_{\alpha} + \| g \|_{\alpha})^{1+\beta/\alpha} \| f - g \|_\alpha.
  \]
\end{lemma}

\begin{proof}
   The difference $F(f) - F'(f)\lpara f$ is given by
   \[
      R_F(f) = F(f) - F'(f)\lpara f = \sum_{i \geqslant -1} [\Delta_i F(f) - S_{i-1} F'(f) \Delta_i f] = \sum_{i \geqslant -1} u_i,
   \]
   and every $u_i$ is spectrally supported in a ball $2^i \CB$. For $i<1$, we simply estimate $\| u_i\|_{L^\infty} \lesssim \|F\|_{C^1_b}(1+\|f\|_\alpha)$. For $i \geqslant 1$ we use the fact that $f$ is a bounded function to write the Littlewood-Paley projections as convolutions and obtain
  \begin{align*}
     &u_i ( x ) \\
     &\hspace{15pt} = \int K_{i} ( x-y ) K_{<i-1} ( x-z ) [F ( f ( y) ) - F'(f(z)) f(y)] \mathd y \mathd z \\
     &\hspace{15pt} = \int K_{i} ( x-y ) K_{<i-1} ( x-z ) [ F ( f ( y ) ) -F ( f ( z ) ) - F'(f(z)) (f(y) - f(z)) ] \mathd y \mathd z,
  \end{align*}
  where $K_{i} = \CF^{-1} \rho_{i}$, $K_{<i-1} = \sum_{j<i-1} K_{j}$,
  and where we used that $\int K_{i} ( y ) \mathd y= \rho_{i} ( 0 ) =0$ for $i \geqslant 0$ and $\int K_{<i-1} ( z ) \mathd z=1$ for $i \geqslant 1$. Now we can apply a first order Taylor expansion to $F$ and use the $\beta/\alpha$--H\"older continuity of $F'$ in combination with the $\alpha$--H\"older continuity of $f$, to deduce
  \begin{align*}
     |u_i(x)| & \lesssim \|F\|_{C^{1+\beta/\alpha}_b} \|f\|_\alpha^{1+\beta/\alpha} \int |K_{i} ( x-y ) K_{<i-1} ( x-z )| \times |z-y|^{\alpha+\beta} \mathd y \mathd z \\
     & \lesssim \|F\|_{C^{1+\beta/\alpha}_b} \|f\|_\alpha^{1+\beta/\alpha} 2^{-i(\alpha+\beta)}.
  \end{align*}
  Therefore, the estimate for $R_F(f)$ follows from Lemma~\ref{lem: Besov regularity of series}. The estimate for $R_F(f) - R_F(g)$ is shown in the same way.
%
\end{proof}

Let $g$ be a distribution belonging to $\CC^{\gamma}$ for some $\gamma <0$. Then
the map $f \mapsto f \reso g$ behaves, modulo smoother correction terms, like
a derivative operator:

\begin{lemma}\label{lemma:para-taylor}
  Let $\alpha \in ( 0,1 )$, $\beta \in (0,\alpha]$, $\gamma < 0$ be such that $\alpha+\beta+\gamma>0$ but $\alpha+\gamma<0$. Let $F \in C^{1+\beta/\alpha}_b$.
  Then there exists a locally bounded map $\Pi_{F} \colon \CC^{\alpha} \times \CC^{\gamma} \rightarrow \CC^{\alpha + \beta + \gamma}$ such that
  \begin{equation}\label{eq:para-taylor}
      F ( f ) \reso g=F' ( f ) ( f \reso g ) + \Pi_{F} ( f,g )
  \end{equation}
  for all $f \in \CC^{\alpha}$ and all smooth $g$. More precisely, we have
  \[
     \| \Pi_{F} ( f,g ) \|_{\alpha + \beta + \gamma} \lesssim \|F\|_{C^{1+\beta/\alpha}_b} (1+ \| f \|_{\alpha}^{1+\beta/\alpha} ) \| g \|_{\gamma}.
  \]
  If $F \in C^{2+\beta/\alpha}_b$, then $\Pi_{F}$ is locally Lipschitz continuous:
  \begin{align*}
     &\| \Pi_{F} ( f,g ) - \Pi_{F} ( u,v ) \|_{\alpha + \beta + \gamma} \\
     &\hspace{40pt} \lesssim \|F\|_{C^{2+\beta/\alpha}_b} (1 + (\|f\|_\alpha + \|u\|_\alpha)^{1+\beta/\alpha} + \|v\|_\gamma)  ( \| f-u \|_{\alpha} + \| g-v \|_{\gamma}).
  \end{align*}
\end{lemma}

\begin{proof}
  Just use the paralinearization and commutator lemmas above to deduce that
  \begin{align*}
     \Pi(f,g) & = F(f) \reso g - F'(f) (f \reso g) = R_F(f) \reso g + (F'(f)\lpara f) \reso g - F'(f) (f \reso g) \\
     & = R_F(f) \reso g + C(F'(f),f,g),
  \end{align*}
  so that the claimed bounds easily follow from Lemma~\ref{lemma:commutator} and Lemma~\ref{lemma:paralinearization}.
%
\end{proof}

Besides this sort of chain rule, we also have a Leibniz rule for $f\mapsto f\reso g$:

\begin{lemma}\label{lem:para-taylor product}
  Let $\alpha \in ( 0,1 )$ and $\gamma <0$ be such that $2 \alpha + \gamma >0$ but $\alpha+\gamma<0$
  Then there exists a bounded trilinear operator
  $\Pi_{\times} \colon \CC^{\alpha} \times \CC^{\alpha} \times \CC^{\gamma} \rightarrow \CC^{2 \alpha + \gamma}$, such that
  \[
     ( f u ) \reso g=f (u \reso g ) +u ( f \reso g ) + \Pi_{\times} ( f,u,g )
  \]
  for all $f,u \in \CC^{\alpha} ( \mathbb{R} )$ and all smooth $g$.
\end{lemma}

\begin{proof}
  It suffices to note that $f u = f\lpara u + f \rpara u + f \reso u$, which leads to
  \[
     \Pi_{\times} ( f,u,g ) = ( f u ) \reso g - f (u \reso g ) +u ( f \reso g ) = C(f,u,g) + C(u,f,g) + (f\reso u)\reso g,
  \]
  so that the result follows from Lemma~\ref{lemma:commutator}.
\end{proof}

\section{Rough differential equations}\label{sec:ODE}

Let us now resume the analysis of Section~\ref{sec:paracontrolled ode}. We want to study the \textsc{rde}
\begin{equation}\label{eq:rde} 
  \partial_{t} u=F ( u ) \xi, \qquad u(0) = u_0,
\end{equation}
where $u_0 \in \R^d$, $u\colon \mathbb{R} \rightarrow \mathbb{R}^{d}$ is a continuous vector
valued function, $\xi\colon \mathbb{R} \rightarrow \mathbb{R}^{n}$ is a vector valued distribution with values in
$\CC^{\alpha -1}$ for some $\alpha \in ( 1/3,1 )$, and
$F\colon \mathbb{R}^{d} \rightarrow \Lcal ( \mathbb{R}^{n} ,\mathbb{R}^{d})$ is a family of vector fields on $\mathbb{R}^{d}$.

In order to obtain concrete estimates,
we have to localize the equation. Therefore, we introduce a smooth cut-off
function $\varphi$ with support on $[ -2,2 ]$, which is equal to 1 on $[ -1,1]$ 
and modify the equation as
\[
   \partial_{t} u= \varphi F ( u ) \xi, \qquad u(0) = u_0.
\]
In the regular setting, if $u$ is solution to this equation, it is also a
solution of the original equation on $[ -1,1 ]$, and thus it is sufficient to
study the last equation for local bounds. To avoid problems with the fact
that the paraproduct is a (mildly) non-local operation, we modify the
paracontrolled ansatz as follows:
\begin{equation}\label{eq:ode paracontrolled ansatz localized}
   u= \varphi ( F ( u )\lpara \vartheta ) +u^{\sharp}.
\end{equation}
If $F \in C^2_b$, an easy computation gives
\begin{align*}
   \partial_{t} u^{\sharp} & = \varphi F ( u ) \xi - ( \partial_{t} \varphi )  ( F ( u ) \lpara \vartheta ) - \varphi  ( (\partial_{t} F( u )) \lpara \vartheta ) - \varphi ( F ( u ) \lpara \xi )\\
   & = \varphi \left[ ( F ( u ) \rpara \xi ) +F' ( u )  ( ( u-u_{0} ) \reso \xi ) + \Pi_{F_{u_{0}}} ( u-u_{0} , \xi ) - ( \partial_{t} F ( u ) \lpara \vartheta ) \right] \\
   &\quad - ( \partial_{t} \varphi )  ( F ( u ) \lpara \vartheta ) ,
\end{align*}
where we set $F_{u_{0}} ( x ) =F ( u_{0} +x )$ and used that $( F_{u_{0}} )' (x-u_{0} ) =F' ( x )$ for all $x \in \R^d$. We subtract the contribution of the initial condition,
because this will eventually allow us to solve the equation on a small
interval whose length does not depend on $u_0$. If we plug in
the paracontrolled ansatz for $u$, then $F' ( u ) ( ( u-u_{0} ) \reso \xi )$ becomes
\[
   F' ( u )  ( ( u-u_{0} ) \reso \xi ) =F' ( u ) ( ( \varphi ( F ( u ) \lpara \vartheta ) ) \reso \xi )  +F' ( u ) ( ( u^{\sharp} -u_{0} ) \reso \xi ).
\]
For the first term on the right hand side we can further use that
\[
   ( \varphi ( F ( u ) \lpara \vartheta ) ) \reso \xi = \varphi ( ( F (u ) \lpara \vartheta ) \reso \xi ) + ( F ( u ) \lpara \vartheta ) (\varphi \reso \xi ) + \Pi_{\times} ( \varphi ,F ( u ) \lpara   \vartheta , \xi ),
\]
where we recall that $\Pi_\times$ was defined in Lemma~\ref{lem:para-taylor product}. Introducing the commutator in order to take care of the resonant product $(F(u) \lpara \vartheta) \reso \xi$, we get
\begin{align}\label{eq:ode phi sharp modified def} \nonumber
   \partial_{t} u^{\sharp} & = \varphi \bigg[ ( F ( u ) \rpara \xi ) + \Pi_{F_{u_{0}}} ( u-u_{0} , \xi ) +F' ( u ) ( ( u^{\sharp} -u_{0} ) \reso \xi ) + ( F ( u ) \lpara \vartheta ) ( \varphi \reso \xi ) \\ \nonumber
   &\hspace{9pt} + \Pi_{\times} ( \varphi ,F ( u ) \lpara \vartheta , \xi ) + \varphi C ( F ( u ) , \vartheta , \xi ) +F' ( u ) \varphi F ( u ) ( \vartheta \reso \xi ) - ( \partial_{t} F ( u ) \lpara \vartheta ) \bigg] \\ \nonumber
   &\quad - ( \partial_{t} \varphi )  ( F ( u ) \lpara \vartheta ) \\
   & = \varphi \Phi^{\sharp} - (\partial_{t} \varphi )  ( F ( u ) \lpara \vartheta ),
\end{align}
where $\Phi^\sharp$ is defined to be the term in the large square brackets. Let us summarize our observations so far.

\begin{lemma}
  Let $\xi$ be a smooth path, let $\vartheta$ be such that $\partial_{t}  \vartheta = \xi$, and let $F \in C^{2}_{b}$. Then
  $u$ solves the ODE
  \[
     \partial_{t} u= \varphi F ( u ) \xi, \qquad u ( 0 ) =u_{0},
  \]
  if and only if $u= \varphi ( F (u ) \lpara \vartheta ) +u^{\sharp}$, where $u^{\sharp}$ solves
  \[
     \partial_{t} u^{\sharp} = \varphi \Phi^{\sharp} - ( \partial_{t} \varphi )  ( F ( u ) \lpara \vartheta ), \qquad u^{\sharp} ( 0 ) =u_{0} - \varphi ( F ( u )\lpara \vartheta ) ( 0 ),
  \]
  and where $\Phi^{\sharp}$ is defined in~\eqref{eq:ode phi sharp modified def}. Moreover, for $\alpha \in ( 1/3,1/2 )$ we have the estimate
  \[
     \| \Phi^{\sharp} \|_{2 \alpha -1} \lesssim C_{F} C_{\xi} ( 1+ \| u-u_{0}  \|_{\alpha} + \| u-u_{0} \|^{2}_{\alpha} + \| u^{\sharp} -u_{0} \|_{2 \alpha} ),
  \]
  where
  \[
     C_{\xi} = \| \xi \|_{\alpha -1} + \| \vartheta \|_{\alpha} + \| \vartheta \reso   \xi \|_{2 \alpha -1} + \| \vartheta \|_{\alpha} \| \xi \|_{\alpha -1} \qquad \tmop{and} \qquad C_{F} = \| F \|_{C^{2}_{b}} + \| F \|_{C^{2}_{b}}^{2} .
  \]   
\end{lemma}

The estimate for $\Phi^{\sharp}$ follows from a somewhat lengthy but
elementary calculation based on the decomposition~\eqref{eq:ode phi sharp modified def}, where we estimate the $L^{\infty}$ norm rather than the
$\CC^{2 \alpha -1}$ norm for each term where this is possible.

Plugging in the correct initial condition for $u^{\sharp}$ leads to
\begin{align*}
   u^{\sharp} ( t ) &= u_{0} - ( F ( u ) \lpara \vartheta ) ( 0 ) + \int_{0}^{t}\partial_{s} u^{\sharp} ( s ) \mathd s \\
   & =u_{0} - ( F ( u ) \lpara \vartheta ) ( 0 ) + \int_{0}^{t} (\varphi \Phi^{\sharp}) ( s ) \mathd s- \int_{0}^{t} ( \partial_{s} \varphi ) ( s )  ( F ( u ) \lpara \vartheta ) ( s ) \mathd s.
\end{align*}
Now $\varphi$ is compactly supported, and therefore Lemma~\ref{lemma:integral} gives estimates for the integrals appearing on
the right hand side in terms of distributional norms of the integrands, and we
obtain the bound
\begin{align*}
   \| u^{\sharp} -u_{0} \|_{2 \alpha} & \lesssim \| F ( u ) \lpara \vartheta \|_{2 \alpha -1} + \| \Phi^{\sharp} \|_{2 \alpha -1} \\
   & \lesssim C_{F} C_{\xi} ( 1+ \| u-u_{0} \|_{\alpha} + \| u-u_{0} \|^{2}_{\alpha} + \| u^{\sharp} -u_{0} \|_{2 \alpha} ).
\end{align*}
Using that $u= \varphi (F ( u ) \lpara \vartheta ) +u^{\sharp}$, we moreover have
\[
   \| u-u_{0} \|_{\alpha} \lesssim  \| F \|_{L^{\infty}} \| \vartheta \|_{\alpha} + \| u^{\sharp} -u_{0} \|_{2 \alpha}.
\]
From these two estimates we deduce that if $C_{F}$ is small enough (depending only on $C_{\xi}$ and $\varphi$ but not on $| u_{0} |$), then $\| u^{\sharp} \|_{2 \alpha} \leqslant | u_{0} | +1$. 
This is the required uniform estimate on the problem.

Similarly we can show that if $F \in C^{3}_{b}$ and if $\| F \|_{C^{3}_{b}}$
is small enough, then the map
\[
   (u_0, \xi , \vartheta , \xi \reso   \vartheta) \mapsto ( u,u^{\sharp} )
\]
is locally Lipschitz continuous from $\CC^{\alpha -1} \times \CC^{\alpha}\times \CC^{2 \alpha -1} \times \mathbb{R}^{d}$ to $\CC^{\alpha} \times \CC^{2 \alpha -1}$. To summarize:

\begin{lemma}\label{lemma:ode conditional existence uniqueness}
  Let $a>0$ and let $\| F\|_{C^{3}_{b}}$ be sufficiently small (depending on $a$). Let $\xi$, $\vartheta$,
  and $\varphi$ be smooth functions with $\xi = \partial_{t} \vartheta$ and
  such that $\varphi$ has compact support. If $\alpha >1/3$ and
  \begin{equation}\label{eq:bounds on ode data}
     \max \{ \| \xi \|_{\alpha -1} , \| \vartheta \|_{\alpha} , \| \xi \reso \vartheta \|_{2 \alpha -1} , \| \varphi \|_{C^{1}_{b}} \} \leqslant a,
  \end{equation}
  then for every $u_0 \in \R^d$ there exists a unique global solution $u$ to
  \[
     \partial_{t} u= \varphi F ( u ) \xi, \qquad u(0) = u_0.
  \]
  For fixed $\varphi$ and $F$, $u$ depends Lipschitz continuously on $(u_0, \xi, \vartheta , \xi \reso \vartheta)$ satisfying~\eqref{eq:bounds on ode data}.
\end{lemma}

In order to ensure that $\| F \|_{C^{3}_{b}}$ is small enough we can use a
dilation argument. Recall that the scaling operator $\Lambda_\lambda$ is defined for $\lambda > 0$ by $\Lambda_\lambda u = u(\lambda \cdot)$. If we let $u^{\lambda} = \Lambda_{\lambda} u$ and
$\xi^{\lambda} = \lambda^{1- \alpha} \Lambda_{\lambda} \xi$ for $\lambda >0$,
then $u^{\lambda}$ solves
\[
   \partial_{t} u^{\lambda} = \lambda^{\alpha} F ( u^{\lambda} ) \xi^{\lambda}, \qquad u^\lambda(0) = u_0.
\]
The rescaling of $\xi^{\lambda}$ is chosen so that its $\CC^{\alpha}$ norm is
uniformly bounded by that of $\xi$ as $\lambda \rightarrow 0$. Indeed,
Lemma~\ref{lemma:scaling} yields
\[
   \| \xi^{\lambda} \|_{\alpha -1} = \lambda^{1- \alpha} \| \Lambda_{\lambda} \xi \|_{\alpha -1} \lesssim ( 1+ \lambda^{1- \alpha} ) \| \xi \|_{\alpha -1} \lesssim \| \xi \|_{\alpha -1}
\]
for $\lambda \leqslant 1$. If moreover we let $\vartheta^{\lambda} = \lambda^{- \alpha} \Lambda_{\lambda} \vartheta$, then $\| \vartheta^{\lambda} \reso \xi^{\lambda} \|_{2 \alpha -1} \lesssim \| \vartheta \reso   \xi \|_{2 \alpha -1} + \| \vartheta\|_\alpha \|\xi\|_{\alpha-1}$ by Lemma~\ref{lem:scaling commutator}. Thus, we deduce from Lemma \ref{lemma:ode conditional existence uniqueness} that for every $\varphi$
of compact support there exists $\lambda >0$, such that for all $u_{0} \in \mathbb{R}^{d}$ we have a unique
global solution $u^{\lambda}$ to
\[
   \partial_{t} u^{\lambda} = \varphi \lambda^{\alpha} F ( u^{\lambda} ) \xi^{\lambda}, \qquad u^\lambda(0) = u_0.
\]
The rescaled problem is
equivalent to the original one upon the change $F \rightarrow \lambda^{\alpha} F$, $\xi \rightarrow \xi^{\lambda}$ and $\vartheta \reso   \xi \rightarrow \vartheta^{\lambda} \reso   \xi^{\lambda}$. So if we set $u= \Lambda_{\lambda^{-1}} u^{\lambda}$, then $u$ is the unique global solution to
\[
   \partial_{t} u= \varphi_{\lambda} F ( u ) \xi, \qquad u(0) = u_0,
\]
where we set $\varphi_\lambda(t) = \varphi(t/\lambda)$. In particular, if $\varphi \equiv 1$
on $[ -1,1 ]$, then $u$ is the unique solution to the original \textsc{rde} in the
interval $[ - \lambda, \lambda]$.
Since $\lambda$ can be chosen independently of $u_{0}$, we can now iterate on
intervals of length $2\lambda$, and obtain a global solution $u \in \CC^{\alpha}_{\tmop{loc}}$.

This analysis can be summarized in the following statement.

\begin{theorem}\label{theorem:rde existence uniqueness}
  Let $\alpha >1/3$. Assume that $(\xi^{\varepsilon} )_{\varepsilon >0}$ is a family of smooth functions with
  values in $\mathbb{R}^{n}$, $(u^\varepsilon_0)$ is a family of initial conditions in $\R^d$, and $F = (F^1,\dots, F^n)$ is a family of $C^{3}_{b}$ vector
  fields on $\mathbb{R}^{d}$. Suppose that there exist $u_0 \in \R^d$, $\xi \in \CC^{\alpha -1}$ and $\eta \in \CC^{2 \alpha -1}$ such that $(u_0^\varepsilon, \xi^{\varepsilon} , \vartheta^\varepsilon, (\vartheta^{\varepsilon} \reso   \xi^{\varepsilon} ) )$ converges to $(u_0, \xi , \vartheta, \eta )$ in $\CC^{\alpha -1} \times \CC^\alpha \times \CC^{2 \alpha -1}$, where
  $\vartheta^{\varepsilon}$ and $\vartheta$ are solutions to $\partial_{t}  \vartheta^{\varepsilon} = \xi^{\varepsilon}$ and $\partial_t \vartheta = \xi$, respectively. Let for $\varepsilon >0$ the function $u^{\varepsilon}$ be the
  unique global solution to the Cauchy problem
  \[
     \partial_{t} u^{\varepsilon} =F ( u^{\varepsilon} ) \xi^{\varepsilon} , \hspace{2em} u^{\varepsilon} ( 0 ) =u^\varepsilon_{0} .
  \]
  Then there exists $u \in \CC^{\alpha}_\mathrm{loc}$ such that $u^{\varepsilon} \rightarrow u$ in $\CC^{\alpha}_{\tmop{loc}}$ as $\varepsilon \rightarrow 0$. The limit $u$ depends only on $( u_{0} , \xi , \vartheta , \eta )$, and not on the approximating family $(u_0^\varepsilon, \xi^{\varepsilon} , \vartheta^\varepsilon,  ( \vartheta^{\varepsilon} \reso   \xi^{\varepsilon} ) )$.
\end{theorem}

\begin{proof}
  The only point which remains to be shown is the convergence of $(u^{\varepsilon} )$ to $u$ in $\CC^{\alpha}_{\tmop{loc}}$. A priori,
  we only know that for sufficiently small $\lambda >0$, the solutions
  $\tilde{u}^{\varepsilon}$ to $\partial_{t} \tilde{u}^{\varepsilon} = \varphi_{\lambda} F ( \tilde{u}^{\varepsilon} ) \xi^{\varepsilon}$ with
  $\tilde{u}^{\varepsilon} ( 0 ) =u_{0}$ converge, as $\varepsilon \rightarrow 0$, in $  \CC^{\alpha}$ to a unique limit $\tilde{u}$.
  But since $\varphi_{\lambda} \equiv 1$ on $[ - \lambda, \lambda]$, we
  have $\tilde{u}^{\varepsilon} |_{[ - \lambda, \lambda]} =u^{\varepsilon} |_{[ - \lambda, \lambda ]}$. So if we define $u|_{[ -  \lambda, \lambda ]} = \tilde{u} |_{[ - \lambda, \lambda]}$, then
  $u|_{[ - \lambda, \lambda]}$ does not depend on $\varphi_{\lambda}$.
  Moreover, for every $\psi \in \CD$ with support contained in $[ - \lambda, \lambda]$, we also have that $\| \psi ( u^{\varepsilon} -u ) \|_{\alpha}$ converges to zero as $\varepsilon \rightarrow 0$. Now we can
  iterate this construction of $u$ on intervals of length $2 \lambda$. We
  end up with a distribution $u \in \CS'$, which only depends
  on $( u_{0} ,F, \xi , \vartheta , \eta )$, but not on $\varphi_{\lambda}$
  or on the approximating sequence $(u_0^\varepsilon, \xi^{\varepsilon} , \vartheta^{\varepsilon} , \xi^{\varepsilon} \reso \vartheta^{\varepsilon} )_{\varepsilon >0}$. If $\psi \in \CD$, then it can be written as a
  finite sum of smooth functions with support contained in intervals of length
  $2 \lambda$, and therefore $\psi u= \lim_{\varepsilon \rightarrow} \psi u^{\varepsilon}$, where convergence takes places in $\CC^{\alpha}$.
\end{proof}

\begin{remark}
  By Lemma~\ref{lemma:para-taylor}, it suffices if $F \in C^{2+\beta/\alpha}_b$ for some $\beta >0$ with $2\alpha+\beta>1$ to obtain existence and uniqueness of solutions. If we only suppose $F \in C^{2+\beta/\alpha}$ and not that $F$ and its derivatives are
  bounded, we still obtain local existence and uniqueness of solutions. In
  that case we may consider a function $G \in C^{2+\beta/\alpha}_{b}$ that coincides with
  $F$ on $\{ | x | \leqslant a \}$ for some $a> | u_{0} |$. The Cauchy problem
  \[
     \partial_{t} v=G ( v ) \xi , \hspace{2em} v ( 0 ) =u_{0},
  \]
  then has a unique global solution in the sense of Theorem~\ref{theorem:rde
  existence uniqueness}. If we stop $v$ upon leaving the set $\{ | x | \leqslant a \}$, we obtain a local solution to the \textsc{rde} with vector field $F$.
\end{remark}

\subsection{Interpreting our RDE solutions}

  So far we showed that under the assumptions of Theorem~\ref{theorem:rde existence uniqueness} there exists a unique limit $u$ of the solutions to the
  regularized equations, which does not depend on the particular approximating
  sequence. In that sense, one may formally call $u$ the unique solution to
  \[
     \partial_{t} u=F ( u ) \xi , \hspace{2em} u ( 0 ) =u_{0} .
  \]
  But $u$ is actually a weak solution to the equation if we interpret the product $F(u) \xi$ appropriately. Below we will introduce a map which extends the pointwise product $F(u) \xi$ from smooth $\xi$ to $\xi \in \CC^{\alpha-1}$ by a continuity argument. But first we present an auxiliary result which shows that the considered topologies and operators do not depend on the particular dyadic
partition of unity that we use to describe them.

\begin{lemma}\label{lem:paraproduct difference is bounded operator}
   Let $\alpha, \beta \in \R$. Let $( \chi, \rho )$ and $( \tilde{\chi}, \tilde{\rho} )$ be two dyadic partitions of unity and 
   let $(\lpara, \rpara, \reso)$ and $(\widetilde{\lpara}, \widetilde{\rpara}, \widetilde{\reso})$ denote paraproducts and resonant term defined in terms of $(\chi, \rho)$ and $(\tilde{\chi}, \tilde{\rho})$, respectively.
Then
   \[
      (u,v) \mapsto ( u \lpara v - u \ \widetilde{\lpara} \ v, u \reso v - u \ \widetilde{\reso} \ v, u \rpara v - u \ \widetilde{\rpara} \ v)
   \]
   is a bounded bilinear operator from $\CC^\alpha \times \CC^\beta$ to $(\CC^{\alpha+\beta})^3$.
\end{lemma}

\begin{proof}
   The statement for $(u,v) \mapsto (u \lpara v - u \ \widetilde{\lpara} \ v)$ (and thus for $(u,v) \mapsto (u \rpara v - u \ \widetilde{\rpara} \ v)$) is shown in Bony~\cite{Bony1981}, Theorem 2.1. But for smooth functions $u$ and $v$ we have $u \reso v = uv - u \lpara v - u \rpara v$, and similarly for $u \ \widetilde{\reso} \ v$. Thus, the bound on $u \reso v - u \ \widetilde{\reso} \ v$ follows from the bounds on $u \lpara v - u \ \widetilde{\lpara} \ v$ and on $u \rpara v - u \ \widetilde{\rpara} \ v$ in combination with a continuity argument.
\end{proof}

Our commutator lemma states that if the product $g \reso h$ is given, then we can unambiguously make sense of the product $(f\lpara g) \reso h$ for suitable $f$. This leads us to the following definition.

\begin{definition}\label{def:paracontrolled}
   Let $\alpha \in \R$, $\beta>0$, and let $v \in \CC^\alpha$. A pair of distributions $(u,u') \in \CC^\alpha \times \CC^{\beta}$ is called \emph{paracontrolled} by $v$ if
   \[
      u^\sharp = u - u'\lpara v \in \CC^{\alpha+\beta}.
   \]
   In that case we abuse notation and write $u \in \CD^\beta = \CD^\beta(v)$, and we define the norm
   \[
     \| u \|_{\CD^\beta} = \|u'\|_\beta + \|u^\sharp\|_{\alpha+\beta}. 
   \]
\end{definition}

According to Lemma~\ref{lem:paraproduct difference is bounded operator}, the space $\CD^\beta$ does not depend on the specific partition of unity used to define it. To construct the product $F(u)\xi$, we could now show that smooth $F$ preserve the paracontrolled structure of $u$. This can be achieved by combining Lemma~\ref{lemma:paralinearization} with another commutator lemma (Theorem~2.3 in~\cite{Bony1981}). But we do not need the full strength of that result, let us just show that if $u$ is paracontrolled by $\vartheta$ and $F$ is smooth enough, then $F(u) \xi$ is well defined.

\begin{theorem}\label{theorem:product operator}
  Let $\alpha \in (0,1)$, $\beta \in (0,\alpha]$, $\gamma < 0$ be such that $\alpha + \beta + \gamma >0$. Let $F \in C^{1+\beta/\alpha}$ and let $v \in \CC^\alpha$, $w \in \CC^\gamma$, $\eta \in \CC^{\alpha+\gamma}$ be such that there exist sequences $(v_n) \subseteq \CS$, $(w_n) \subseteq \CS$, converging to $v$ and $w$ respectively, such that $(v_n \reso w_n)$ converges to $\eta$. Then
  \begin{align}\label{eq:paracontrolled product def}
     \CD^\beta(v) \ni u \mapsto F(u) w  & = F(u)\rpara w + F(u) \lpara w + \Pi_F(u,w) + F'(u) (u^\sharp \reso w) \\ \nonumber
     &\quad + F'(u) C(u', v, w) + F'(u) u' \eta \in \CC^\gamma
  \end{align}
  defines a locally Lipschitz continuous function. If $w \in \CS$ and $\eta = v \reso w$, then $F(u) w$ is simply the pointwise product.
  
  The product $F(u)w$ does not depend on the specific dyadic partition used to construct it: If $(\widetilde{\lpara}, \widetilde{\rpara}, \widetilde{\reso})$ denote paraproducts and resonant term defined in terms of another partition unity, if
  \[
     \tilde \eta = \eta + v \lpara w + v \rpara w - v \ \widetilde \lpara \ w - v \ \widetilde \rpara \ w,
  \]
  and $\tilde{u}^\sharp = u' \ \widetilde \lpara \ v$, then $F(u) w$ is equal to the right hand side of~\eqref{eq:paracontrolled product def} if we replace every operator by the corresponding operator defined in terms of $(\widetilde{\lpara}, \widetilde{\rpara}, \widetilde{\reso})$, and we replace $(u^\sharp,\eta)$ by $(\tilde{u}^\sharp,\tilde \eta)$.
\end{theorem}

\begin{proof}
  The local Lipschitz continuity of the product follows from its definition in
  combination with Lemma~\ref{lemma:commutator},
  Lemma~\ref{lemma:para-taylor}, and the paraproduct estimates
  Lemma~\ref{thm:paraproduct}. 
 
  If $w$ is a Schwartz function and $\eta = v \reso w$, then
  \[
     F'(u) C(u', v, w) + F'(u) u' \eta = F'(u) ((u'\lpara v) \reso w),
  \]
  and therefore
  \begin{align*}
     &\Pi_{F} ( u,w ) +F' ( u ) ( u^{\sharp} \reso w ) +F' ( u ) C ( u' ,v,w ) +F' ( u ) u' \eta \\
     &\hspace{100pt} = \Pi_{F} ( u,w ) +F' ( u ) ( u \reso w ) =F ( u ) \reso w,
  \end{align*}
  which shows that we recover $F(u) \lpara w + F(u) \rpara w + F(u) \reso w$, i.e. the pointwise product.
  
  It remains to show that $F(u) w$ does not depend on the specific dyadic partition of unity. By continuity of the operators involved, we have
  \begin{align*}
     F(u) w & = \lim_{n \to \infty} \Big[ F ( u ) \lpara w_n + F ( u ) \rpara w_n + \Pi_{F} ( u,w_n ) +F' ( u ) (u^\sharp \reso w_n) \\
     &\hspace{50pt} +F'(u) C(u', v_n, w_n) + F'(u) u' (v_n \reso w_n) \Big] \\
     & = \lim_{n \to \infty} \Big[ F(u) w_n + F'(u) ((u' \lpara (v_n - v)) \reso w_n) \Big].
  \end{align*}
  Assume now that we defined $F(u) \cdot w$ in terms of another partition of unity, as described above. Then Lemma~\ref{lem:paraproduct difference is bounded operator} implies the convergence of $(v_n \ \widetilde \reso \ w_n)$ to $\tilde \eta$ in $\CC^{\alpha+\gamma}$, and therefore
  \[
     F(u) \cdot w = \lim_{n \to \infty} \Big[ F(u) w_n + F'(u) ((u' \ \widetilde \lpara \ (v_n - v)) \ \widetilde \reso \ w_n) \Big].
  \]
  Another application of Lemma~\ref{lem:paraproduct difference is bounded operator} then yields $F(u) w = F(u) \cdot w$.
\end{proof}

\begin{remark}\label{rmk:product smooth perturbation}
   If in the setting of Theorem~\ref{theorem:product operator} we let $\tilde v = v + f$ for some $f \in \CC^{\alpha+\beta}$, then we have $\CD^\beta(v) = \CD^\beta(\tilde v)$, and it is easy to see that if we set $\tilde \eta = \eta + f \reso w$, $\tilde u^\sharp = u - u'\lpara \tilde v$, and define $\widetilde{F(u)w}$ like $F(u)w$, with $\tilde v$, $\tilde u^\sharp$, $\tilde \eta$ replacing $u^\sharp, v, \eta$, then $\widetilde{F(u)w} = F(u)w$.
\end{remark}

  With this product operator at hand, it is relatively straightforward to show that if $\xi$ has compact support (which in general is necessary to have $u \in \CC^\alpha$ and not just in $\CC^\alpha_{\mathrm{loc}}$), then the solution $u$ that we constructed in Theorem~\ref{theorem:rde existence uniqueness} is the unique element of $\CD^\alpha$ which solves $\partial_t u = F(u) \xi$, $u(0) = u_0$, in the weak sense. Remark~\ref{rmk:product smooth perturbation} explains why we did not fix the initial condition $\vartheta(0)$ in Theorem~\ref{theorem:rde existence uniqueness}: it is of no importance whatsoever.

\subsection{Alternative approach}

We briefly describe an alternative approach to \textsc{rde}s which avoids the
paracontrolled ansatz. The idea is to control $u \reso   \xi$ directly by
exploiting that $u$ solves the differential equation $\partial_{t} u=F ( u ) \xi$. Indeed, let as above $\vartheta$ be a solution to $\partial_{t} \vartheta = \xi$ and observe that the Leibniz rule yields
\[
   u \reso   \xi =u \reso   \partial_{t} \vartheta = \partial_{t} ( u \reso  \vartheta ) - (\partial_{t} u) \reso \vartheta = \partial_{t} ( u \reso   \vartheta ) - ( F ( u ) \xi ) \reso   \vartheta.
\]
Now the second term on the right hand side can be rewritten as
\begin{align*}
   ( F ( u ) \xi ) \reso   \vartheta & = ( F ( u ) \lpara \xi ) \reso   \vartheta + ( F ( u ) \reso \xi ) \reso   \vartheta + ( F ( u ) \rpara \xi ) \reso  \vartheta \\
   & =F ( u ) ( \xi \reso   \vartheta ) +C ( F ( u ) , \xi , \vartheta ) + ( F'( u ) ( u \reso   \xi ) ) \reso   \vartheta + \\
   &\quad + \Pi_{F} ( u, \xi ) \reso  \vartheta + ( F ( u ) \rpara \xi ) \reso   \vartheta.
\end{align*}
Combining these two equations, we see that
\begin{gather*}
   u \reso   \xi = \Phi - ( F' ( u ) ( u \reso   \xi ) ) \reso   \vartheta , \qquad \text{where} \\
   \Phi = \partial_{t} ( u \reso   \vartheta ) -F ( u ) ( \xi \reso \vartheta ) -C ( F ( u ) , \xi , \vartheta ) - \Pi_{F} ( u, \xi ) \reso \vartheta - ( F ( u ) \rpara \xi ) \reso   \vartheta.
\end{gather*}
This is an implicit equation for $u \reso   \xi$ which can be solved by fixed point methods. For example, it is easy to obtain the estimate
\[
   \| u \reso   \xi \|_{2 \alpha -1} \lesssim \| \Phi \|_{2 \alpha -1} +C_{F} \| u \reso   \xi \|_{2 \alpha -1} \| \vartheta \|_{\alpha},
\]
and if $C_{F}$ is small enough this leads to $\| u \reso   \xi \|_{2 \alpha -1} \lesssim \| \Phi \|_{2 \alpha -1}$. Moreover, we have $\| \Phi \|_{2 \alpha -1} \lesssim C_{\xi} [ \| u \|_{\alpha} +C_{F} ( 1+ \| u \|_{\alpha} )^{2} ]$. These estimates can be
reinjected into the equation
\[
   \partial_{t} u=F ( u ) \xi =F ( u ) \lpara \xi +F' ( u ) ( u \reso   \xi ) +F ( u ) \rpara \xi + \Pi_F(u,\xi)
\]
to obtain a local estimate for $u$.

\subsection{Connections to rough paths and existence of the area}

We saw in the previous section that the solution $u$ to an \textsc{rde} of the form
$\partial_{t} u=F ( u ) \xi$ depends on the driving signal in a continuous way, provided
that we not only keep track of $\xi$ but also of $\vartheta \reso \xi$. From
the theory of rough paths it is well known that the same holds true if we keep
track of $\vartheta$ and its iterated integrals $\int \int \mathd \vartheta \mathd \vartheta$. But in fact the convergence of $( \vartheta^{\varepsilon} \reso \xi^{\varepsilon} )$
is equivalent to the convergence of the iterated integrals $\int \int \mathd \vartheta^{\varepsilon} \mathd \vartheta^{\varepsilon}$:

\begin{corollary}\label{corollary:classical rough path}
  Let $( u^{\varepsilon} ,v^{\varepsilon} )_{\varepsilon >0} \subseteq \CS(\R)^2$ and define for every $\varepsilon >0$ the ``area''
  \[
     A^{\varepsilon}_{s,t} = \int_{s}^{t} \int_{s}^{r_{2}} \mathd u^{\varepsilon} ( r_{1} ) \mathd v^{\varepsilon} ( r_{2} ) , \hspace{2em} s<t \in \mathbb{R}.
  \]
  Let $\alpha , \beta \in ( 0,1 )$ with $\alpha + \beta <1$ and let $u \in \CC^\alpha ,v \in \CC^\beta , \eta \in \CC^{\alpha + \beta -1}$. Then $( u^{\varepsilon} ,v^{\varepsilon} ,u^{\varepsilon} \reso \partial_{t} v^{\varepsilon} )$ converges to $( u,v, \eta )$ in $\CC^{\alpha} \times \CC^\beta \times \CC^{\alpha + \beta -1}$ if and only if $( u^{\varepsilon},v^{\varepsilon} )$ converges to $( u,v )$ in $\CC^{\alpha} \times \CC^\beta$, and if moreover
  \begin{equation}\label{eq:classical rough path}
     \lim_{\varepsilon \rightarrow 0} \left( \sup_{s \neq t \in \mathbb{R}, | s-t | \leqslant 1} \frac{| A_{s,t} -A^{\varepsilon}_{s,t} |}{| t-s |^{\alpha + \beta}} \right) =0,
  \end{equation}
  where we set $A_{s,t} = \int_{s}^{t} ( \eta + ( u  \lpara \partial_{t} v ) + ( u \rpara \partial_{t} v ) ) ( r ) \mathd r-u ( s ) ( v ( t ) -v ( s ) ) $ for $s,t \in \mathbb{R}$.
\end{corollary}

\begin{proof}
  First suppose that $( u^{\varepsilon} ,v^{\varepsilon} ,u^{\varepsilon} \reso \partial_{t} v^{\varepsilon} )$ converges to $( u,v, \eta )$ in
  $\CC^{\alpha} \times \CC^\beta \times \CC^{\alpha + \beta -1}$, and let $s, t \in \mathbb{R}$ with $| s-t | \leqslant 1$.
  We have
  \begin{align}\label{eq:classical rough path pr1} \nonumber
     A_{s,t} -A^{\varepsilon}_{s,t} 
     & = \int_{s}^{t} ( \eta +u \rpara \partial_{t} v-u^{\varepsilon} \reso  \partial_{t} v^{\varepsilon} -u^{\varepsilon} \rpara \partial_{t}  v^{\varepsilon} ) ( r ) \mathd r \\ \nonumber
     &\quad + \int_{s}^{t} ( ( u^{\varepsilon} -u ) \lpara \partial_{t} v^{\varepsilon} ) ( r ) \mathd r- ( u^{\varepsilon} -u ) ( s ) ( v^{\varepsilon} ( t ) -v^{\varepsilon} ( s ) ) \\
     &\quad + \int_{s}^{t} ( u \lpara \partial_{t} ( v^{\varepsilon} -v ) ) ( r ) \mathd r- ( u ) ( s ) ( ( v^{\varepsilon} -v ) ( t ) - ( v^{\varepsilon} -v ) ( s ) ) .
  \end{align}
  The first term on the right hand side can be estimated with the help of
  Lemma~\ref{lemma:integral}, which allows us to bound increments of the integral in terms of Besov norms of the integrand. We get
  \begin{align*}
     &\left| \int_{s}^{t} ( \eta +u \rpara \partial_{t} v-u^{\varepsilon} \reso  \partial_{t} v^{\varepsilon} -u^{\varepsilon} \rpara \partial_{t} v^{\varepsilon} ) ( r ) \mathd r \right| \\
     &\hspace{5pt} \lesssim ( \| \eta -u^{\varepsilon} \reso \partial_{t} v^{\varepsilon} \|_{\alpha + \beta -1} + \| u-u^{\varepsilon} \|_{\alpha} \| \partial_{t}  v \|_{\beta -1} + \| u^{\varepsilon} \|_{\alpha} \| \partial_{t} (v^{\varepsilon} -v ) \|_{\beta -1} ) | t-s |^{\alpha + \beta}.
  \end{align*}
  Since $\| \partial_{t} ( v^{\varepsilon} -v ) \|_{\beta -1} \lesssim \| v^{\varepsilon} -v \|_{\beta}$, the right hand side goes to zero if we divide it by $| t-s |^{\alpha + \beta}$ and let $\varepsilon \rightarrow 0$.
  
  The second term on the right hand side of~{\eqref{eq:classical rough path pr1}} can be estimated using Lemma~\ref{lemma:controlled old and new}, which roughly states that time integral and paraproduct commute with each other, at the price of introducing a smoother remainder term:
  \[
     \left| \int_{s}^{t} ( ( u^{\varepsilon} -u ) \lpara \partial_{t} v^{\varepsilon} ) ( r ) \mathd r- ( u^{\varepsilon} -u ) ( s ) ( v^{\varepsilon} ( t ) -v^{\varepsilon} ( s ) ) \right| \lesssim | t-s |^{\alpha + \beta} \| u^{\varepsilon} -u \|_{\alpha} \| v^{\varepsilon} \|_{\beta},
  \]
  The third term on the right hand side of~{\eqref{eq:classical rough path pr1}} is of the same type as the second term, and therefore the convergence
  in~{\eqref{eq:classical rough path}} follows.
  
  Conversely, assume that $( u^{\varepsilon} ,v^{\varepsilon} )$ converges to
  $( u,v )$ in $\CC^{\alpha} \times \CC^\beta$, and that the
  convergence in~{\eqref{eq:classical rough path}} holds. It follows from the
  representation~{\eqref{eq:classical rough path pr1}} that also
  \[
     \lim_{\varepsilon \rightarrow 0} \Bigg( \sup_{s \neq t \in \mathbb{R}, | s-t | \leqslant 1} \frac{\big| \int_{s}^{t} ( \eta -u^{\varepsilon} \reso \partial_{r} v^{\varepsilon} ) ( r ) \mathd r \big|}{| t-s |^{\alpha + \beta}} \Bigg) =0.
  \]
  Due to the restriction $| s-t | \leqslant 1$, it is not entirely obvious
  that this implies the convergence of $u^{\varepsilon} \reso \partial_{r} v^{\varepsilon}$ to $\eta$ in $\CC^{\alpha + \beta -1}$. However, here
  we can use an alternative characterization of Besov spaces in terms of
  local means. Let $k^{0}$ and $k$ be infinitely differentiable functions on
  $\mathbb{R}$ with support contained in $( -1,1 )$, such that $\mathcal{F}k^{0} (0 ) \neq 0$, and such that there exists $\delta >0$ with $\mathcal{F}k ( z ) \neq 0$ for all $0< | z | < \delta$. Then an equivalent norm on $\CC^{\alpha + \beta -1} ( \mathbb{R} )$ is given by
  \[
     \| w \|_{\alpha + \beta -1} \simeq \max \Big\{ \| k^{0} \ast w \|_{L^{\infty}} , \sup_{j \geqslant 0}  2^{j ( \alpha + \beta -1 )} \| 2^{j} k ( 2^{j}  \cdummy ) \ast w \|_{L^{\infty}} \Big\} ,
  \]
  see~\cite{Triebel2006}, Theorem 1.10. Let us write $f= \int_{0}^{\cdummy} (\eta -u^{\varepsilon} \reso \partial_{r} v^{\varepsilon} ) ( r ) \mathd r$
  and let $t \in \mathbb{R}$ and $j \geqslant 0$. Then
  \begin{align*}
     | 2^{j} k ( 2^{j} \cdummy ) \ast (\partial_t f) ( t ) | & = 2^{2j} \left| \int_{\mathbb{R}} (\partial_{t} k) (2^j (t-s) ) ( f ( t ) -f ( s ) ) \mathd s \right| \\
     & \lesssim 2^{2j} \int_{\mathbb{R}} | ( \partial_{t} k ) ( 2^{j} ( t-s ) ) |   | t-s |^{\alpha + \beta} \mathd s \sup_{| a-b | \leqslant 1} \frac{| f( b ) -f ( a ) |}{| b-a |^{\alpha + \beta}}\\
     & \lesssim 2^{-j ( \alpha + \beta -1 )} \sup_{| a-b | \leqslant 1} \frac{| f ( b ) -f ( a ) |}{| b-a |^{\alpha + \beta}},
  \end{align*}
  where we used that $\int_{\mathbb{R}} \partial_{t} k ( t ) \mathd t=0$,
  and that $k$ is supported in $( -1,1)$. Similarly, we obtain
  \begin{align*}
     | k^{0} \ast ( \partial_{t} f ) ( t ) | & \lesssim \int_{\mathbb{R}} | \partial_{t} k^{0} ( t-s ) |   | t-s |^{\alpha + \beta} \mathd s \sup_{| a-b | \leqslant 1} \frac{| f ( b ) -f ( a ) |}{| b-a |^{\alpha + \beta}} \\
     & \lesssim \sup_{| a-b | \leqslant 1} \frac{| f ( b ) -f ( a ) |}{| b-a |^{\alpha + \beta}},
  \end{align*}
  from where the convergence of $u^{\varepsilon} \reso \partial_{t} v^{\varepsilon}$ to $\eta$ in $\CC^{\alpha + \beta -1}$ follows.
\end{proof}

\begin{corollary}\label{corollary:existence of rough path area}
  Let $X$ be an $n$--dimensional centered Gaussian process with independent components and measurable trajectories, whose covariance function satisfies for some $H \in ( 1/4,1 )$ the inequalities
  \begin{gather}\label{eq:classical rough path covariance} \nonumber
     \mathbb{E} [ | X_{t} -X_{s} |^{2} ] \lesssim | t-s |^{2H} \hspace{2em} \tmop{and} \\
     | \mathbb{E} [ ( X_{s+r} -X_{s} ) ( X_{t+r} -X_{t} ) ] | \lesssim | t-s |^{2H-2}  r^{2}
  \end{gather}
  for all $s,t \in \mathbb{R}$ and all $r \in [ 0, | t-s | )$. Then $\varphi X \in \CC^{\alpha}$ for all $\alpha <H$ and all $\varphi \in \CD$, and there
  exists $\eta \in \CC^{2 \alpha -1}$ such that for every $\psi \in \CS$ with
  $\int \psi \mathd t=1$ and for every $\delta >0$ we have
  \[
     \lim_{\varepsilon \rightarrow 0} \mathbb{P} \left( \| \psi^{\varepsilon}  \ast ( \varphi X ) - ( \varphi X ) \|_{\alpha} + \| ( \psi^{\varepsilon}  \ast ( \varphi X ) ) \reso \partial_{t} ( \psi^{\varepsilon} \ast (\varphi X ) ) - \eta \|_{\CC^{2 \alpha -1}} > \delta \right) =0,
  \]
  where we define $\psi^{\varepsilon} = \varepsilon^{-1} \psi (\varepsilon^{-1} \cdummy )$.
\end{corollary}

\begin{proof}
  Since $\varphi$ is smooth and of compact support, it is easy to see that
  also the Gaussian process $\varphi X$ satisfies the covariance
  condition~{\eqref{eq:classical rough path covariance}}, and using Gaussian hypercontractivity we obtain $\mathbb{E} [ | \varphi(t)X_{t} -\varphi(s)X_{s} |^{2p} ] \lesssim | t-s |^{2Hp}$ for all $p \geqslant 1$. Using the fact that $X$ has measurable trajectories, we can apply this estimate to show that $\E[\| \varphi X\|_{B^\alpha_{2p,2p}}^{2p}] < \infty$ for all $p \geqslant 1$, $\alpha < H$. Now it suffices to apply Besov embedding, Lemma~\ref{l:besov embedding torus}, to obtain that $\varphi X \in \CC^\alpha$.

  Moreover, $\varphi X$ has compact support. So by Theorem 15.45 of~{\cite{Friz2010}}, for every $p \geqslant 1$, the
  iterated integrals $\int_{s}^{t} \int_{s}^{r_{2}} \mathd \psi^{\varepsilon} \ast ( \varphi X ) ( r_{1} ) \mathd \psi^{\varepsilon} \ast ( \varphi X ) ( r_{2} )$ converge in $L^{p}$ in the sense of~{\eqref{eq:classical rough path}}. The statement then follows from Corollary~\ref{corollary:classical rough path}.
\end{proof}

\begin{remark}
  The proof of Corollary~\ref{corollary:classical rough path} actually shows
  more than the equivalence of the convergence of $A^{\varepsilon}$ and of
  $u^{\varepsilon} \reso \partial_{t} v^{\varepsilon}$: it shows that the norm
  of $( u^{\varepsilon} \reso \partial_{t} v^{\varepsilon} - \eta )$ can be
  controlled by a polynomial of the norms of $( A^{\varepsilon} -A )$, $(u^{\varepsilon} -u )$, and $(v^{\varepsilon} -v)$. So in fact we
  have $L^{p}$--convergence in Corollary~\ref{corollary:existence of rough path area}, and not just convergence in probability. Alternatively, the $L^p$--convergence is obtained from the convergence in probability because we are considering random variables living in a fixed Gaussian chaos, see Theorem 3.50 of~\cite{Janson1997}.
\end{remark}

Combining Corollary~\ref{corollary:existence of rough path area} with Theorem~\ref{theorem:rde existence uniqueness}, we obtain the following corollary:

\begin{corollary}
  Let $X$ be a $n$--dimensional centered Gaussian process satisfying the conditions of Corollary~\ref{corollary:classical rough path} for some $H > 1/3$, and let $\varphi \in \CD$ and $F \in C^3_b$. Then there exists a unique solution $u$ to
  \[
     \partial_t u = F(u) \partial_t (\varphi X), \qquad u(0) = u_0,
  \]
  in the following sense: If $\psi \in \CS$ with $\int \psi \mathd t=1$ and if for $\varepsilon > 0$ the function $u^\varepsilon$ solves
  \[
     \partial_t u^\varepsilon = F(u^\varepsilon) \partial_t (\varphi X)^\varepsilon, \qquad u(0) = u_0,
  \]
  where $(\varphi X)^\varepsilon = \varepsilon^{-1} \psi(\varepsilon \cdot) \ast (\varphi X)$, then $u^\varepsilon$ converges to $u$ in probability in $\CC^\alpha$ for all $\alpha < H$.
\end{corollary}

\section{Rough Burgers equation}\label{sec:burgers}

Fix now $\sigma > 5/6$ and consider the following PDE on $[ 0,T ] \times \mathbb{T}$ for some fixed $T>0$:
\begin{equation}\label{eq:burgers-rough}
  L u=G ( u ) \partial_{x} u+ \xi, \qquad u(0) = u_0,
\end{equation}
where $L= \partial_{t} + ( - \Delta )^{\sigma}$. We would like to consider solutions $u$ in the case of a
distributional $\xi$, and in particular we want to allow $\xi$ to be a typical
realization of a space-time white noise. We will see below that in this case
the solution $\vartheta$ to the linear equation $L  \vartheta = \xi$,
$\vartheta ( 0 ) =0$, belongs (locally in time) to $\CC^{\alpha} ( \mathbb{T})$ for any $\alpha < \sigma -1/2$, but it is not better than that. This is
also the regularity to be expected from the solution $u$ of the non-linear
problem~{\eqref{eq:burgers-rough}}, and so for $\sigma \leqslant 1$ the term $G ( u ( t ) ) \partial_{x} u ( t )$ is not well defined since $G ( u ( t ) ) \in \CC^{\alpha} ( \mathbb{T} )$ and $\partial_{x} u ( t ) \in \CC^{\alpha -1} (\mathbb{T} )$, and the sum of their regularities fails to be positive.

For $\sigma =1$ equation~\eqref{eq:burgers-rough} has been solved by
Hairer~{\cite{Hairer2011}}, who used rough path integrals to define the
product $G ( u ) \partial_{x} u$. In the following, we will show how to solve
the equation using paracontrolled distributions.

While in general it is possible to set up the equation in a space-time Besov
space, the fact that the distribution $\xi$ (which is a genuine space-time
distribution) enters the problem linearly allows for a small simplification.
Indeed, if we let $w=u- \vartheta$, then $w$ solves the PDE
\begin{equation}\label{eq:burgers-alt}
  L w=G ( \vartheta +w ) \partial_{x} ( \vartheta +w ),
\end{equation}
which can be studied as an evolution equation for a continuous function of
time with values in a suitable H{\"o}lder-Besov space.

Recall that for $T>0$
and $\beta \in \mathbb{R}$ we defined the spaces $C_{T} \CC^{\beta} =C( [ 0,T ], \CC^{\beta})$ with norm $\| u \|_{C_{T} \CC^{\beta}} = \sup_{0 \leqslant s \leqslant T} \| u ( s ) \|_{\beta}$. By the regularity theory for $L$ we expect $w \in C_{T} \CC^{\alpha -1+2 \sigma}$ whenever $G ( \vartheta +w ) \partial_{x} ( \vartheta +w ) \in C_{T}\CC^{\alpha -1}$ (at least in the sense of uniform estimates as the
regularization goes to zero). The paraproduct allows us to decompose the right
hand side of~{\eqref{eq:burgers-alt}} as
\begin{align*}
   G ( \vartheta +w ) \partial_{x} ( \vartheta +w ) & =G ( \vartheta +w ) \lpara \partial_{x} \vartheta +G ( \vartheta +w ) \reso \partial_{x} \vartheta \\
   &\quad +G( \vartheta +w ) \rpara \partial_{x} \vartheta +G ( \vartheta +w ) \partial_{x} w,
\end{align*}
where we have expanded only the term containing $\partial_{x} \vartheta$ since
the one linear in $\partial_{x} w$ is well defined under the hypothesis that
$w \in C_{T} \CC^{\alpha -1+2 \sigma}$. Note that here we only let the
paraproduct act on the spatial variables, i.e. $G ( \vartheta +w ) \lpara \partial_{x} \vartheta$ should really be understood as
\[
   t \mapsto G (\vartheta ( t ) +w ( t ) ) \lpara \partial_{x} \vartheta ( t ),
\]
an element of $C_{T} \CC^{\alpha -1}$. A simple modification of the proof of
Lemma~\ref{lemma:paralinearization} (see also Lemma~\ref{lem:mod paralin}) shows that, for $\alpha \in (0,1/2)$, we have
\begin{align*}
   \| G ( \vartheta +w ) -G' ( \vartheta +w ) \lpara \vartheta \|_{2 \alpha} & \lesssim \| G \|_{C^{2}_{b}} ( 1+ \| \vartheta \|^{2}_{\alpha})(1 + \| w \|_{2 \alpha} ) \\
   & \lesssim \| G \|_{C^{2}_{b}} ( 1+ \| \vartheta \|^{2}_{\alpha})(1 +  \| w \|_{\alpha -1+2 \sigma} ) ,
\end{align*}
where we used that $\alpha -1+2 \sigma >2 \alpha$, which holds because $\alpha < \sigma -1/2<2 \sigma -1$. The linear dependence on the norm of $w$ will be
crucial for obtaining global solutions. We can now rewrite
\begin{align*}
   G ( \vartheta +w ) \reso \partial_{x} \vartheta & = ( G ( \vartheta +w ) -G'( \vartheta +w ) \lpara \vartheta ) \reso \partial_{x} \vartheta +C ( G' ( \vartheta +w ) , \vartheta , \partial_{x} \vartheta ) \\
   &\quad + G' ( \vartheta +w )( \vartheta \reso \partial_{x} \vartheta ).
\end{align*}
So if we assume that $( \vartheta \reso \partial_{x} \vartheta ) \in C_{T} \CC^{2 \alpha -1}$, then we have a well behaved representation of the resonant
term $G ( \vartheta +w ) \reso \partial_{x} \vartheta$, and
\begin{align}\label{eq:burgers resonant estimate} \nonumber
   \| G ( \vartheta +w ) \reso \partial_{x} \vartheta \|_{2 \alpha -1} & \lesssim \| G \|_{C^{2}_{b}} ( 1+ \| \vartheta \|^{2}_{\alpha})(1 + \| w \|_{\alpha -1+2 \sigma} ) \| \partial_{x} \vartheta \|_{\alpha -1}\\ \nonumber
   &\quad + \| G'( \vartheta +w ) \|_{\alpha} \| \vartheta \|_{\alpha}^{2} + \| G' ( \vartheta +w ) \|_{\alpha} \| \vartheta \reso \partial_{x} \vartheta \|_{2 \alpha -1}\\
   & \lesssim C_{G} C_{\vartheta} ( 1+ \| w \|_{\alpha -1+2 \sigma} ) ,
\end{align}
where we set
\[
   C_{G} = \| G \|_{C^{2}_{b}} \hspace{2em} \tmop{and} \hspace{2em}  C_{\vartheta} = 1+ \| \vartheta \|_{C_T \CC^\alpha}^{3} + \|\vartheta\|_{C_T\CC^\alpha} \| \vartheta \reso \partial_{x} \vartheta \|_{C_T\CC^{2 \alpha -1}}.
\]
Let us now define
\[
   \Phi = G(\vartheta + w) \partial_x \vartheta = G ( \vartheta +w ) \lpara \partial_{x} \vartheta +G ( \vartheta +w ) \rpara \partial_{x} \vartheta +G ( \vartheta +w ) \reso \partial_{x} \vartheta ,
\]
so that~{\eqref{eq:burgers resonant estimate}} and the paraproduct estimates yield
\begin{equation}\label{eq:burgers phi-estimate}
   \| \Phi \|_{\alpha -1} \lesssim C_{G}C_{\vartheta} ( 1+ \| w \|_{\alpha -1+2 \sigma} ) ,
\end{equation}
and $w$ satisfies $L w= \Phi +G ( \vartheta +w ) \partial_{x} w$. So if we denote by $( P_{t} )_{t \geqslant 0}$ the semigroup generated by $- (- \Delta )^{\sigma}$, then
\begin{equation}\label{eq:mild-burgers}
  w ( t ) =P_{t} u_{0} + \int_{0}^{t} P_{t-s} \Phi ( s ) \mathd s+ \int_{0}^{t} P_{t-s} ( G ( \vartheta ( s ) +w ( s ) ) \partial_{x} w ( s ) ) \mathd s,
\end{equation}
where we assumed that $\vartheta(0) = 0$. Applying the Schauder estimates for the fractional Laplacian (Lemma~\ref{lemma:schauder} and Lemma~\ref{lemma:heat flow smoothing}) to~{\eqref{eq:mild-burgers}}, we obtain for all $t>0$ that
\begin{align*}
   &\| w ( t ) \|_{\alpha -1+2 \sigma} \\
   &\hspace{15pt} = \left\| P_{t} u_{0} + \int_{0}^{t} P_{t-s} \Phi ( s ) \mathd s+ \int_{0}^{t} P_{t-s} ( G ( \vartheta ( s ) +w ( s ) ) \partial_{x} w ( s ) ) \mathd s \right\|_{\alpha -1+2 \sigma} \\
   &\hspace{15pt} \lesssim t^{- ( 2 \sigma -1 ) /2 \sigma} \Big( \| u_{0} \|_{\alpha} + \sup_{s \in [ 0,t ]} ( s^{( 2 \sigma -1 ) / ( 2 \sigma )} \| \Phi ( s )\|_{\alpha -1} ) \Big) \\
   &\hspace{15pt} \quad + \int_{0}^{t} \frac{\| G ( \vartheta ( s ) +w ( s ) ) \partial_{x} w ( s ) \|_{L^{\infty}}}{( t-s )^{( \alpha -1+2 \sigma ) / ( 2 \sigma )}} \mathd s.
\end{align*}
But
now recall from {\eqref{eq:burgers phi-estimate}} that $\| \Phi ( s )\|_{\alpha -1} \lesssim C_{G} C_{\vartheta} ( 1+ \| w ( s ) \|_{\alpha -1+2 \sigma} )$. Moreover, if we choose $\alpha \in ( 1/3, \sigma -1/2 )$ close enough to $\sigma -1/2$, then $\alpha +2 \sigma -2>0$ (recall that $\sigma >5/6$), and therefore
\[
   \| G ( \vartheta ( s ) +w ( s ) ) \partial_{x} w ( s ) \|_{L^{\infty}} \lesssim \| G \|_{L^{\infty}} \| \partial_{x} w ( s ) \|_{\alpha -2+2 \sigma} \lesssim \| G \|_{L^{\infty}} \| w ( s ) \|_{\alpha -1+2 \sigma} .
\]
Thus, we get for all $t \in [ 0,T ]$ that
\begin{align*}
   ( t^{1-1/ ( 2 \sigma )} \| w ( t ) \|_{\alpha -1+2 \sigma} ) & \lesssim  \| u_{0} \|_{\alpha} + C_{\vartheta} C_{G} ( 1+ \sup_{s \in [ 0,t ]} ( s^{( 2 \sigma -1 ) / ( 2 \sigma )} \| w ( s ) \|_{\alpha -1+2 \sigma} ) ) \\
   &\quad + C_{G} t^{1-1/ ( 2 \sigma )} \int_{0}^{t} \frac{( s^{1-1/ ( 2 \sigma )} \| w ( s ) \|_{\alpha +1} )}{( t-s )^{( \alpha -1+2 \sigma ) / ( 2 \sigma )} s^{1-1/ ( 2 \sigma )}} \mathd s.
\end{align*}
Since $( \alpha -1+2 \sigma ) / ( 2 \sigma ) <1$, we have
\[
   t^{1-1/ ( 2 \sigma )} \int_{0}^{t} \frac{\mathd s}{( t-s )^{( \alpha -1+2 \sigma ) / ( 2 \sigma )} s^{1-1/ ( 2 \sigma )}} \lesssim t^{1- ( \alpha -1+2 \sigma ) / ( 2 \sigma )} \lesssim 1
\]
for $t \in [ 0,T ]$. Putting everything together, we conclude that
\begin{equation}\label{eq:burgers unscaled w estimate}
   ( t^{1-1/ ( 2 \sigma )} \| w ( t ) \|_{\alpha -1+2 \sigma} ) \lesssim \| u_{0} \|_{\alpha} + C_{\vartheta} C_{G} ( 1+ \sup_{s \in [ 0,t ]} ( s^{( 2 \sigma -1 ) / ( 2 \sigma )} \| w ( s ) \|_{\alpha -1+2 \sigma} ) ) .
\end{equation}
Using similar arguments, we can also show that uniformly in $t \in [0,T]$
\begin{equation}\label{eq:burgers unscaled w alpha estimate}
   \| w ( t ) \|_{\alpha} \lesssim \| u_{0} \|_{\alpha} + C_{\vartheta} C_{G} ( 1+ \sup_{s \in [ 0,t ]} ( s^{( 2 \sigma -1 ) / ( 2 \sigma )} \| w ( s ) \|_{\alpha -1+2 \sigma} ) ) .
\end{equation}

In order to turn~\eqref{eq:burgers unscaled w estimate} into a bound on $\| w \|_{C_{T} \CC^{\alpha -1+2 \sigma}}$, we use again a scaling argument. We extend the scaling
transformation to the time variable in such a way that it leaves the operator
$L$ invariant. More precisely, for $\lambda >0$ we set $\Lambda_{\lambda} u (t,x ) =u ( \lambda^{2 \sigma} t, \lambda x )$, so that $L  \Lambda_{\lambda} =\lambda^{2 \sigma} \Lambda_{\lambda} L$. Now let $u^{\lambda} =\Lambda_{\lambda} u$, $w^{\lambda} = \Lambda_{\lambda} w$, and
$\vartheta^{\lambda} = \Lambda_{\lambda} \vartheta$. Note that $u^{\lambda} \colon [ 0,T/ \lambda^{2 \sigma} ] \times \mathbb{T}_{\lambda} \rightarrow \mathbb{R}$, where $\mathbb{T}_{\lambda} = \R / (2\pi\lambda^{-1}\Z)$ is a
rescaled torus, and that $w^{\lambda}$ solves the equation
\[
   L w^{\lambda} = \lambda^{2 \sigma} \Lambda_{\lambda} L w= \lambda^{2 \sigma} \Lambda_{\lambda} ( \Phi +G ( w+ \vartheta ) \partial_{x} w ) = \lambda^{2 \sigma} \Lambda_{\lambda} \Phi + \lambda^{2 \sigma -1} G ( w^{\lambda} + \vartheta^{\lambda} ) \partial_{x} w^{\lambda}.
\]
The same derivation as above shows that
\[
   \| \Lambda_{\lambda} \Phi ( t ) \|_{\alpha -1} = \| G(\vartheta^\lambda(t) + w^\lambda(t)) \Lambda_\lambda (\partial_x \vartheta)(t) \|_{\alpha -1} \lesssim C_{G} C_{\vartheta^{\lambda}} ( 1+ \| w^{\lambda} ( t ) \|_{\alpha -1+2 \sigma} ),
\]
where we get using Lemma~\ref{lemma:scaling} and Lemma~\ref{lem:scaling commutator}
\[
   C_{\vartheta^{\lambda}} = \sup_{t \in [ 0,T ]} ( 1+ \| \vartheta^{\lambda} ( t ) \|_{\alpha} )^{3} ( 1+ \| \vartheta^\lambda \reso \Lambda_\lambda (\partial_{x} \vartheta) (t) \|_{2 \alpha -1} ) \lesssim \lambda^{2 \alpha -1} C^2_{\vartheta} \leqslant \lambda^{-1} C^2_{\vartheta}
\]
as long as $\lambda \in ( 0,1 ]$. Thus, we finally conclude that
\begin{align*}
   &( t^{1-1/ ( 2 \sigma )} \| w^{\lambda} ( t ) \|_{\alpha -1+2 \sigma} ) \\
   &\hspace{45pt} \lesssim \| \Lambda_{\lambda} u_{0} \|_{\alpha} + \lambda^{2 \sigma -1} C^2_{\vartheta} C_{G} ( 1+ \sup_{s \in [ 0,t ]} ( s^{( 2 \sigma -1 ) / ( 2\sigma )} \| w^{\lambda} ( s ) \|_{\alpha -1+2 \sigma} ) ) \\
   &\hspace{45pt} \lesssim \| u_{0} \|_{\alpha} + \lambda^{2 \sigma -1} C^2_{\vartheta} C_{G} ( 1+ \sup_{s \in [ 0,t ]} ( s^{( 2 \sigma -1 ) / ( 2 \sigma )} \| w^{\lambda} ( s ) \|_{\alpha -1+2 \sigma} ) )
\end{align*}
for all $\lambda \in ( 0,1 ]$. Since $2 \sigma -1>0$, we get for small enough
$\lambda >0$, depending only on $C_{\vartheta}$ and $C_{G}$ but not on
$u_{0}$, that
\[
   \sup_{t \in [ 0,T ]} ( t^{1-1/ ( 2 \sigma )} \| w^{\lambda} ( t ) \|_{\alpha -1+2 \sigma} ) \lesssim  \| u_{0} \|_{\alpha} +1.
\]
Equation~\eqref{eq:burgers unscaled w alpha estimate} then yields $\| w^{\lambda} \|_{C_T \CC^\alpha}\lesssim  \| u_{0} \|_{\alpha} +1$ and since $u= \Lambda_{\lambda^{-1}} ( w^{\lambda} + \vartheta^{\lambda} )$ get get
\[
   \sup_{t \in [ 0, \lambda^{2 \sigma} T ]} \| u ( t ) \|_{\alpha} \lesssim_{\lambda} \| u_{0} \| + C_{\vartheta}.
\]
This provides the key ingredient for obtaining a uniform estimate on the full time interval $[ 0,T ]$, and then the existence of global solutions to the Burgers equation.

Uniqueness in the space of solutions $u$ with decomposition $u= \vartheta +w$
with $w \in C_{T} \CC^{\alpha -1+2 \sigma}$ can be handled easily along the
lines above, and we obtain the following result:

\begin{theorem}\label{theorem:burgers existence uniqueness}
  Let $\sigma > 5/6$, $\alpha \in ( 1/3, \sigma -1/2 )$, let $T>0$,
  and assume that $( \xi^{\varepsilon} )_{\varepsilon >0}$ is a family of
  smooth functions on $[ 0,T ] \times \mathbb{T}$ with values in
  $\mathbb{R}^{n}$, and $G \in C^{3}_{b} ( \mathbb{R}^{n} ,\Lcal (\mathbb{R}^{n} ,\mathbb{R}^{n} ) )$. Suppose that there exist $\vartheta \in C_{T} \CC^{\alpha}$ and $\eta \in C_{T} \CC^{2 \alpha -1}$ such that $(\vartheta^{\varepsilon} , ( \vartheta^{\varepsilon} \reso   \partial_{x} \vartheta^{\varepsilon} ) )$ converges to $( \vartheta , \eta )$ in $C_{T}
  \CC^{\alpha -1} \times C_{T} \CC^{2 \alpha -1}$, where
  $\vartheta^{\varepsilon}$ are solutions to $L \vartheta^{\varepsilon} = \xi^{\varepsilon}$ and $\vartheta^{\varepsilon} ( 0 ) =0$, and where $L= \partial_{t} + ( - \Delta )^{\sigma}$. Let for $\varepsilon >0$ the function $u^{\varepsilon}$ be the unique global solution to the Cauchy problem
  \[
     L u^{\varepsilon} =G ( u^{\varepsilon} ) \partial_{x} u^{\varepsilon} + \xi^{\varepsilon} , \hspace{2em} u^{\varepsilon} ( 0 ) =u_{0} ,
  \]
  where $u_{0} \in \CC^{\alpha}$. Then there exists $u \in C_{T} \CC^{\alpha}$ such that $u^{\varepsilon} \rightarrow u$ in $C_{T} \CC^\alpha$. The limit $u$ depends only on $( u_{0} , \vartheta ,\eta )$, and not on the approximating family $( \vartheta^{\varepsilon} , (\vartheta^{\varepsilon} \reso   \partial_{x} \vartheta^{\varepsilon} ) )$.
\end{theorem}

\begin{remark}
  As for \textsc{rde}s, the limit $u$ of the regularized solutions $u^{\varepsilon}$ actually solves the equation
  \[
     L u=G ( u ) \partial_{x} u+ \xi , \hspace{2em} u ( 0 ) =u_{0}
  \]
  in the weak sense as long as we interpret the product $G ( u ) \partial_{x} u$ correctly. According to Remark~\ref{rmk:product smooth perturbation}, it is not important that $\vartheta(0) = 0$, and we could consider any other initial condition in $\vartheta(0) \in \CC^\alpha$ to obtain the same solution $u$. However, the right choice of $\vartheta(0)$ may facilitate the proof of existence and uniqueness of paracontrolled solutions.
\end{remark}

\begin{remark}
   Of course, the solution $u$ to the fractional Burgers type equation also depends continuously on the initial condition $u_0$.
\end{remark}

\subsection{Construction of the area}

It remains to show that if $\xi$ is a space-time white noise, then the solution $\vartheta$ to $L \vartheta = \xi$, $\vartheta(0) = 0$, is in $C_T\CC^\alpha$ for all $\alpha < \sigma - 1/2$, and that the area $\vartheta \reso \partial_x \vartheta$ is in $C_T \CC^{2 \alpha - 1}$. Some general results on the existence of the area for Gaussian processes indexed by a one-dimensional spatial variable are shown in~\cite{Friz2012}. However, in the present setting it is relatively straightforward to construct the area ``by hand'', using Fourier analytic methods.
%

In this section, we use $\CF$ to denote the spatial Fourier transform, i.e. $\CF u(t,\cdot) (k) = \int_\mathbb{T} e^{-\imath k x} u(t,x) \mathd x$. Recall that $\CF \xi$ is a complex valued, centered Gaussian space-time distribution, whose covariance is formally given by
\[
   \E[\CF \xi^i (t,\cdot)(k) \overline{\CF \xi^{i'} (t',\cdot)(k')}] = 2\pi \mathbf{1}_{i=i'} \mathbf{1}_{k=k'} \delta(t - t')
\]
for $i,i'\in \{1, \dots, n\}$, $t, t' \in [0,T]$, $k, k' \in \Z$, where $\delta$ denotes the Dirac delta. If $(P_t)_{t\geqslant 0} = (e^{-t |\cdot|^{2\sigma}}(\mathD))_{t \geqslant 0}$ denotes the semigroup generated by $-(-\Delta)^\sigma$, then $\vartheta(t,x) = \int_{0}^t (P_{t-s} \xi)(x)\mathd s$, $t \in [0,T]$, from where a straightforward calculation yields the following result:

\begin{lemma}\label{l:heat covariance}
   The spatial Fourier transform $\CF \vartheta$ of $\vartheta$ is a complex-valued Gaussian process with zero mean and covariance
   \begin{align*}
      &\E[\CF{\vartheta}^i(t,\cdot)(k) \overline{\CF{\vartheta}^{i'}(t',\cdot)(k')}] \\
      &\hspace{50pt}= \begin{cases}
                                                                                                       2 \pi \mathbf{1}_{i=i'} \mathbf{1}_{k=k'} (e^{-|t'-t| | k |^{2 \sigma}} - e^{-(t+t')|k|^{2\sigma} })/(2|k|^{2 \sigma}), & k \neq 0,\\
                                                                                                       2\pi \mathbf{1}_{i=i'} \mathbf{1}_{k=k'} t \wedge t', & k = 0,
                                                                                                   \end{cases}
   \end{align*}
   for $i,i' \in \{1, \dots, n\}$, $k, k' \in \mathbb{Z}$, and $t,t' \in [0,T]$. Thus, $\E[\CF{\vartheta}^i_{s,t}(0) \overline{\CF{\vartheta}^{i}_{s,t}(k')}] = 2\pi \mathbf{1}_{i=i'} \mathbf{1}_{k'=0} |t-s|$, and for $k\neq 0$
   \begin{align*}
      &\E[\CF{\vartheta}^i_{s,t}(k) \overline{\CF{\vartheta}^{i'}_{s,t}(k')}] \\
      &\hspace{45pt} = \pi \mathbf{1}_{i=i'} \mathbf{1}_{k=k'} \frac{2 - e^{-2 s | k |^{2 \sigma}}  - e^{-2 t | k |^{2 \sigma}} - 2 e^{- 2 |t-s| |k|^{2\sigma}} + 2 e^{-(s+t)|k|^{2\sigma}}}{|k|^{2 \sigma}},
   \end{align*}
   where we write $\CF{\vartheta}^i_{s,t}(k) = \CF{\vartheta}^i(t,\cdot)(k) - \CF{\vartheta}^i(s,\cdot)(k)$ for all $0 \leqslant s < t \leqslant T$. In particular,
   \begin{equation}
      |\E[\CF{\vartheta}^i_{s,t}(k) \overline{\CF{\vartheta}^{i}_{s,t}(k)}]| \lesssim |t-s|^{\delta} |k|^{-2\sigma(1 - \delta)}
   \end{equation}
   for all $\delta \in [0,1]$ and all $k \neq 0$.
\end{lemma}


Our first concern is to study the H\"older-Besov regularity of the process $\vartheta$.

\begin{lemma}\label{l:stationary besov}
   For any $\alpha < \sigma - 1 / 2$ and any $p\geqslant 1$, the process $\vartheta$ satisfies
   \[
      \E[ \| \vartheta \|_{C_T \CC^{\alpha}(\mathbb{T})}^p] < \infty.
   \]
\end{lemma}

\begin{proof}
   Let $s, t \in [ 0, T ]$ and $\ell \geqslant -1$. 
   Using Gaussian hypercontractivity (\cite{Janson1997}, Theorem 3.50), we obtain for $p \geqslant 1$ that
   \begin{equation}\label{e:burgers hypercontractivity}
      \E[\Vert  \Delta_{\ell} \vartheta_{s,t} \Vert_{L^{2 p} (\mathbb{T})}^{2 p}] \lesssim_p \Vert  \E [| \Delta_{\ell} \vartheta_{s,t}( x ) |^2] \Vert_{L^p_x ( \mathbb{T})}^p.
   \end{equation}
   If $\ell \geqslant 0$, then Fourier inversion and Lemma~\ref{l:heat covariance} imply
   \begin{align*}
      \E[| \Delta_{\ell} \vartheta_{s,t} (x)|^2] & = (2 \pi)^{-2} \sum_{k, k' \in \Z} \rho_{\ell} (k) \rho_{\ell}(k') e^{\imath (k - k') x } \E[\CF{\vartheta}_{s,t} ( k ) \overline{\CF{\vartheta}_{s,t} ( k')}] \\
      & \lesssim \sum_{k \in \Z} \rho^2_\ell(k) |t-s|^{\delta} |k|^{2\sigma(\delta - 1)} \lesssim |t-s|^{\delta} \sum_{k \in \tmop{supp} (\rho_{\ell})} |k|^{2\sigma(\delta - 1)} \\
      & \lesssim |t-s|^{\delta} 2^{\ell(1 - 2 \sigma (1-\delta))}
   \end{align*}
   for all $\delta \in (0,1]$. The case $\ell = -1$ can be treated using essentially the same arguments, except that then we need to distinguish the cases $k= 0$ and $k\neq 0$, where $k$ is the argument in the Fourier transform.
%
   Hence, we obtain from \eqref{e:burgers hypercontractivity}
   \begin{align*}
      \E[\lVert \vartheta(t, \cdot) - \vartheta(s,\cdot)\rVert_{B^\alpha_{2p,2p}(\mathbb{T})}^{2p} ] & \lesssim \sum_{\ell \geqslant -1} 2^{\ell \alpha 2p} \E [\lVert \Delta_\ell \vartheta_{s,t}\rVert_{L^{2p}(\mathbb{T})}^{2p} ] \\
      & \lesssim \sum_{\ell \geqslant -1} 2^{\ell \alpha 2p} \left(|t-s|^{\delta} 2^{2\ell(1/2 - \sigma (1-\delta))}\right)^p
   \end{align*}
   for any $\alpha \in \R$ and any $p \geqslant 1$. For $\alpha < \sigma - 1/2$ there exists $\delta \in (0,1]$ small enough so that the series converges. Since we can choose $p$ arbitrarily large, Kolmogorov's continuity criterion 
   implies that $\vartheta$ has a continuous version with $\E[ \|\vartheta \|_{C_T B^{\alpha}_{2 p, 2 p}(\mathbb{T})}^{2p}]< \infty$ for all $\alpha < \sigma - 1/2$. Now we use again that $p$ can be chosen arbitrarily large, so that the Besov embedding theorem, Lemma~\ref{l:besov embedding torus}, shows that this continuous version takes its values in $C_T \CC^{\alpha}(\mathbb{T})$ for all $\alpha < \sigma - 1/2$.
\end{proof}

%

Next, we construct the area $\vartheta \reso \partial_x \vartheta$.

\begin{lemma}\label{l:burgers area}
   Define
   \[
      \vartheta \reso \partial_x \vartheta = (\vartheta^k \reso \partial_x \vartheta^\ell)_{1\leqslant k, \ell \leqslant n} = \bigg(\sum_{|i-j| \leqslant 1} \Delta \vartheta^i \Delta_j \partial_x \vartheta^j \bigg)_{1\leqslant k, \ell \leqslant n}.
   \]
   Then almost surely $\vartheta \reso \partial_x \vartheta \in C_T\CC^{2\alpha-1}(\mathbb{T}; \R^{n \times n})$ for all $\alpha < \sigma - 1/ 2$. Moreover, if $\psi \in \CS$ is such that $\int \psi(x) \mathd x = 1$ and $\vartheta^\varepsilon = \psi^\varepsilon \ast \vartheta$, where $\psi^\varepsilon = \varepsilon^{-1} \psi(\varepsilon^{-1} \cdot)$, then we have for all $p \geqslant 1$ that
   \begin{equation}\label{eq:burgers area lp convergence}
      \lim_{\varepsilon \rightarrow 0} \E[\| \vartheta^\varepsilon \reso \partial_x \vartheta^\varepsilon - \vartheta \reso \partial_x \vartheta \|_{C_T \CC^{2\alpha - 1}}^p] = 0.
   \end{equation}
\end{lemma}

\begin{proof}
   Without loss of generality we can argue for $\vartheta^1 \reso \partial_x \vartheta^2$. The case $\vartheta^1 \reso \partial_x \vartheta^1$ is easy, because Leibniz' rule yields $\vartheta^1 \reso \partial_x \vartheta^1 = \frac{1}{2} \partial_x(\vartheta^1 \reso \vartheta^1)$.
  
   Let $\ell \in \N$. Note that if $i$ is smaller than $\ell - N$ for a suitable $N$, and if $| i - j | \leqslant 1$, then $\Delta_{\ell}(\Delta_i f \Delta_j g) = 0$ for all $f, g \in \CS'$. Hence, the projection of $\vartheta^1\reso \partial_x\vartheta^2$ onto the $\ell$--th dyadic Fourier block is given by
   \begin{align*}
      \Delta_{\ell} (\vartheta^1\reso \partial_x \vartheta^2 ) = \sum_{| i- j | \leqslant 1} \Delta_{\ell}(\Delta_i \vartheta^1 \Delta_j\partial_x \vartheta^2) = \sum_{|i-j|\leqslant 1} \mathbf{1}_{\ell \lesssim i} \Delta_{\ell} (\Delta_i \vartheta^1 \Delta_j \partial_x \vartheta^2).
   \end{align*}
   To avoid case distinctions, we only argue for $\ell \geqslant N$, so that we can always assume $i, j  \geqslant 0$. The case $\ell < N$ can be handled using essentially the same arguments.
      
   We use the equivalence of moments for random variables living in an inhomogeneous Gaussian chaos of fixed degree (\cite{Janson1997}, Theorem~3.50) to obtain
   \begin{gather} \nonumber
      \E [\Vert  (\Delta_{\ell} ( \vartheta^1\reso \partial_x \vartheta^2 - \vartheta^{1,\varepsilon} \reso \partial_x \vartheta^{2,\varepsilon}) )_{s,t} \Vert_{L^{2 p}(\mathbb{T})}^{2 p}] \\ \label{eq:burgers area hypercontractivity}
      \lesssim \biggl\Vert  \E\biggl[\biggl| \sum_{|i-j| \leqslant 1} \mathbf{1}_{\ell \lesssim i} ( \Delta_{\ell} ( \Delta_i \vartheta^1 \Delta_j \partial_x \vartheta^2 -  \Delta_i \vartheta^{1,\varepsilon} \Delta_j \partial_x \vartheta^{2,\varepsilon} )(x))_{s,t}\biggr|^2\biggr] \biggr\Vert_{L^p_x(\mathbb{T})}^p,
   \end{gather}
   where we write $\vartheta^{1, \varepsilon} = \psi^\varepsilon \ast \vartheta$ and similarly for $\vartheta^{2, \varepsilon}$. 
   
   Let us start by estimating
   \begin{align}\label{e:burgers area pr1}
      &\E\biggl[\biggl| \sum_{|i-j| \leqslant 1} \mathbf{1}_{\ell \lesssim i} \Delta_{\ell} ( \Delta_i \vartheta^1(t,\cdot) \Delta_j \partial_x \vartheta^2_{s,t} - \Delta_i \vartheta^{1,\varepsilon} (t,\cdot) \Delta_j \partial_x \vartheta^{2,\varepsilon}_{s,t})(x)\biggr|^2\biggr] \\ \nonumber
      &\hspace{20pt} = \sum_{| i - j | \leqslant 1} \sum_{| i' - j' | \leqslant 1} \mathbf{1}_{\ell \lesssim i} \mathbf{1}_{\ell \lesssim i'} \E\bigg[\Delta_{\ell} (\Delta_i \vartheta^1(t,\cdot) \Delta_j \partial_x \vartheta^2_{s,t} - \Delta_i \vartheta^{1,\varepsilon} (t,\cdot) \Delta_j \partial_x \vartheta^{2,\varepsilon}_{s,t})(x)\\ \nonumber
      &\hspace{110pt} \times \overline{\Delta_{\ell} ( \Delta_{i'} \vartheta^1(t,\cdot) \Delta_{j'} \partial_x \vartheta^2_{s,t} - \Delta_{i'} \vartheta^{1,\varepsilon} (t,\cdot) \Delta_{j'} \partial_x \vartheta^{2,\varepsilon}_{s,t})(x)}\bigg].
   \end{align}
   Taking the infinite sums outside of the expectation can be justified a posteriori, because for every finite partial sum we will obtain a bound on the $L^2$--norm below, which does not depend on the number of terms that we sum up. The Gaussian hypercontractivity \eqref{eq:burgers area hypercontractivity} then provides a uniform $L^p$--bound for all $p \geqslant 2$, which implies that the squares of the partial sums are uniformly integrable, and thus allows us to exchange summation and expectation.
   
   Recall that $\CF(u v)(k) = (2\pi)^{-1} \sum_{k'} \CF u(k') \CF v(k -k')$, and $\CF(\partial_x u)(k) = \imath k \CF(u)(k)$, and therefore 
   \begin{align*}
      &\Delta_{\ell} (\Delta_i \vartheta^1(t,\cdot) \Delta_j \partial_x \vartheta^2_{s,t} - \Delta_i \vartheta^{1,\varepsilon} (t,\cdot) \Delta_j \partial_x \vartheta^{2,\varepsilon}_{s,t})(x) \\
      &\hspace{15pt} = (2\pi)^{-1} \sum_{k \in \Z} \rho_\ell(k) e^{\imath \langle k, x \rangle} \CF(\Delta_i \vartheta^1(t,\cdot) \Delta_j \partial_x \vartheta^2_{s,t} - \Delta_i \vartheta^{1,\varepsilon} (t,\cdot) \Delta_j \partial_x \vartheta^{2,\varepsilon}_{s,t})(k) \\
      &\hspace{15pt} = (2\pi)^{-2} \sum_{k,k' \in \Z} \rho_\ell(k) e^{\imath \langle k, x \rangle} \rho_i(k') \rho_j(k-k') \imath (k-k') \CF \vartheta^1(t,\cdot)(k') \CF\vartheta^2_{s,t}(k-k') \\
      &\hspace{70pt} \times (1 - \CF \psi(\varepsilon k') \CF \psi(\varepsilon(k-k'))).
   \end{align*}
   From this expression it is clear that if we can show $\E[\| \vartheta^\varepsilon \reso \partial_x \vartheta^\varepsilon \|_{C_T \CC^{2\alpha - 1}}^p] < \infty$, then the convergence result in~\eqref{eq:burgers area lp convergence} will follow by dominated convergence, because $\CF \psi$ is bounded and $\CF \psi(0) = 1$ by assumption.
   
   Using the covariance of $\CF \vartheta$ that we calculated in Lemma~\ref{l:heat covariance}, we obtain
   \begin{align*}
      &\E\biggl[\biggl|\sum_{|i-j| \leqslant 1} \mathbf{1}_{\ell \lesssim i} \Delta_{\ell} ( \Delta_i \vartheta^1(t,\cdot) \Delta_j \partial_x \vartheta^2_{s,t})(x)\biggr|^2\biggr] \\
      &\hspace{40pt} \lesssim \sum_{| i - j | \leqslant 1} \sum_{| i' - j' | \leqslant 1} \mathbf{1}_{\ell \lesssim i} \mathbf{1}_{\ell \lesssim i'} \sum_{k, k' \in \mathbb{Z}^d} \rho_{\ell}^2 (k+k') \rho_i ( k ) \rho_{i'} ( k ) \rho_j (k') \rho_{j'}(k')\\
      &\hspace{170pt} \times \frac{1 - e^{-2t|k|^{2\sigma}}}{2 |k|^{2 \sigma}} |k'|^2 |t-s|^{\delta} |k'|^{-2\sigma(1 - \delta)}\\
      &\hspace{40pt} \lesssim \sum_{| i - j | \leqslant 1} \mathbf{1}_{\ell \lesssim i} \sum_{k \in \tmop{supp} ( \rho_i ), k' \in \tmop{supp} ( \rho_j )} \rho_{\ell}^2 ( k + k' ) 2^{2i(1 - 2\sigma + \sigma \delta)} |t-s|^{\delta} \\
      &\hspace{40pt} \lesssim \sum_{i \gtrsim \ell} 2^\ell 2^{2i(1 + 1/2 - 2\sigma + \sigma\delta)} |t-s|^{\delta} 
   \end{align*}
   for all $\delta \in [0,1]$. Since $\sigma > 5/6$, there exists $\delta >0$ small enough so that the sum is finite, and we obtain
   \[
      \E\biggl[\biggl|\sum_{|i-j| \leqslant 1} \mathbf{1}_{\ell \lesssim i} \Delta_{\ell} ( \Delta_i \vartheta^1(t,\cdot) \Delta_j \partial_x \vartheta^2_{s,t})(x)\biggr|^2\biggr]  \lesssim 2^{2i(2 - 2\sigma + \sigma \delta)}  |t-s|^{\delta},
   \]
   and by the same arguments
   \[
      \E\biggl[\biggl|\sum_{|i-j| \leqslant 1} \mathbf{1}_{\ell \lesssim i} \Delta_{\ell} ( \Delta_i \vartheta^1_{s,t} \Delta_j \partial_x \vartheta^2(s,\cdot))(x)\biggr|^2\biggr]  \lesssim 2^{2i(2 - 2\sigma + \sigma \delta)}  |t-s|^{\delta}.
   \]   
%
%
   Noting that
   \begin{align*}
      &\Delta_i \vartheta^1(t,\cdot) \Delta_j \partial_x \vartheta^2(t,\cdot) - \Delta_i \vartheta^1(s,\cdot) \Delta_j \partial_x \vartheta^2(s,\cdot) \\
      &\hspace{80pt} = \Delta_i \vartheta^1(t,\cdot) \Delta_j \partial_x \vartheta^2_{s,t} + \Delta_i \vartheta^1_{s,t} \Delta_j \partial_x \vartheta^2(s,\cdot),
   \end{align*}
   we get for sufficiently small $\delta > 0$ and for arbitrarily large $p\geqslant 1$ that
   \begin{align*}
      \E [\Vert  \Delta_{\ell} (\vartheta^1 \reso \partial_x \vartheta^2)_{s,t} \Vert_{L^{2 p}(\mathbb{T} )}^{2 p}] \lesssim 2^{-2\ell(2\sigma - 2 - \sigma \delta ) p} |t - s|^{\delta p}.
   \end{align*}
   From this point on we use the same arguments as in the proof of Lemma \ref{l:stationary besov} to obtain the required $L^p$-bound for $\| \vartheta^1\reso \partial_x \vartheta^2\|_{C_T \CC^{2\alpha-1}}$ with $\alpha < \sigma - 1 / 2$.
\end{proof}

%

Now Lemma~\ref{l:burgers area} and Theorem~\ref{theorem:burgers existence uniqueness} give us the existence and uniqueness of solutions to the fractional Burgers type equation driven by space-time white noise:

\begin{corollary}
  Let $\sigma > 5/6$, $\alpha \in ( 1/3, \sigma -1/2 )$, $T>0$, $G \in C^3_b$, $u_0 \in \CC^\alpha(\mathbb{T})$, $L = \partial_t + (-\Delta)^\sigma$, and let $\xi$ be a space-time white noise on $[0,T] \times \mathbb{T}$ with values in $\R^n$. Then there exists a unique solution $u$ to
  \[
     L u = G(u) \partial_x u + \xi, \qquad u(0) = u_0,
  \]
  in the following sense: If $\psi \in \CS$ with $\int \psi \mathd t=1$ and if for $\varepsilon > 0$ the function $u^\varepsilon$ solves
  \[
     L u^\varepsilon = G(u^\varepsilon) \partial_x u^\varepsilon + \xi^\varepsilon, \qquad u(0) = u_0,
  \]
  where $\xi^\varepsilon = \varepsilon^{-1} \psi(\varepsilon \cdot) \ast \xi$, then $u^\varepsilon$ converges in probability in $C_T \CC^\alpha$ to $u$.
%
%
\end{corollary}

\begin{remark}
   There is no problem in considering the equation on $\mathbb{T}^d$ rather than on $\mathbb{T}$, 
   and the analysis works exactly as in the one-dimensional case. The proof of Lemma~\ref{l:stationary besov} shows that if $\xi$ is a space-time white noise on $[0,T] \times \mathbb{T}^d$, then the solution $\vartheta$ to $L \vartheta = \xi$, $\vartheta(0) = 0$, will be in $C_T\CC^\alpha(\mathbb{T}^d)$ for every $\alpha < \sigma - d/2$. So as long as $\sigma - d/2 > 1/3$, we can solve the Burgers equation on $\mathbb{T}^d$. For the existence of the area $\vartheta \reso \partial_x \vartheta$ we need the additional condition $2 \sigma - d/2 - 1 > 0$; see~\cite{Perkowski2014Thesis}, Lemma~5.4.3. But if $\sigma - d/2 > 1/3$, then this is always satisfied. 
\end{remark}

\section{A generalized parabolic Anderson model}\label{sec:PAM}

Consider now the following PDE on $[ 0,T ] \times \mathbb{T}^{2}$ for some fixed $T>0$:
\begin{equation}\label{eq:pam}
  L u=F ( u ) \xi, \qquad u(0) = u_0,
\end{equation}
where $L = \partial_{t} - \Delta$, the function $F$ is continuous from
$\mathbb{R}$ to $\mathbb{R}$, $\xi$ is a spatial white noise, and $u_0 \in \CC^\alpha$ for suitable $\alpha \in \R$.

The linear case $F(u) = u$ is the parabolic Anderson model, the discrete version of which has been intensely studied during the past decades~\cite{Carmona1994,Koenig2015}. The continuous version in $d=2$ was solved by Hu~\cite{Hu2002} with the help of Wick products and explicit chaos expansions; however, the renormalization performed by taking the Wick product is not very transparent, and it does not seem easy to show that Hu's solution is the universal continuum limit of the discrete parabolic Anderson model. Here we will carry out a simple renormalization that easily translates to discrete models, and indeed one can show that our solution is the universal continuum limit of the 2d lattice Anderson model with small potential~\cite{Chouk2015}.

The general case seems not to have been studied before, see however~\cite{Hairer2013a} for an alternative but equivalent approach to the same equation. There are several reasons for studying such a nonlinear generalization. First of all it is a generic equation for the nonlinear evolution of particles in a random stationary medium. Moreover, equation~\eqref{eq:pam} is formally very similar to the rough differential equation~{\eqref{eq:rde}} and thus a natural benchmark problem. And if $u$ solves~\eqref{eq:pam} with $F(u) = u$ and if we set $v = \varphi (u)$ for
some invertible $\varphi : \mathbb{R} \rightarrow \mathbb{R}$ such that
$\varphi' > 0$, then formally
\[
   L v = \varphi' (u) L u - \varphi'' (u) | \partial_x u |^2 = \varphi' (u) i \xi - \varphi''
   (u) (\varphi' (u))^{- 2} | \partial_x u |^2
\]
and thus $v$ satisfies the PDE
\[
   L v = F_1 (v) \xi + F_2(v) |\partial_x v|^2,
\]
where $F_1 (x) = \varphi' (\varphi^{- 1} (x)) \varphi^{- 1} (x)$ and $F_2 (x)
= - \varphi'' (\varphi^{- 1} (x)) (\varphi' (\varphi^{- 1} (x)))^{- 2}$. In
the situation we are interested in, the second term in the right hand side is
easier to treat than the first term, so we will drop it
and concentrate on the case $F_2 = 0$.

The regularity of the spatial white noise $\eta$ on $\mathbb{T}^{d}$ is $\eta \in \CC^{-d/2- \varepsilon}$ for all $\varepsilon >0$. Since we are
in dimension $d=2$, we have $\xi \in \CC^{-1- \varepsilon}$. The
Laplacian increases the regularity by 2, so we expect that for fixed $t>0$ we
have $u ( t ) \in \CC^{1- \varepsilon}$, and therefore the product
$F ( u ) \xi$ is ill-defined.

However, let us assume that $\xi \in \CC^{\alpha -2} ( \mathbb{T}^{2} )$ for
some $2/3< \alpha <1$. Since $\xi$ does not depend on time, there exists $\vartheta \in \CC^\alpha$ such that $-\Delta \vartheta = \xi - (2\pi)^{2} \CF \xi(0)$. More precisely, we can take
\begin{equation}\label{eq:pam theta def}
   \vartheta = \int_0^\infty P_t (\xi - (2\pi)^{2} \CF \xi(0)) \mathd t,
\end{equation}
where $(P_t)_{t \geqslant 0}$ denotes the heat flow. In particular we have $L\vartheta - \xi \in C^\infty(\mathbb{T}^2)$ and $\|\vartheta\|_\alpha \lesssim \|\xi\|_{\alpha-2}$. Consider the paracontrolled ansatz
\[
   u=F ( u ) \lpara \vartheta +u^{\sharp}
\]
with $u^{\sharp} \in C_{T} \CC^{2 \alpha}$, and where as in Section~\ref{sec:burgers} the paraproduct $\lpara$ is only acting on the spatial
variables. If $u$ is of this form, then Lemma~\ref{lemma:para-taylor} and
Lemma~\ref{lemma:commutator} imply that
\begin{align*}
   F ( u ) \xi & =F ( u ) \lpara \xi +F ( u ) \rpara \xi +F' ( u ) F ( u )  (
   \vartheta \reso \xi ) +F' ( u ) C ( F ( u ) , \vartheta , \xi ) \\
   &\quad +F' ( u )  ( u^{\sharp} \reso \xi ) + \Pi_{F} ( u, \xi )
\end{align*}
is well defined provided that $( \vartheta \reso   \xi ) \in \CC^{2 \alpha -2}$. Moreover, the algebraic rules for $\partial_{t}$ and $\Delta$
acting on products imply that
\[
   L u= ( L F ( u ) ) \lpara \vartheta +F ( u ) \lpara L  \vartheta - 2 \mathD_{x} F ( u ) \lpara \mathD_{x} \vartheta +L u^{\sharp},
\]
and thus we find the following equation for $u^{\sharp}$:
\begin{align*}
   L u^{\sharp} & = 2 \mathD_{x} F ( u ) \lpara \mathD_{x} \vartheta - ( L F ( u ) ) \lpara \vartheta +F ( u ) \rpara \xi + F' ( u ) F ( u )  ( \vartheta \reso \xi ) \\
    &\quad + F(u)\lpara (\xi - L\vartheta) +F' ( u ) C ( F ( u ) , \vartheta , \xi ) +F' ( u )  ( u^{\sharp} \reso \xi) + \Pi_{F} ( u, \xi ).
 \end{align*}
We would like all the terms on the right hand side to be in $C_{T} \CC^{2 \alpha -2}$. However, it is not easy to estimate $( LF ( u ) ) \lpara \vartheta$ in $C_{T} \CC^{\beta}$ for any $\beta \in \mathbb{R}$: the term $\Delta F ( u )$ can be controlled in $\CC^{\alpha-2}$, but there are no straightforward estimates available for the time derivative $\partial_{t} F ( u )$ appearing in $L F ( u )$. Indeed, it would
be more convenient to treat the generalized parabolic Anderson model in a
space-time parabolic Besov space adapted to the operator $L$ and to use the
natural paraproduct associated to this space. An alternative strategy would be
to stick with the simpler space $C_{T} \CC^{\alpha -2}$ and to observe that
\[
   L F ( u ) =F' ( u ) L u - F'' ( u ) ( \mathD_{x} u )^{2} =F' ( u ) F ( u ) \xi - F'' ( u ) ( \mathD_{x} u )^{2},
\]
and that the terms on the right hand side can be analyzed using the
paracontrolled ansatz. Since this strategy seems to require a lot of
regularity from $F$, we do not pursue it further.

Instead, we keep working on $C_{T} \CC^{\alpha -2}$, but we modify the paraproduct
appearing in the paracontrolled ansatz. Let $\varphi \colon \mathbb{R} \rightarrow \mathbb{R}_{+}$ be a positive smooth function with compact support and total
mass $1$, and for all $i \geqslant -1$ define the operator $Q_{i} \colon C_{T} \CC^{\beta} \rightarrow C_{T} \CC^{\beta}$ by
\[
   Q_{i} f ( t ) = \int_{\mathbb{R}} 2^{2i} \varphi ( 2^{2i} ( t-s ) ) f ( (s \wedge T ) \vee 0 ) \mathd s.
\]
For $Q_{i}$ we have the following standard estimates, which we leave to the reader to prove:
\begin{gather}\label{eq:Q-bounds}
   \| Q_{i} f ( t ) \|_{L^{\infty}} \leqslant \| f \|_{C_{T} L^{\infty}} , \hspace{1em} \| \partial_{t} Q_{i} f ( t ) \|_{L^{\infty}} \leqslant 2^{( 2-2 \gamma ) i} \| f \|_{C_{T}^{\gamma} L^{\infty}}, \\ \nonumber
   \| ( Q_{i} f-f ) ( t ) \|_{L^{\infty}} \leqslant 2^{-2 \gamma i} \| f \|_{C_{T}^{\gamma} L^{\infty}}
\end{gather}
for all $t \in [ 0,T ]$ and all $\gamma \in ( 0,1 )$; for the second estimate
we use that $\int \varphi' ( t ) \mathd t=0$, and for the third estimate we
use that $\varphi$ has total mass 1.
With the help of $Q_{i}$, let us define a modified paraproduct by setting
\begin{equation}\label{eq:modpara}
   f \mpara g= \sum_{i} ( S_{i-1} Q_{i} f ) \Delta_{i} g
\end{equation}
for $f,g \in C_{T} \CS'$. While we were not able to find any references, we think it quite likely that such a modified paraproduct appeared previously in the PDE literature. It is easy to show that for the modified paraproduct we have
essentially the same estimates as for the pointwise paraproduct $f \lpara g$,
only that we have to bound $f$ uniformly in time; for example
\[
   \| ( f \mpara g ) ( t ) \|_{\alpha} \lesssim \| f \|_{C_{T} L^{\infty}} \| g ( t ) \|_{\alpha} .
\]
for all $t \ [0,T]$. For us, the following two estimates are the most useful properties of $\mpara$.

\begin{lemma}\label{lemma:mod-para-est}
  Let $T>0$, $\alpha \in (0,1)$, $\beta \in \R$, and let $u \in C_{T} \CC^{\alpha} \cap C^{\alpha/2}_T L^\infty$ and $v \in C_T \CC^\beta$. Then
  \begin{equation}\label{eq:mod-para-1}
     \| L(u \mpara v ) - u\mpara (Lv) \|_{C_T\CC^{\alpha+\beta -2}} \lesssim ( \| u \|_{C_{T}^{\alpha /2} L^{\infty}} + \| u \|_{C_{T} \CC^\alpha} ) \| v \|_{C_T\CC^\beta},
  \end{equation}
  as well as
  \begin{equation}\label{eq:paraproduct difference}
     \|u \lpara v - u \mpara v\|_{C_T\CC^{\alpha+\beta}} \lesssim \| u \|_{C_{T}^{\alpha /2} L^{\infty}} \| v \|_{C_T\CC^\beta}.
  \end{equation}
\end{lemma}

\begin{proof}
  For~\eqref{eq:mod-para-1}, observe that $L(u \mpara v ) - u\mpara (Lv) = (Lu) \mpara v - 2 \mathD_x u \mpara \mathD_x v$. The second term on the right hand side is easy to estimate. The first term is given by
  \[
     ( L u ) \mpara v = \sum_{i} ( S_{i-1} Q_{i} L u  )  \Delta_{i} v = \sum_{i} ( L S_{i-1} Q_{i} u ) \Delta_{i} v.
  \]
  Observe that, as for the standard paraproduct, $( L S_{i-1} Q_{i}  F ( u )) \Delta_{i} v$ has a spatial Fourier transform localized in an
  annulus $2^{i} \CA$, so that according to Lemma~\ref{lem: Besov regularity of series} it will be sufficient to control its $C_T L^\infty$ norm. But
  \begin{align*}
     \| L S_{i-1} Q_{i} u \|_{C_TL^{\infty}} & \leqslant \| \partial_{t} Q_{i} S_{i-1}  u  \|_{C_TL^{\infty}} + \| Q_{i} \Delta  S_{i-1}  u  \|_{C_TL^{\infty}} \\
     & \lesssim 2^{- ( \alpha -2 ) i} \big( \| S_{i-1}  u  \|_{C_{T}^{\alpha /2} L^{\infty}} + \| u \|_{C_{T} \CC^\alpha}\big),
  \end{align*}
  where we used the bounds~{\eqref{eq:Q-bounds}}. It is easy to see that $ \| S_{i-1}  u  \|_{C_{T}^{\alpha /2}} \lesssim \| u \|_{C_T^{\alpha/2}}$, and therefore we obtain~\eqref{eq:mod-para-1}.
  
  As for~\eqref{eq:paraproduct difference}, we have
  \[
     u \lpara v - u \mpara v = \sum_{i} ( Q_{i} S_{i-1}  u -S_{i-1}  u ) \Delta_{i} v,
  \]
  and again it will be sufficient to control the $C_T L^\infty$ norm of each term of the series. But using once more~{\eqref{eq:Q-bounds}}, we obtain
  \begin{align*}
     \| (Q_{i} S_{i-1}  u -S_{i-1}  u ) \Delta_{i} v \|_{C_T L^{\infty}} & \lesssim 2^{-i \alpha} \| S_{i-1} u \|_{C_{T}^{\alpha /2} L^{\infty}} \| \Delta_{i} v \|_{C_T L^{\infty}} \\
     & \lesssim 2^{-i(\alpha+\beta)} \| u \|_{C_{T}^{\alpha /2} L^{\infty}} \| v \|_{C_T\CC^{\beta}},
  \end{align*}
  and the result is proved.
\end{proof}

Letting
\begin{equation}\label{eq:mod-paracontrolled}
  u=F ( u ) \mpara \vartheta +u^{\sharp}
\end{equation}
and redoing the same computation as above, we end up with
\begin{align}\label{eq:phi sharp def pam} \nonumber
   L u^{\sharp} = \Phi^{\sharp} & = - [ L (F ( u )  \mpara \vartheta) - F(u) \mpara L\vartheta]  + [ F ( u ) \lpara \xi -F ( u ) \mpara L\vartheta ] \\
   &\quad + F ( u ) \rpara \xi + F(u) \reso \xi. 
\end{align}
Lemma~\ref{lemma:mod-para-est} (and the fact that $L\vartheta - \xi \in C^\infty(\mathbb{T}^2)$) takes care of the first two terms on the right hand side. The term $F(u) \rpara \xi$ can be controlled using the paraproduct estimates, so that it remains to control the resonant product $F(u) \reso \xi$. In principle, this can be achieved by combining the decomposition described above with~\eqref{eq:paraproduct difference}, which enables us to switch between the two paraproducts~$\mpara$ and~$\lpara$. However, in that way we pick up a superlinear estimate from Lemma~\ref{lemma:paralinearization}. By being slightly more careful, we can get an estimate which depends linearly on $\|u^\sharp(t)\|_{\alpha+\beta}$ and is quadratic only in $\|u\|_{C_T L^\infty}^2$. This allows us to obtain a ``conditional global existence result'', which shows that there exists a paracontrolled solution up to the explosion time of the $L^\infty$ norm of $u$.

\begin{lemma}\label{lem:pam circ apriori}
   Let $\alpha \in (2/3, 1)$ and $\beta \in (0,\alpha]$ be such that $2\alpha + \beta > 2$. Let $T>0$, $\xi \in C(\mathbb{T}^2, \R)$, let $\vartheta$ be as defined in \eqref{eq:pam theta def}, $u \in C_T \CC^\alpha$, and let $F \in C^{1+\beta/\alpha}_{b}$. Define $u^\sharp = u - F(u) \mpara \vartheta$. Then
  \begin{equation}\label{eq:pam circ apriori 1}
     \|(F(u) \reso \xi) (t) \|_{\alpha+\beta-2} \lesssim C_F C_\xi \big( 1 + \|u\|_{C_T\CC^\alpha}^{1+\beta/\alpha} + \|u\|_{C^{\alpha/2}_TL^\infty} + \|u^\sharp (t)\|_{\alpha+\beta} \big),
  \end{equation}
  for all $t \in [0,T]$, where
  \begin{equation}\label{eq:pam Cxi CF}
     C_{\xi} = (1+\| \xi \|_{\alpha -2})^{2+\beta/\alpha} + \| \vartheta \reso   \xi \|_{C_{T} \CC^{2 \alpha -2}} \quad \text{and} \quad C_{F} = \| F \|_{C^{1+\beta/\alpha}_{b}}+\| F \|_{C^{1+\beta/\alpha}_{b}}^{2+\beta/\alpha}.
  \end{equation}
  If $F$ is in $C^3_b$, then
  \begin{align}\label{eq:pam circ apriori 2} \nonumber
     \|(F(u) \reso \xi) (t) \|_{\alpha+\beta-2} & \lesssim \|F\|_{C^3_b} (1+C_F C_\xi) (1+ \|u\|_{C_T L^\infty}^2) \\
     &\hspace{30pt} \times \big( 1 + \|u\|_{C_T\CC^\alpha} + \|u\|_{C^{\alpha/2}_TL^\infty} + \|u^\sharp (t)\|_{\alpha+\beta} \big).
  \end{align}
\end{lemma}

We pay attention to indicate that, for fixed $t \in [ 0,T ]$, the
estimate depends only on the $\CC^{\alpha + \beta}$ norm of
$u^{\sharp} ( t )$ and not on $\|u^\sharp\|_{C_T\CC^{\alpha+\beta}}$. This will come useful below when introducing the right
norm to control the contribution of the initial condition.

\begin{proof}
   We decompose
   \begin{align}\label{eq:pam circ apriori pr1} \nonumber
      F(u) \reso \xi & = (F(u) - F'(u) \lpara u) \reso \xi + (F'(u) \lpara u^\sharp) \reso \xi + C(F'(u),F(u)\mpara \vartheta, \xi) \\ \nonumber
      &\quad + F'(u) [(F(u) \mpara \vartheta - F(u) \lpara \vartheta)\reso \xi] + F'(u) C(F(u),\vartheta,\xi) \\
      &\quad + F'(u)F(u) (\vartheta \reso \xi),
   \end{align}
   from where we can use Lemma~\ref{lemma:mod-para-est} and the commutator estimate Lemma~\ref{lemma:commutator} to see that
   \begin{align*}
      &\| (F(u) \reso \xi - (F(u) - F'(u) \lpara u) \reso \xi)(t) \|_{\alpha+\beta-2} \\
      &\hspace{80pt} \lesssim C_F C_\xi \big(1 + \|u\|_{C_T\CC^\alpha} + \|u\|_{C^{\alpha/2}_T L^\infty} + \|u^\sharp(t) \|_{\alpha+\beta}\big).
   \end{align*}
   It remains to treat the first term on the right hand side of~\eqref{eq:pam circ apriori pr1}. Lemma~\ref{lemma:paralinearization} shows that
   \begin{align*}
      \| (F(u) - F'(u) \lpara u) \reso \xi \|_{C_T\CC^{2\alpha+\beta-2}} & \lesssim \| F(u) - F'(u) \lpara u \|_{C_T\CC^{\alpha+\beta}} \| \xi\|_{\alpha-2} \\
      & \lesssim \|F\|_{C^{1+\beta/\alpha}_b} ( 1 + \| u \|_{C_T \CC^\alpha}^{1+\beta/\alpha}) \|\xi\|_{\alpha-2},
   \end{align*}
   from where we get~\eqref{eq:pam circ apriori 1}.
   
   If $F$ is in $C^3_b$, then we apply a modified version of the paralinearization lemma, Lemma~\ref{lem:mod paralin}, to obtain
   \begin{align*}
      \| (F(u) - F'(u)\lpara u)(t) \|_{\alpha+\beta} & \lesssim \| F\|_{C^3_b} (1 + \| (F(u) \mpara \vartheta)(t) \|_\alpha^{1+\beta/\alpha}) \\
      &\hspace{30pt} \times (1 + \| u^\sharp(t) \|_{L^\infty}^2)(1 + \| u^\sharp(t) \|_{\alpha+\beta}),
   \end{align*}
   so that~\eqref{eq:pam circ apriori 2} follows.
\end{proof}

Let us summarize our observations so far.

\begin{lemma}
  Let $\alpha \in (2/3, 1)$, $\beta \in (2-2\alpha, \alpha]$, and $T>0$. Let $u_0 \in \CC^\alpha$, $\xi \in C(\mathbb{T}^2,\R)$, let $\vartheta$ be as defined in \eqref{eq:pam theta def}, and let $F \in C^{1+\beta/\alpha}_{b}$. Then
  $u$ solves the PDE
  \[
     L u=F ( u ) \xi , \hspace{2em} u ( 0 ) =u_{0} \in \CC^\alpha
  \]
  on $[ 0,T ]$ if and only if $u=F ( u ) \mpara \vartheta +u^{\sharp}$,
  where $u^{\sharp}$ solves
  \[
     L u^{\sharp} = \Phi^{\sharp} , \hspace{2em} u^{\sharp} ( 0 ) =u_{0} - ( F ( u ) \mpara \vartheta ) ( 0 )
  \]
  on $[ 0,T ]$, for $\Phi^{\sharp}$ as defined in {\eqref{eq:phi sharp def pam}}. Moreover, for all $t \in [ 0,T ]$ we have the estimate
  \begin{equation}\label{eq:phi sharp estimate pam}
     \| \Phi^{\sharp} ( t ) \|_{\alpha +\beta - 2} \lesssim C_F C_\xi \big( 1 + \|u\|_{C_T\CC^\alpha}^{1+\beta/\alpha} + \|u\|_{C^{\alpha/2}_TL^\infty} + \|u^\sharp (t)\|_{\alpha+\beta} \big),
  \end{equation}
  where $C_F$ and $C_\xi$ are as defined in~\eqref{eq:pam Cxi CF}. If $F$ is in $C^3_b$, then
  \begin{align}\label{eq:phi sharp estimate pam 2} \nonumber
     \| \Phi^{\sharp} ( t ) \|_{\alpha+\beta-2} & \lesssim \|F\|_{C^3_b} (1+C_F C_\xi) (1+ \|u\|_{C_T L^\infty}^2) \\
     &\hspace{30pt} \times \big( 1 + \|u\|_{C_T\CC^\alpha} + \|u\|_{C^{\alpha/2}_TL^\infty} + \|u^\sharp (t)\|_{\alpha+\beta} \big).
  \end{align}
\end{lemma}

Next, we would like to close the estimate~\eqref{eq:phi sharp estimate pam},
so that the right hand side only depends on $\Phi^{\sharp}$. In order to
estimate the terms depending on $u$, observe that $u = u^\sharp + F(u) \mpara \vartheta$ and thus
\[
   \|u\|_{C_T\CC^\alpha} + \|u\|_{C^{\alpha/2}_TL^\infty} \lesssim \|u^\sharp\|_{C_T\CC^\alpha} + \|u^\sharp\|_{C^{\alpha/2}_T L^\infty} + \| F(u) \mpara \vartheta \|_{C_T\CC^\alpha} +  \| F(u) \mpara \vartheta \|_{C^{\alpha/2}_TL^\infty}.
\]
To estimate the contribution of $F(u) \mpara \vartheta$, we observe that
\[
   \|L(F(u)\mpara \vartheta)\|_{C_T \CC^{\alpha-2}} \lesssim \|F(u)\|_{C_TL^\infty} \|\xi\|_{\alpha-2} \lesssim \|F\|_{L^\infty} \|\xi\|_{\alpha-2}
\]
(compare also the proof of Lemma~\ref{lemma:mod-para-est}). Thus, we can apply the heat flow estimates Lemma~\ref{lemma:heat flow smoothing}, Lemma~\ref{lemma:semigroup hoelder in time}, and Lemma~\ref{lemma:schauder}, to deduce
\begin{align*}
   \| u \|_{C_{T} \CC^\alpha} + \|u\|_{C^{\alpha/2}_TL^\infty} & \lesssim \| u^\sharp \|_{C_{T} \CC^\alpha} + \|u^\sharp \|_{C^{\alpha/2}_TL^\infty} + \|F(u) \mpara \vartheta(0)\|_\alpha \\
   &\quad + \| L (F(u) \mpara \vartheta) \|_{C_T\CC^{\alpha-2}} \\
   & \lesssim \| u^\sharp \|_{C_{T} \CC^\alpha} + \|u^\sharp \|_{C^{\alpha/2}_TL^\infty} + \|u_0\|_\alpha + \|F\|_{L^\infty}  \|\xi\|_{\alpha-2}.
\end{align*}
We plug this into~\eqref{eq:phi sharp estimate pam} and use $1 + \|u\|_{C_T\CC^\alpha}^{1+\beta/\alpha} + \|u\|_{C^{\alpha/2}_TL^\infty} \lesssim 1 + (\|u\|_{C_T\CC^\alpha} + \|u\|_{C^{\alpha/2}_TL^\infty})^{1+\beta/\alpha}$, which gives
\[
   \| \Phi^{\sharp} ( t ) \|_{\alpha + \beta - 2} \lesssim C_F C_\xi \big( 1 + (C_F C_\xi + \|u_0\|_\alpha + \| u^\sharp \|_{C_{T} \CC^\alpha} + \|u^\sharp \|_{C^{\alpha/2}_TL^\infty})^{1+\beta/\alpha} + \|u^\sharp (t)\|_{\alpha+\beta}\big).
\]
Moreover, since $u^{\sharp} ( 0 ) =u_{0} - ( F ( u ) \mpara \vartheta ) ( 0)$ and $L u^{\sharp} = \Phi^{\sharp}$, Lemma~\ref{lemma:heat flow smoothing} and Lemma~\ref{lemma:schauder} yield
\[
   t^{\beta /2} \| u^{\sharp} ( t ) \|_{\alpha+\beta} \lesssim \| u_{0} \|_{\alpha} + C_{F} C_{\xi} + \sup_{s \in [ 0,t ]} (s^{\beta /2} \| \Phi^{\sharp} ( s ) \|_{\alpha + \beta - 2} ),
\]
so that our new estimate for $\Phi^{\sharp}$ reads
\begin{align*}
   t^{\beta /2} \| \Phi^{\sharp} ( t ) \|_{\alpha+\beta -2} & \lesssim C_F C_\xi \Big( 1 + (C_F C_\xi + \|u_0\|_\alpha + \| u^\sharp \|_{C_{T} \CC^\alpha} + \|u^\sharp \|_{C^{\alpha/2}_TL^\infty})^{1+\beta/\alpha} \\
   &\hspace{55pt} + \sup_{s \in [ 0,t ]} (s^{\beta /2} \| \Phi^{\sharp} ( s ) \|_{\alpha + \beta - 2}) \Big),
\end{align*}
uniformly in $t \in [ 0,T ]$. It remains to control $u^\sharp$ in $C^{\alpha/2}_TL^\infty \cap C_T \CC^\alpha$. For $0\leqslant s < t \leqslant T$, we have
\begin{align*}
   \| u^\sharp(t) - u^\sharp(s)\|_{L^\infty} & \leqslant \|(P_{t-s} - \mathrm{id}) P_s (u^\sharp(0))\|_{L^\infty} + \bigg\| \int_s^t P_{t-s} \Phi^\sharp(r) \mathd r \bigg\|_{L^\infty} \\
   &\quad + \bigg\| \int_0^s (P_{t-s} - \mathrm{id}) P_{s-r} \Phi^\sharp(r) \mathd r \bigg\|_{L^\infty}.
\end{align*}
An application of Lemma~\ref{lemma:semigroup hoelder in time} to the first and third term and Lemma~\ref{lemma:heat flow smoothing} to the second term leads to
\begin{align*}
   \| u^\sharp(t) - u^\sharp(s)\|_{L^\infty} & \lesssim (t-s)^{\alpha/2} \|u^\sharp(0)\|_{\alpha} + \int_s^t (t-s)^{-1+\alpha/2+\beta/2} \|\Phi^\sharp(r)\|_{\alpha+\beta-2} \mathd r \\
   &\quad + (t-s)^{\alpha/2} \int_0^s \| P_{s-r} \Phi^\sharp(r) \|_{\alpha} \mathd r \\
   &\lesssim (t-s)^{\alpha/2} (C_F C_\xi + \|u_0\|_\alpha) \\
   &\quad + (t-s)^{\alpha/2} \int_0^t (t-r)^{-1+\beta/2} r^{-\beta/2} \mathd r \sup_{r \in [ 0,t ]} ( r^{\beta /2} \| \Phi^{\sharp} ( r ) \|_{\alpha + \beta -2} ) \\
   &\quad + (t-s)^{\alpha/2} \int_0^s (s-r)^{-1+\beta/2} r^{-\beta/2}\mathd r \sup_{r \in [ 0,s ]} ( r^{\beta /2} \| \Phi^{\sharp} ( r ) \|_{\alpha + \beta -2} ).
\end{align*}
For the time integrals we have
$\int_{0}^{t} ( t-r )^{-1+ \beta /2} r^{-\beta /2} \mathd r = \int_{0}^{1} ( 1-r )^{1- \beta /2} r^{-\beta /2} \mathd r \lesssim 1$,
so that
\[
   \|u^\sharp\|_{C^{\alpha/2}_TL^\infty} \lesssim C_F C_\xi + \|u_0\|_\alpha + \sup_{s \in [ 0,T]} ( s^{\beta /2} \| \Phi^{\sharp} ( s ) \|_{\alpha + \beta -2} ).
\]
Similar (but easier) arguments can be used to bound the $C_T\CC^\alpha$ norm of $u^\sharp$, and thus we obtain our final estimate for $\Phi^\sharp$:
\begin{align}\label{eq:pam phi sharp final estimate} \nonumber
    &\sup_{t \in [ 0,T]} (t^{\beta /2} \| \Phi^{\sharp} ( t ) \|_{\alpha+\beta -2}) \\
    &\hspace{40pt} \lesssim C_F C_\xi (1+C_FC_\xi) \Big( 1 + \|u_0\|_\alpha + \sup_{t \in [ 0,T ]} (t^{\beta /2} \| \Phi^{\sharp} ( t ) \|_{\alpha + \beta - 2}) \Big)^{1+\beta/\alpha}.
\end{align}
%
%
%
In order to use this estimate to bound $\Phi^{\sharp}$, we will apply the
usual scaling argument. More
precisely, we set $\Lambda_{\lambda} f ( t,x ) = f ( \lambda^{2} t, \lambda x)$, so that $L  \Lambda_{\lambda} = \lambda^{2} \Lambda_{\lambda} L$. Now let
$u^{\lambda} = \Lambda_{\lambda} u$, $u_0^\lambda = \Lambda_\lambda u_0$, $\xi^{\lambda} = \lambda^{2- \alpha} \Lambda_{\lambda} \xi$, and $\vartheta^{\lambda} = \lambda^{-\alpha} \Lambda_{\lambda} \vartheta$. Note that $u^{\lambda} \colon [ 0,T/ \lambda^{2} ] \times \mathbb{T}_{\lambda}^{2} \rightarrow \mathbb{R}$, where
$\mathbb{T}_{\lambda}^{2} = (\R / (2\pi \lambda^{-1} \Z))^{2}$ is a rescaled torus,
and that $u^{\lambda}$ solves the equation
\[
   L u^{\lambda} = \lambda^{2} F ( u^{\lambda} ) \Lambda_{\lambda} \xi = \lambda^{\alpha} F ( u^{\lambda} ) \xi^{\lambda}, \qquad u^\lambda(0) = u^\lambda_0.
\]
The scaling is chosen in such a way that $\| u^\lambda_0 \|_\alpha \lesssim \|u_0\|_\alpha$, $\| \xi^{\lambda} \|_{\CC^{\alpha -2}} \lesssim \| \xi \|_{\CC^{\alpha -2}}$, and according to Lemma~\ref{lem:scaling commutator} also $\| \vartheta^{\lambda} \reso \xi^{\lambda} \|_{2 \alpha -2} \lesssim \| \vartheta \reso \xi \|_{2 \alpha -2} + \|\xi\|_{\alpha-2}^2$, all uniformly in $\lambda \in (0,1]$. In particular, $C_{\xi^{\lambda}} \lesssim C_{\xi}$ and
$C_{\lambda^{\alpha} F} \leqslant \lambda^{\alpha} C_{F}$ for all $\lambda \in (0,1]$. Injecting these estimates into~\eqref{eq:pam phi sharp final estimate}, we obtain
\[
   \sup_{t \in [ 0,T]} (t^{\beta /2} \| \Phi^{\sharp,\lambda} ( t ) \|_{\alpha+\beta -2}) \lesssim 1 + \|u_0^\lambda\|_\alpha
\]
for all sufficiently small $\lambda > 0$ (depending only on $C_\xi$, $C_F$, and $u_0$), where $\Phi^{\sharp,\lambda}$ is defined analogously to $\Phi^\sharp$. From here we easily get the existence of local-in-time paracontrolled solutions to~\eqref{eq:pam}. 
%
Similar arguments show that if $F \in C^{2+\beta/\alpha}_{b}$, then the map $(u_{0} , \xi , \vartheta , \xi \reso \vartheta ) \mapsto u \in C_T \CC^{\alpha}$ is
locally Lipschitz continuous, and in particular there is a unique paracontrolled solution on a small time interval.

If $F \in C^3$, then~\eqref{eq:phi sharp estimate pam 2} allows us to control the paracontrolled norm of the solution $u$ in terms of its $L^\infty$ norm, and in particular for every $C > 0$ there exists a unique paracontrolled solution $u$ on $[0,\tau_C]$, where
\[
   \tau_C = \inf\{ t \geqslant 0: \| u(t) \|_{L^\infty} \ge C\}.
\]
While we are currently not able to establish the existence of global-in-time solutions, this insight allows us to gain a better understanding of the possible blow up, by showing that the only way in which the paracontrolled norm of $u$ can explode is by $u$ diverging to $\pm \infty$.

\subsection{Renormalization}

So far we argued under the assumption that there exist continuous functions $(\xi^\varepsilon)$ such that $(\xi^\varepsilon,\vartheta^\varepsilon, \vartheta^{\varepsilon} \reso \xi^{\varepsilon})$ converges to $(\xi,\vartheta, \vartheta \reso \xi)$ in $\CC^{\alpha-2} \times C_{T} \CC^{2 \alpha -2} \times C_T \CC^{2\alpha-2}$ as $\varepsilon \rightarrow 0$. Note that here the superscript $\varepsilon$ refers to a smooth regularization of the noise, whereas in the previous section the
superscript $\lambda$ referred to a scaling transform. From now on we will no longer
consider scaling transforms, so that no confusion should arise.

One further difficulty is that the resonant product $(\vartheta^{\varepsilon} \reso \xi^{\varepsilon} )$ does not converge in some
relevant cases; in particular, if $\xi$ is a spatial white noise. However, what we will show below is that for the white noise  there
exist constants $c_{\varepsilon} \in \mathbb{R}$ such that $((\vartheta^{\varepsilon} \reso \xi^{\varepsilon} ) -c_{\varepsilon})$ converges in probability
in $C_{T} \CC^{2 \alpha -2}$. In order to make the term $c_{\varepsilon}$
appear in the equation, we can introduce a suitable correction term in the
regularized problems and consider the renormalized PDE
\begin{equation}\label{eq:pam renormalized}
   L u^{\varepsilon} =F ( u^{\varepsilon} ) \xi^{\varepsilon} -c_{\varepsilon} F' ( u^{\varepsilon} ) F ( u^{\varepsilon} ) .
\end{equation}
For this equation we use again the paracontrolled ansatz~\eqref{eq:mod-paracontrolled}. The same derivation as for~\eqref{eq:phi sharp def pam} yields
\begin{align*}
   L u^{\sharp, \varepsilon} & = G(u^\varepsilon, \vartheta^\varepsilon, \xi^\varepsilon) + F(u^\varepsilon) \reso \xi^\varepsilon  - c_{\varepsilon} F' ( u^{\varepsilon} ) F ( u^{\varepsilon} )
\end{align*}
for some bounded functional $G$, and as in Lemma~\ref{lem:pam circ apriori} we decompose
\[
   F(u^\varepsilon) \reso \xi^\varepsilon  - c_{\varepsilon} F' ( u^{\varepsilon} ) F ( u^{\varepsilon} ) = H(u^\varepsilon,u^{\sharp,\varepsilon}, \vartheta^\varepsilon,\xi^\varepsilon) + F'(u^\varepsilon) F(u^\varepsilon) (\vartheta^\varepsilon \reso \xi^\varepsilon - c_\varepsilon)
\]
for another bounded functional $H$. 
We see that $Lu^{\sharp , \varepsilon}$ only depends on $\xi^{\varepsilon}$, $\vartheta^{\varepsilon}$, and $(\vartheta^{\varepsilon} \reso \xi^{\varepsilon} ) -c_{\varepsilon}$. Thus,
the convergence of $(\xi^\varepsilon,\vartheta^\varepsilon, \vartheta^{\varepsilon} \reso \xi^{\varepsilon} - c_\varepsilon)$ to $(\xi,\vartheta, \eta)$ in $\CC^{\alpha-2} \times C_{T} \CC^{2 \alpha -2} \times C_T \CC^{2\alpha-2}$ implies that the solutions $( u^{\varepsilon} )$ to~\eqref{eq:pam renormalized} converge to a limit which only depends on $\xi$, $\vartheta$, and $\eta$,
but not on the approximating family.

\begin{theorem}\label{theorem:pam existence uniqueness}
  Let $\alpha \in ( 2/3,1 )$, $\beta \in (2-2\alpha, \alpha]$ and assume that $( \xi^{\varepsilon})_{\varepsilon >0} \subset C(\mathbb{T}^2,\R)$ and $F \in C^{2+\beta/\alpha}_{b}$. Suppose that there exist $\xi \in \CC^{\alpha -2}$ and $\eta \in C_{T} \CC^{2 \alpha -2}$ such that $( \xi^{\varepsilon}, (\vartheta^{\varepsilon} \reso   \xi^{\varepsilon} ) -c_{\varepsilon} )$
  converges to $( \xi, \eta )$ in $\CC^{\alpha -2} \times \CC^{2 \alpha -2}$, where $\vartheta = \int_0^\infty P_t (\xi - (2\pi)^2 \CF \xi(0))\mathd t$, $\vartheta^{\varepsilon} = \int_0^\infty P_t (\xi^\varepsilon - (2\pi)^2 \CF \xi^\varepsilon(0))\mathd t$, and where $c_{\varepsilon} \in \mathbb{R}$ for all $\varepsilon >0$.
  Let for $\varepsilon >0$ the function $u^{\varepsilon}$ be the unique
  solution to the Cauchy problem
  \[
     L u^{\varepsilon} =F ( u^{\varepsilon} ) \xi^{\varepsilon} -c_{\varepsilon} F' ( u^{\varepsilon} ) F ( u^{\varepsilon} ) , \hspace{2em} u^{\varepsilon} ( 0 ) =u_{0},
  \]
  where $u_{0} \in \CC^{\alpha}$. Then there exists $T^\ast > 0$ such that for all $T < T^\ast$ there is $u \in C_{T} \CC^{\alpha}$ with
  $u^{\varepsilon} \rightarrow u$ in $C_{T} \CC^{\alpha}$. The limit $u$ depends only on $( u_{0} , \xi , \eta )$, and not on the approximating family $( \xi^{\varepsilon} , (\vartheta^{\varepsilon} \reso   \xi^{\varepsilon} ) -c_{\varepsilon} )$. If furthermore $F \in C^3$, then we can take
  \[
     T^\ast = \inf \{ t \geqslant 0: \|u(t) \|_{L^\infty} = \infty\}.
  \]
\end{theorem}
As for the previous equations, $u$ is the unique paracontrolled weak solution to $L u=F ( u )\diamond \xi$ with $u ( 0 ) =u_{0}$ if we interpret the renormalized product $F ( u ) \diamond\xi$ in the right way, and $u$ depends continuously on $u_0$.

\begin{remark}
   In the linear case $F(u) = u$ we can skip the application of the paralinearization theorem. Since this was the only step in which we picked up a superlinear estimate, and all the other estimates that we used were linear in $u$, we then obtain the global-in-time existence of solutions.
\end{remark}

\subsection{Regularity of the area and renormalized products}

It remains to study the regularity of the area $\vartheta \reso \xi$. As already indicated, we will have to renormalize the product by ``subtracting an infinite constant'' in order to obtain a well-defined object.

Let therefore $\xi$ be a white noise on $\mathbb{T}^2$. By definition $(\CF \xi (k))_{k\in \Z^2}$ is a complex valued, centered Gaussian process with covariance
\[
   \E[\CF \xi (k) \CF \xi(k')] = (2\pi)^2 \mathbf{1}_{k = -k'}
\]
and such that $\overline{\CF \xi(k)} = \CF \xi(-k)$ for all $k,k' \in \Z^2$. This yields, using Gaussian hypercontractivity and Besov embedding, that $\E[\| \xi \|_{\CC^{\alpha-2}(\mathbb{T}^2)}^p] < \infty$ for all $\alpha < 1$ and $p \geqslant 1$. Moreover, setting
\[
   \vartheta = \int_0^\infty P_t (\xi - (2\pi)^2\CF\xi(0)) \mathd t,
\]
we have that $(\CF \vartheta(k))$ is a centered, complex valued Gaussian process with covariance
\[
   \E[ \CF \vartheta(k) \CF \vartheta(k')] = (2\pi)^2 \frac{1}{|k|^4} \mathbf{1}_{k = - k'} \mathbf{1}_{k \neq 0}
\]
and such that $\overline{\CF \vartheta(k)} = \CF \vartheta(-k)$ for all $k,k' \in \Z^2$. In the following we define for notational convenience
\[
   \Pi \xi = \xi - (2\pi)^2\CF\xi(0),
\]
so that $\vartheta = \int_0^\infty \Pi \xi \mathd t$. Since $P_t \Pi \xi$ is a smooth function for $t > 0$, the resonant term $P_t \Pi \xi \reso \xi$ is a smooth function, and therefore we could formally set $\vartheta \reso \xi = \int_0^\infty (P_t \Pi \xi \reso \xi) \mathd t$. However, this expression is not well defined:

\begin{lemma}\label{l:anderson area expectation}
   For any $x \in \mathbb{T}^2$ and $t > 0$ we have
   \[
      g_t = \E[(P_t \Pi \xi \reso \xi)(x)] = \E[ \Delta_{- 1} (P_t \Pi \xi \reso \xi ) ( x ) ] = (2\pi)^{-2} \sum_{k \in\mathbb{Z}^2\setminus \{0\}} e^{- t | k |^2}.
   \]
   In particular, $g_t$ does not depend on the partition of unity used to define the $\reso$ operator, and $\int_0^\varepsilon g_t \mathd t = \infty$ for all $\varepsilon > 0$.
\end{lemma}

\begin{proof}
   Let $x \in \mathbb{T}^2$, $t > 0$, and $\ell \geqslant - 1$. Then
   \begin{align*}
      \E[\Delta_\ell ( P_t \Pi \xi \reso \xi) ( x ) ] = \sum_{| i - j |\leqslant 1} \E [ \Delta_\ell ( \Delta_i ( P_t \Pi \xi) \Delta_j \xi ) ( x ) ], 
   \end{align*}
   where exchanging summation and expectation is justified because it can be easily verified that the partial sums of $\Delta_\ell ( P_t \Pi \xi \reso \xi ) ( x )$ are uniformly $L^p$--bounded for any $p \geqslant 1$. 
   Now $P_t = e^{-t\lvert \mathD \rvert^2}$, and therefore
   \begin{align*}
      &\E[\Delta_\ell ( \Delta_i ( P_t \Pi \xi )\Delta_j \xi ) ( x ) ] \\
      &\hspace{20pt} = ( 2 \pi )^{-4} \sum_{k \in \mathbb{Z}^2 \setminus \{0\}, k' \in \Z^2} e^{\imath \langle k+k', x \rangle}\rho_\ell ( k + k' ) \rho_i ( k ) e^{- t | k|^2} \rho_j (k') \E [\CF{\xi} ( k ) \CF{\xi}(k')]\\
      &\hspace{20pt} = (2\pi)^{-2} \sum_{k \in \mathbb{Z}^2\setminus \{0\}} \rho_\ell ( 0 ) \rho_i ( k) e^{- t | k |^2} \rho_j ( k ) \\
      &\hspace{20pt} =  (2\pi)^{-2}\mathbf{1}_{\ell=-1} \sum_{k \in \mathbb{Z}^2 \setminus \{0\}} \rho_i(k) \rho_j ( k ) e^{- t | k |^2}.
   \end{align*}
   For $| i - j | > 1$ we have $\rho_i(k) \rho_j(k) = 0$. This implies, independently of $x \in \mathbb{T}^2$, that
   \[
      g_t = \E[ ( P_t \xi \reso \xi)(x)] = \sum_{k \in \mathbb{Z}^2\setminus \{0\}} \sum_{i, j} \rho_i ( k) \rho_j ( k ) e^{- t | k |^2} = (2\pi)^{-2} \sum_{k \in \mathbb{Z}^2 \setminus{\{0\}}} e^{- t | k |^2} ,
   \]
   while $\E[(P_t \xi \reso \xi)(x) - \Delta_{- 1} ( P_t \xi \reso \xi))( x ) ] = 0$.
\end{proof}

\begin{remark}
   The same calculation shows that if $\psi \in \CS$, and if $\xi^\varepsilon = \varepsilon^{-2} \psi (\varepsilon^{-1} \cdot) \ast \xi$, then
   \[
      \E[ (P_t \Pi \xi^\varepsilon \reso \xi^\varepsilon) ( x ) ] = \E[ \Delta_{- 1} (P_t \Pi \xi^\varepsilon \reso \xi^\varepsilon ) ( x ) ] = (2\pi)^{-2} \sum_{k \in\mathbb{Z}^2\setminus \{0\}} e^{- t | k |^2} |\CF \psi (\varepsilon k)|^2.
   \]
\end{remark}

The diverging time integral motivates us to study the renormalized product $\vartheta \reso \xi  - \int_0^\infty g_t \mathd t$, where $\int_0^\infty g_t \mathd t$ is an infinite constant:

\begin{lemma}\label{l:pam area}
   Set
   \[
      (\vartheta \diamond \xi) = \int_0^\infty (P_t \Pi \xi \reso \xi - g_t) \mathd t.
   \]
   Then $\E [\| \vartheta \diamond \xi\|_{2\alpha - 2}^p]<\infty$ for all $\alpha < 1$, $p \geqslant 1$. Moreover, if $\psi \in \CS$ satisfies $\int \psi(x) \mathd x = 1$, and if $\xi^\varepsilon = \varepsilon^{-2} \psi(\varepsilon \cdot) \ast \xi$ for $\varepsilon > 0$, and $\vartheta^\varepsilon = \int_0^\infty P_t \Pi \xi^\varepsilon \mathd t$, then
   \[
      \lim_{\varepsilon \rightarrow 0} \E[ \| \vartheta \diamond \xi - ( \vartheta^\varepsilon \reso \xi^\varepsilon - c_\varepsilon) \|_{2\alpha - 2}^p] = 0
   \]
   for all $p \geqslant 1$, where for $x \in \mathbb{T}^2$
   \begin{align*}
      c_\varepsilon & = \E[\vartheta^\varepsilon(x) \xi^\varepsilon(x)] = \E[\vartheta^\varepsilon \reso \xi^\varepsilon(x)] = \int_0^\infty \E[P_t \Pi \xi^\varepsilon \reso \xi^\varepsilon(x)] \mathd t \\
      & =  (2\pi)^{-2} \sum_{k \in \Z^2 \setminus\{0\}} \frac{|\CF \psi(\varepsilon k)|^2}{|k|^2}.
   \end{align*}
\end{lemma}

\begin{proof}
   We split the time integral into two components, $\int_0^1 \dots \mathd t$ and $\int_1^\infty \dots \mathd t$. The second integral can be treated without relying on probabilistic estimates: Given $x \in \mathbb{T}^2$, we have
   \begin{align*}
      &\bigg\| \int_1^\infty (P_t \Pi \xi \reso \xi - g_t) \mathd t - \int_1^\infty (P_t \Pi \xi^\varepsilon \reso \xi^\varepsilon - \E[P_t \Pi \xi^\varepsilon \reso \xi^\varepsilon(x)]) \mathd t \bigg\|_{2\alpha-2} \\
      &\hspace{20pt}\lesssim \int_1^\infty \| P_t \Pi \xi \reso \xi - P_t \Pi \xi^\varepsilon \reso \xi^\varepsilon \|_{2\alpha} \mathd t + \int_1^\infty \sum_{k \in \Z^2 \setminus\{0\}} e^{-t |k|^2} |1 - |\CF \psi(\varepsilon k)|^2| \mathd t \\
      &\hspace{20pt}\lesssim  \int_1^\infty (\| P_t \Pi(\xi - \xi^\varepsilon) \|_{\alpha + 2} \| \xi \|_{\alpha-2} + \| P_t \Pi \xi^\varepsilon \|_{\alpha + 2} \| \xi - \xi^\varepsilon \|_{\alpha - 2}) \mathd t \\
      &\hspace{20pt}\quad+ \sum_{k \in \Z^2 \setminus\{0\}} \frac{e^{- |k|^2}}{|k|^2} |1 - |\CF \psi(\varepsilon k)|^2|,
   \end{align*}
   Since $\CF \Pi \xi^\varepsilon (0) = 0$, the estimate $\| P_t \Pi \xi^\varepsilon \|_{\alpha + 2} \lesssim t^{-2} \| \xi^\varepsilon \|_{\alpha-2}$ of Lemma~\ref{lemma:heat flow smoothing} holds uniformly over $t > 0$, and thus the time integral is finite. The convergence in $L^p(\mathbb{P})$ now easily follows from the dominated convergence theorem.
   
   We will treat the integral from 0 to 1 using similar arguments as in the proof of Lemma~\ref{l:burgers area}. To lighten the notation, we will only show that $\E[\| \int_0^1 (P_t \Pi \xi \reso \xi - g_t) \mathd t\|_{2\alpha-2}^p]<\infty$. The difference
   \[
      \E\Big[ \Big\| \int_0^1 (P_t \Pi \xi \reso \xi - g_t) \mathd t - \int_0^1 (P_t \Pi \xi^\varepsilon \reso \xi^\varepsilon - \E[P_t \Pi \xi^\varepsilon \reso \xi^\varepsilon(x)]) \mathd t \Big\|_{2\alpha - 2}^p\Big]
   \]
   can be treated with the same arguments, we only have to include some additional factors of the form $|1 - \CF \psi (\varepsilon k)|^2$ in the sums below. The convergence of the expectation can then be shown using dominated convergence.
   
   Let $t \in (0,1]$ and define $\Xi_t = P_t \Pi \xi \reso \xi - g_t$. By the equivalence of moments for random variables living in an inhomogeneous Gaussian chaos of fixed degree, we obtain for $p\geqslant 1$ and $m \geqslant - 1$ that
   \begin{equation}\label{e:anderson area hyper}
      \E [\Vert  \Delta_m \Xi_t \Vert_{L^{2 p} ( \mathbb{T}^2 )}^{2 p}] \lesssim_p \lVert \E[|\Delta_m \Xi_t ( x ) |^2] \rVert ^p_{L^p_x(\mathbb{T}^2)}.
   \end{equation}
   By Lemma~\ref{l:anderson area expectation} we have
   \begin{equation}\label{e:anderson area m-block}
      \E[| \Delta_m \Xi_t(x)|^2] = \tmop{Var} ( \Delta_m ( P_t \xi\reso \xi ) ( x ) ),
   \end{equation}
   for all $m \ge -1$, where $\tmop{Var}(\cdot)$ denotes the variance.
   Now
   \begin{align*}
      \Delta_m ( P_t \xi\reso \xi ) (x) & = ( 2 \pi )^{- 4} \sum_{k_1\in \mathbb{Z}^2 \setminus \{0\}, k_2 \in \Z^2}\sum_{| i - j | \leqslant 1} e^{\imath \langle k_1 + k_2, x \rangle} \rho_m ( k_1 + k_2 ) \rho_i ( k_1 )  \\
      &\hspace{120pt} \times e^{- t | k_1 |^2} \CF{\xi} ( k_1 ) \rho_j (k_2) \CF{\xi} ( k_2 ),
   \end{align*}
   and therefore
   \begin{align*}
      &\tmop{Var} ( \Delta_m ( P_t \xi\reso \xi ) ( x ) )\\
      &\hspace{10pt} = ( 2 \pi )^{- 8} \sum_{k_1, k'_1 \in \mathbb{Z}^2 \setminus \{0\}} \sum_{k_2, k'_2 \in \mathbb{Z}^2} \sum_{| i - j | \leqslant 1} \sum_{| i' - j' | \leqslant 1} e^{\imath \langle k_1 + k_2, x \rangle} \rho_m ( k_1 + k_2 ) \rho_i ( k_1) e^{- t | k_1 |^2} \rho_j ( k_2 )\\
      &\hspace{160pt} \times e^{\imath \langle k'_1 + k'_2, x \rangle} \rho_m (k'_1 + k'_2 ) \rho_{i'} ( k'_1 ) e^{- t | k'_1 |^2}\rho_{j'} ( k'_2 ) \\
      &\hspace{160pt} \times \tmop{cov}(\CF{\xi} ( k_1 )\CF{\xi} ( k_2 ), \CF{\xi} ( k'_1 ) \CF{\xi}( k'_2 ) ),
   \end{align*}
   where the exchange of summation and expectation can again be justified a posteriori by the uniform $L^p$--boundedness of the partial sums, and where $\tmop{cov}$ denotes the covariance. Since $( \widehat{\xi} (k) )_{k \in \mathbb{Z}^2}$ is a centered Gaussian process, we can apply Wick's theorem (\cite{Janson1997}, Theorem~1.28) to deduce
   \begin{align*}
      &\tmop{cov} ( \widehat{\xi} ( k_1 ) \widehat{\xi} (k_2), \widehat{\xi} ( k'_1 ) \widehat{\xi} (k'_2)) = ( 2 \pi )^4 ( \mathbf{1}_{k_1=-k_1'} \mathbf{1}_{k_2 = - k'_2} + \mathbf{1}_{k_1 = - k'_2} \mathbf{1}_{k_2 = - k_1'} ),
   \end{align*}
   and therefore
   \begin{align*}
      &(2\pi)^{4} \tmop{Var} ( \Delta_m ( P_t \xi\reso \xi ) ( x ) ) \\
      &\hspace{15pt}=  \sum_{k_1 \neq 0, k_2} \sum_{| i - j|\leqslant 1} \sum_{| i' - j'| \leqslant 1} \Big[ \mathbf{1}_{m \lesssim i} \mathbf{1}_{m \lesssim i'} \rho^2_m(k_1 + k_2) \rho_i ( k_1 ) \rho_j ( k_2 )\rho_{i'} ( k_1 ) \rho_{j'} ( k_2 ) e^{- 2 t |k_1 |^2}\\
      &\hspace{80pt} + \mathbf{1}_{m \lesssim i} \mathbf{1}_{m \lesssim i'} \rho^2_m (k_1 + k_2 ) \rho_i ( k_1 ) \rho_j ( k_2 ) \rho_{i'} ( k_2 ) \rho_{j'} ( k_1 ) e^{- t |k_1 |^2 - t | k_2 |^2} \Big].
   \end{align*}
   There exists $c > 0$ such that $e^{- 2 t | k|^2} \lesssim e^{- t c 2^{2 i}}$ for all $k \in \tmop{supp} ( \rho_i )$ and for all $i \geqslant -1$. In the  remainder of the proof the value of this strictly positive $c$ may change from line to line. If $| i - j | \leqslant 1$, then we also have $e^{- t | k |^2} \lesssim e^{- t c 2^{2 i}}$ for all $k \in \tmop{supp} ( \rho_j )$. Thus
   \begin{align}\label{e:anderson area var}\nonumber
      & \tmop{Var} ( \Delta_m  ( P_t \xi \reso \xi )) ( x ) )\\ \nonumber
      &\hspace{35pt} \lesssim \sum_{i, j, i', j'} \mathbf{1}_{m \lesssim i} \mathbf{1}_{i \sim j \sim i' \sim j'} \sum_{k_1, k_2} \mathbf{1}_{\tmop{supp}(\rho_m)}(k_1 + k_2) \mathbf{1}_{\tmop{supp}( \rho_i )}(k_1) \mathbf{1}_{\tmop{supp} ( \rho_j)}(k_2) e^{- 2 t c 2^{2 i}} \\
      &\hspace{35pt} \lesssim \sum_{i:i \gtrsim m} 2^{2 i} 2^{2 m} e^{- t c 2^{2 i}} \lesssim \frac{2^{2 m}}{t} \sum_{i: i \gtrsim m} e^{- t c 2^{2 i}} \lesssim \frac{2^{2 m}}{t} e^{- t c 2^{2 m}},
   \end{align}
   where we used that $t 2^{2 i} \lesssim e^{t ( c - c' ) 2^{2 i}}$ for any $c' < c$.
   
   Now let $\alpha < 1$. We apply Jensen's inequality and combine \eqref{e:anderson area hyper}, \eqref{e:anderson area m-block}, and \eqref{e:anderson area var} to obtain
   \begin{align*}
      \E[\Vert  \Xi_t \Vert_{B^{2 \alpha - 2}_{2 p, 2 p}}] & \lesssim \biggl( \sum_{m \geqslant - 1} 2^{(2 \alpha - 2) m 2 p} \E[\Vert  \Delta_m \Xi_t \Vert_{L^{2 p} (\mathbb{T}^2 )}^{2 p} ] \biggr)^{\frac{1}{2 p}} \\
      & \lesssim t^{- 1 / 2} \biggl( \sum_{m \geqslant - 1} 2^{(2 \alpha - 2) m 2 p} 2^{2 m p} e^{- t c p 2^{2 m}} \biggr)^{\frac{1}{2 p}} \\
      & \lesssim t^{- 1 / 2} \biggl( \int_{- 1}^{\infty} ( 2^x )^{2 p ( 2 \alpha - 1 )} e^{- c t p ( 2^x )^2} \mathd x \biggr)^{\frac{1}{2 p}}.
   \end{align*}
   The change of variables $y = \sqrt{t} 2^x$ then yields
   \begin{align*}
      \E[\Vert  \Xi_t \Vert_{B^{2 \alpha - 2}_{2 p, 2 p}}] \lesssim t^{- 1 / 2} \biggl( t^{- p ( 2 \alpha -1)} \int_0^{\infty} y^{2 p ( 2 \alpha - 1 ) - 1} e^{- c p y^2} \mathd y \biggr)^{\frac{1}{2 p}}.
   \end{align*}
   If $\alpha > 1/2$, the integral is finite for all sufficiently large $p$, and therefore $\E[\Vert  \Xi_t \Vert_{B^{2 \alpha - 2}_{2 p, 2 p}}] \lesssim_p t^{- \alpha}$, so that $\int_0^1 \E[\Vert  \Xi_t \Vert_{B^{2 \alpha - 2}_{2 p, 2 p}}] \mathd t < \infty$ for all $\alpha < 1$. The equivalence of moments for $\int_0^1 \Xi_t \mathd t$ allows us to conclude that also
   \[
      \E \Big[\Big\| \int_0^1 \Xi_t \mathd t\Big\|^p_{B^{2\alpha - 2}_{2p, 2p}} \Big] < \infty
   \]
   for all $p \geqslant 1$. The result now follows from the Besov embedding theorem, Lemma~\ref{l:besov embedding torus}.
\end{proof}


Combining the construction of the renormalized product $\vartheta \diamond \xi$ with Theorem~\ref{theorem:pam existence uniqueness}, we obtain the existence and uniqueness of solutions to the generalized parabolic Anderson model:

\begin{corollary}
  Let $\alpha \in ( 2/3, 1)$, $\beta\in (2-2\alpha,\alpha]$, $F \in C^{2+\beta/\alpha}_b$, $u_0 \in \CC^\alpha$, $L = \partial_t-\Delta$, and let $\xi$ be a spatial white noise on $\mathbb{T}^2$. Then there exists a unique solution $u$ to
  \[
     L u = F(u)\diamond \xi, \qquad u(0) = u_0,
  \]
  in the following sense: For $\psi \in \CS$ with $\int \psi \mathd t=1$ and for $\varepsilon > 0$ consider the solution $u^\varepsilon$ to
  \[
     L u^{\varepsilon} =F ( u^{\varepsilon} ) \xi^{\varepsilon} -c_{\varepsilon} F' ( u^{\varepsilon} ) F ( u^{\varepsilon} ) , \hspace{2em} u^{\varepsilon} ( 0 ) =u_{0},
  \]
  on $[0, \infty)\times\mathbb{T}^2$, where $\xi^\varepsilon = \varepsilon^{-1} \psi(\varepsilon \cdot) \ast \xi$, and where $c_\varepsilon$ is as defined in Lemma~\ref{l:pam area}. Then there exists a $(u_0, \xi)$--measurable random time $\tau$ such that $\P(\tau>0) = 1$ and such that $\| u^\varepsilon - u \|_{C_\tau \CC^\alpha}$ converges to 0 in probability.
\end{corollary}

\begin{remark}
Concerning the convergence of $( \vartheta^{\varepsilon} \reso
\xi^{\varepsilon} )$, let us make the following remark: Since $-\Delta \vartheta^{\varepsilon} = \xi^{\varepsilon} +C^\infty$ (with a $C^\infty$ remainder that can be controlled uniformly in $\varepsilon > 0$), we have
\[
   \vartheta^{\varepsilon} \reso \xi^{\varepsilon} = \vartheta^{\varepsilon} \reso (-\Delta) \vartheta^{\varepsilon} + C^\infty = \frac{1}{2} (-\Delta) ( \vartheta^{\varepsilon} \reso   \vartheta^{\varepsilon} ) + (\mathD_{x} \vartheta^{\varepsilon} \reso \mathD_{x}  \vartheta^{\varepsilon} ) + C^\infty,
\]
from which we see that the only problem in passing to the limit is given by
the second term on the right hand side. This integration by parts formula is the crucial difference with what happens in the \textsc{rde} case, which otherwise shares many structural properties with the \textsc{pam} model. The fact that $-\Delta$ is a second order operator generates the term $(\mathD_{x} \vartheta^{\varepsilon} \reso \mathD_{x}  \vartheta^{\varepsilon} )$ in the above computation, which is absent in case of the operator $\partial_t$. This term, whose convergence is equivalent to the convergence of the positive term $|\mathD_x \vartheta^\varepsilon|^2$, cannot have simple cancellation properties and it is the origin for the need of introducing an additive renormalization when considering \textsc{pam}.

Our previous analysis easily implies  that the solutions to the modified
problem
\[
   L u^{\varepsilon} =F ( u^{\varepsilon} ) \xi^{\varepsilon} - F'( u^{\varepsilon} ) F (u^{\varepsilon} ) | \mathD_{x} \vartheta^{\varepsilon}  |^2
\]
will converge as soon as $\xi^{\varepsilon} \rightarrow \xi$ in $\CC^{\alpha
-2}$, without any requirements on the bilinear term $\vartheta^{\varepsilon} \reso
\xi^{\varepsilon}$ . 

\end{remark}

\section{Relation with regularity structures}\label{sec:regularity}

In~{\cite{Hairer2013a}} Hairer introduces a general framework that allows to
describe distributions which locally behave like a linear combination of a set
of basic distributions. He calls this set a \emph{model}. A \emph{modelled distribution} is the result of
patching up in a coherent fashion the local models according to a set of
coefficients. At the core of his theory of
regularity structures is the \emph{reconstruction map} $\mathcal{R}$
which, for a given set of coefficients, delivers a modelled distribution that
has the required local behavior up to small errors. In this section we review
the concepts of model and modelled distribution and we use paracontrolled
techniques to explicitly identify modelled distributions as distributions that are 
paracontrolled by a given model, and thus partially bridge the gap between the two
theories. We conjecture that there is a complete correspondence between paracontrolled and modelled distributions, however for now this remains an open problem.

\medskip

We denote by $( K_{i} )_{i \geqslant -1}$ the convolution kernels corresponding to the family of Littlewood--Paley projectors $( \Delta_{i} )_{i \geqslant -1}$, and we write $K_{<i} = \sum_{j<i} K_j$ and $K_{\leqslant i} = \sum_{j \leqslant i} K_j$. For any integral kernel $V$ denote $V_{x} (
y ) =V ( x-y )$ so for example $K_{i,x} ( y ) =K_{i} ( x-y )$.

Let us briefly recall the basic setup of regularity structures. For more
details the reader is referred to Hairer's original paper~{\cite{Hairer2013a}}.

\begin{definition}
  Let $A \subset \R$ be bounded from below and without accumulation
  points except possibly at $\infty$, and let $T =  \oplus_{\alpha \in A} T_{\alpha}$ be a vector space
  graded by $A$ and such that $T_{\alpha}$ is a Banach space for all $\alpha
  \in A$. Let $G$ be a group of continuous operators on $T$ such that for all
  $\tau \in T_{\alpha}$ and $\Gamma \in G$ we have $\Gamma \tau - \tau \in
  \oplus_{\beta < \alpha} T_{\beta}$. The triple $\mathcal{T}= ( A,T,G )$ is
  called a {\emph{regularity structure}} with \emph{model space} $T$ and \emph{structure
  group} $G$.
\end{definition}

For $\tau \in T$ we write $\| \tau \|_{\alpha}$ for the norm of the component of
$\tau$ in $T_{\alpha}$. We assume also that $0 \in A$ and $T_{0} \simeq
\R$ and that $T_{0}$ is invariant under $G$. We will often write
$\varphi^{\lambda}_{x} ( y ) = \lambda^{-d} \varphi ( ( y-x ) / \lambda )$.

\begin{definition}
  Given a regularity structure $\mathcal{T}$ and an integer $d \geqslant 1$, a
  model for $\mathcal{T}$ on $\R^{d}$ consists of maps
  \[ \begin{array}{ccc}
       \Pi :\R^{d} \rightarrow \mathcal{L} ( T,\CS' (
       \R^{d} ) ) &  & \Gamma :\R^{d} \times \R^{d}
       \rightarrow G\\
       x \mapsto \Pi_{x} &  & ( x,y ) \mapsto \Gamma_{x,y}
     \end{array} \]
  such that $\Gamma_{x,y} \Gamma_{y,z} = \Gamma_{x,z}$ and $\Pi_{x}
  \Gamma_{x,y} = \Pi_{y}$. Furthermore, given $r> | \min  A |$, 
  $\gamma >0$, there exists a constant $C$ such that the bounds
  \[
     | ( \Pi_{x} \tau ) ( \varphi^{\lambda}_{x} ) | \leqslant C \lambda^{\alpha} \| \tau \|_{\alpha} , \hspace{2em} \| \Gamma_{x,y} \tau  \|_{\beta} \leqslant C | x-y |^{\alpha - \beta} \| \tau \|_{\alpha}
  \]
  hold uniformly over $\varphi \in C^{r}_b ( \R^{d} )$ with $\| \varphi
  \|_{C^{r}_b} \leqslant 1$ and with support in the unit ball of $\R^{d}$,
  $x,y  \in \R^d$, $0< \lambda \leqslant 1$ and $\tau \in
  T_{\alpha}$ with $\alpha \leqslant \gamma$ and $\beta < \alpha$.
\end{definition}

In~\cite{Hairer2013a}, these conditions are only required to hold locally uniformly, that is for $x,y$ contained in a compact subset of $\R^d$. To simplify the presentation and to facilitate the comparison with the paracontrolled approach, we will work here under global assumptions. In that case we can extend the bounds on the model from compactly supported smooth functions to rapidly decaying smooth functions:

\begin{lemma}\label{lem:Pi acting on Schwartz}
   Let $\varphi$ be a Schwartz function, let $\gamma > 0$, and $r > |\min A|$. Then there exists $C_\varphi > 0$ such that
   \[
      |(\Pi_x \tau)(\varphi^\lambda_x)| \leqslant C_\varphi \lambda^{\alpha} \| \tau\|_\alpha
   \]
   holds uniformly over $0< \lambda \leqslant 1$ and $\tau \in T_{\alpha}$ with $\alpha \leqslant \gamma$. The constant $C_\varphi$ can be chosen proportional to
   \[
      \sup_{|\mu|\le \lceil r\rceil}\sup_{x \in \R^d} (1+|x|)^{d+r+\gamma} |\partial^\mu \varphi(x)|.
   \]
\end{lemma}

\begin{proof}
   We can decompose $\varphi = \sum_{k \in \Z^d} \varphi_k$, where every $\varphi_k \in C^\infty_c$ is supported in the ball with radius $\sqrt{d}$, centered at $k \in \Z^d$. Then $\psi = \sum_{|k|\leqslant \sqrt{d}+1} \varphi_k$ is a compactly supported smooth function, and therefore
   \[
      |(\Pi_x \tau)(\psi^\lambda_x)| \lesssim_\varphi \lambda^{\alpha} \| \tau\|_\alpha.
   \]
   For $|k|>\sqrt{d}+1$ we have $(\varphi_k)^\lambda_x = (\widetilde{\varphi}_k)^\lambda_{x-k}$ for $\widetilde{\varphi}_k$ supported in a ball centered at 0. Using that $\varphi$ is a Schwartz function, we can estimate $\|(\widetilde\varphi_k)^\lambda\|_{C^r_b}\lesssim_{\varphi} \lambda^{-r-d} (|k|/\lambda)^{-(d+r+\alpha)}$. Therefore,
   \begin{align*}
      \sum_{|k| > \sqrt{d} + 1}  |(\Pi_x \tau)((\varphi_k)^\lambda_x)| & \lesssim \sum_{|k| > \sqrt{d} + 1}  |(\Pi_{x-k} \Gamma_{x-k,x} \tau)((\widetilde \varphi_k)^\lambda_{x-k})| \\
      &\lesssim_{\varphi,m} \sum_{|k| > \sqrt{d} + 1} \sum_{\beta\leqslant\alpha} |k|^{\alpha-\beta} \|\tau\|_\alpha |k|^{-(d+r+\alpha)} \lambda^{-r-d+(d+r+\alpha)} \\
      &\lesssim \|\tau\|_\alpha \lambda^\alpha.
   \end{align*}
\end{proof}

In the theory of regularity structures, the usual spaces of regular functions are replaced by spaces of ``modelled distributions''.

\begin{definition}
  For $\gamma \in \R$, the space of {\em modelled distributions} $\mathcal{D}^{\gamma} ( \mathcal{T}, \Gamma )$ consists of all
  functions $f^{\pi}\colon \R^{d} \rightarrow \oplus_{\alpha < \gamma}
  T_{\alpha}$ such that for every $\alpha  < \gamma$ there exists a constant $C$ with
  \[
     \| f^{\pi}_{x} - \Gamma_{x,y} f^{\pi}_{y} \|_{\alpha} \leqslant C | x-y |^{\gamma - \alpha}, \qquad \|f^\pi_x \|_\alpha \leqslant C,
  \]
  uniformly over $x,y \in \R^d$.
\end{definition}

One of the key difficulties is to show that for every modelled distribution $f^\pi$ there exists an associated element of $\CS'$ whose local description is given by $f^\pi$. This is achieved with the help of Hairer's reconstruction operator, for which we give an alternative construction based on paraproducts below.

\subsection{The reconstruction operator}

\begin{definition}\label{def:reconstruction}
   Let $\gamma \in \R$ and $r> | \min  A |$. A reconstruction $\mathcal{R}f^{\pi}$ of $f^{\pi} \in \mathcal{D}^{\gamma} ( \mathcal{T}, \Gamma )$ is a distribution such
  that
  \begin{equation}\label{eq:uniq-rec-gen}
    | \mathcal{R}f^{\pi} ( \varphi^{\lambda}_{x} ) - \Pi_{x} f^{\pi}_{x} (\varphi^{\lambda}_{x} ) | \lesssim \lambda^{\gamma}
  \end{equation}
  for all $0< \lambda \leqslant 1$, uniformly in $x \in
  \R^{d}$ and uniformly over $\varphi \in C^{r+\gamma}_b ( \R^{d} )$ with $\| \varphi
  \|_{C^{r+\gamma}_b} \leqslant 1$ and with support in the unit ball of $\R^{d}$.
\end{definition}

In~\cite{Hairer2013a} inequality~\eqref{eq:uniq-rec-gen} is assumed to hold for all $\varphi \in C^{r}_b ( \R^{d} )$ with $\| \varphi
  \|_{C^{r}_b} \leqslant 1$ and with support in the unit ball of $\R^{d}$. It should be possible to show that this follows from~\eqref{eq:uniq-rec-gen} and the definition of $\Pi$ and $\mathcal{D}^{\gamma} ( \mathcal{T}, \Gamma )$. But for our purposes Definition~\ref{def:reconstruction} will be sufficient.

\begin{lemma}
  \label{lemma:rec-equivalence}Property~{\eqref{eq:uniq-rec-gen}} is
  equivalent to
  \begin{equation}
    | \mathcal{R}f^{\pi} ( K_{< i,x} ) - \Pi_{x} f^{\pi}_{x} (
    K_{< i,x} ) | \lesssim 2^{-i \gamma} \label{eq:uniq-rec-LP}
  \end{equation}
  for all $i \geqslant 0$ and $x \in \R^{d}$. 
\end{lemma}

\begin{proof}
  Start by assuming~\eqref{eq:uniq-rec-LP}. Lemma~\ref{lem:Pi acting on Schwartz} yields $|\Pi_{x} f^{\pi}_{x} (K_{< i,x} )| \lesssim 2^{-i \alpha_0}$, where $\alpha_0 = \min A$, and therefore $| \mathcal{R}f^{\pi} ( K_{< i,x} ) | \lesssim 2^{-i \alpha_0}$. In particular, $\mathcal R f^\pi \in \CC^{\alpha_0}$ and $|\mathcal R f^\pi (\psi)| \lesssim \|\psi\|_{C^r_b}$ for all $\psi \in C^r_b$. If now $\varphi \in C_b^{\gamma+r}$ is supported in the unit ball and if $i \geqslant 0$ is such that $2^{-i}\simeq \lambda$, then Lemma~\ref{lem:Pi acting on Schwartz} yields
  \[
     | ( \mathcal{R}f^{\pi} - \Pi_{x} f^{\pi}_{x} ) ( \varphi^{\lambda}_{x} - S_i \varphi^{\lambda}_{x} ) | \lesssim 2^{-i \gamma} \|\varphi\|_{C^{\gamma+r}_b} \lesssim \lambda^{\gamma} \|\varphi\|_{C^{\gamma+r}_b}.
  \]
  Next, observe that
  \begin{align*}
     ( \mathcal{R}f^{\pi} - \Pi_{x} f^{\pi}_{x} ) ( S_i \varphi^{\lambda}_{x} ) &= \int \mathd z ( \mathcal{R}f^{\pi} - \Pi_{x} f^{\pi}_{x} ) ( K_{< i,z} ) \lambda^{-d} \varphi ( \lambda^{-1} ( x-z ) ) \\
     & = \int \mathd z ( \mathcal{R}f^{\pi} - \Pi_{z} f^{\pi}_{z} ) ( K_{< i,z} ) \lambda^{-d} \varphi ( \lambda^{-1} ( x-z ) ) \\
     &\quad + \int \mathd z \Pi_{z} ( f^{\pi}_{z} - \Gamma_{z,x} f^{\pi}_{x} ) (K_{< i,z} ) \lambda^{-d} \varphi ( \lambda^{-1} ( x-z ) ).
  \end{align*}
  In the second term of this sum we can estimate $| \Pi_{z} ( f^{\pi}_{z} - \Gamma_{z,x} f^{\pi}_{x} ) ( K_{< i,z} ) | \lesssim \sum_{\beta < \gamma} 2^{-i \beta} | x-z |^{\gamma - \beta}$, where we used that $f^{\pi} \in \mathcal{D}^{\gamma}$. The first term in the sum is estimated using~{\eqref{eq:uniq-rec-LP}}, which gives
  \[
     |( \mathcal{R}f^{\pi} - \Pi_{x} f^{\pi}_{x} ) ( S_i \varphi^{\lambda}_{x} )| \lesssim 2^{-i \gamma} + \sum_{\beta < \gamma}  2^{-i \beta} \int \mathd z | x-z |^{\gamma - \beta} \lambda^{-d} \varphi     ( \lambda^{-1} ( z-x ) ) \lesssim 2^{-i \gamma}.
  \]
  So requiring~{\eqref{eq:uniq-rec-LP}} is sufficient to have the general
  bound~{\eqref{eq:uniq-rec-gen}}. To see that~{\eqref{eq:uniq-rec-gen}} implies~{\eqref{eq:uniq-rec-LP}} we can use similar arguments as in the proof of Lemma~\ref{lem:Pi acting on Schwartz}.
\end{proof}

The characterization of the reconstruction given by~{\eqref{eq:uniq-rec-LP}}
is better suited for us, so we will stick with it in the following.

\begin{lemma}
  If $\gamma > 0$, the reconstruction operator is unique.
\end{lemma}

\begin{proof}
  Indeed, for the difference of two reconstructions $\mathcal{R}f^{\pi}$ and
  $\tilde{\mathcal{R}} f^{\pi}$ we have
  \[
     \| S_{i} (\mathcal{R}f^{\pi} - \tilde{\mathcal{R}} f^{\pi}) \|_{L^\infty} \lesssim 2^{-i \gamma},
  \]
  and therefore $0 = \lim_{i\to \infty}  S_{i} (\mathcal{R}f^{\pi} - \tilde{\mathcal{R}} f^{\pi}) = \mathcal{R}f^{\pi} - \tilde{\mathcal{R}} f^{\pi}$.
\end{proof}

\subsection{Paraproducts and modelled distributions}

We are now going to generalize the paraproduct defined previously in order to
apply it to a given model. Fix a model $\Pi$ and for every $i \geqslant 0$ and $\gamma \in \R$
define the operator $P_{i} :\mathcal{D}^{\gamma} ( \mathcal{T}, \Gamma ) \rightarrow \mathcal{S}' (
\R^{d} )$ by
\[
   P_{i}  f^{\pi} ( x ) = \int \mathd zK_{< i-1,x} ( z ) \Pi_{z} f^{\pi}_{z} ( K_{i,x} ) .
\]
Note that
\begin{align*}
   P_{i} f^{\pi} ( x ) & = \int \mathd zK_{< i-1,x} ( z ) \Pi_{x} f^{\pi}_{x} ( K_{i,x} ) + \int \mathd zK_{< i-1,x} ( z ) \Pi_{x} (\Gamma_{x,z} f^{\pi}_{z} -f^{\pi}_{x} ) ( K_{i,x} ) \\
   & = \Pi_{x} f^{\pi}_{x} ( K_{i,x} ) +O ( 2^{-i \gamma} )
\end{align*}
for all $i \geqslant 1$, where we used that $\int \mathd z K_{<i-1,x}(z) =1$, and where the estimate for the second integral follows from arguments similar to
those used in Lemma~\ref{lemma:rec-equivalence}. Now define the operator
\[
   P f^{\pi} = P(f^\pi, \Pi) = \sum_{i \geqslant 0} P_{i} f^{\pi}
\]
and note that this always gives a well defined distribution since every $P_{i} f^{\pi}$ is spectrally supported in an annulus $2^{i} \CA$. In the particular case where $\Pi_{z} f^{\pi}_{z} ( z' ) =u(z) v ( z' )$, we get $P_{i} ( f^{\pi} ) = S_{i-1} u \Delta_{i} v$ and $P f^{\pi} =u \lpara v$, which justifies the claim that $P$ is a generalization of
the usual paraproduct.

The following lemma links $P f^{\pi}$ with the local behavior of the
distribution $\Pi_{x} f^{\pi}_{x}$ around the point $x$.

\begin{lemma}\label{lem:Ti-bound}
  Let $\gamma \in \R$ and $f^\pi \in \mathcal{D}^{\gamma} ( \mathcal{T}, \Gamma )$ and set
  \[ T_{i} f^{\pi} ( x ) =P f^{\pi} ( K_{i,x} ) - \Pi_{x} f^{\pi}_{x} (
     K_{i,x} ) \]
  for all $i \geqslant 0$. Then $\| T_{i} f^{\pi} \|_{L^{\infty}} \lesssim 2^{-i \gamma}$.
\end{lemma}

\begin{proof}
  Observe that
  \[ 
     P f^{\pi} ( K_{i,x} ) = \sum_{j} ( P_{j} f^{\pi} ) ( K_{i,x} ) = \sum_{j: j \sim i} \int \mathd y \mathd zK_{i,x} ( y ) K_{< j-1,y} ( z ) \Pi_{z} f^{\pi}_{z} ( K_{j,y} )
  \]
  and also that, since $\sum_{j:j \sim i} K_{i} \ast K_{j} =K_{i}$,
  \[
     \Pi_{x} f^{\pi}_{x} ( K_{i,x} ) = \sum_{j: j \sim i} \int \mathd y K_{i,x} ( y ) \Pi_{x} f^{\pi}_{x} ( K_{j,y} ).
  \]
  Using the decomposition $\Pi_{z} f^{\pi}_{z} ( K_{j,y} ) - \Pi_{x}
  f^{\pi}_{x} ( K_{j,y} ) = \Pi_{y} \Gamma_{y,z} ( f^{\pi}_{z} - \Gamma_{z,x} f^{\pi}_{x} )
  ( K_{j,y} )$, we further have
  \begin{align*}
     T_{i} f^{\pi} ( x ) & =P f^{\pi} ( K_{i,x} ) - \Pi_{x} f^{\pi}_{x} (K_{i,x} ) \\
     & = \sum_{j:j \sim i} \int \mathd y \mathd zK_{i,x} ( y ) K_{< j-1,y} ( z ) \Pi_{y} \Gamma_{y,z} ( f^{\pi}_{z} - \Gamma_{z,x} f^{\pi}_{x} ) ( K_{j,y} )
  \end{align*}
  from which the claimed bound can be shown to hold. Indeed, using the fact that $f^{\pi} \in \mathcal{D}^{\gamma}( \mathcal{T}, \Gamma )$ we obtain
  \begin{align*}
     &\sum_{j:j \sim i} \left\lvert \int \mathd y \mathd z K_{i,x} ( y ) K_{< j-1,y} ( z ) \Pi_{y} \Gamma_{y,z} ( f^{\pi}_{z} - \Gamma_{z,x}f^{\pi}_{x} ) ( K_{j,y} ) \right\rvert \\
     &\hspace{50pt} \lesssim \sum_{j:j \sim i} \sum_{\beta<\gamma} \int \mathd y \mathd z \left\lvert K_{i,x} ( y ) K_{< j-1,y} ( z ) \right\rvert \|\Gamma_{y,z} ( f^{\pi}_{z} - \Gamma_{z,x}f^{\pi}_{x} )\|_\beta 2^{-j\beta} \\
     &\hspace{50pt} \lesssim \sum_{j:j \sim i} \sum_{\beta<\gamma} \sum_{\alpha: \beta<\alpha<\gamma} \int \mathd y \mathd z \left\lvert K_{i,x} ( y ) K_{< j-1,y} ( z ) \right\rvert |y-z|^{\alpha-\beta} |z-x|^{\gamma-\alpha} 2^{-j\beta}.
  \end{align*}
  Now it suffices to note that $|z-x|^{\gamma-\alpha} = |(z-y)+ (y-x)|^{\gamma-\alpha}$ to complete the proof.
\end{proof}

\begin{lemma}\label{lem:T-hoelder}
  Let $\gamma >0$ and $f^\pi \in \mathcal{D}^{\gamma} ( \mathcal{T}, \Gamma )$ and define
  \[
     T f^{\pi} ( x ) = \sum_{i} T_{i} f^{\pi} ( x ) = \sum_{i} [P f^{\pi} ( K_{i,x} ) -  \Pi_{x} f^{\pi}_{x} ( K_{i,x} )].
  \]
  Then $T f^{\pi} \in \CC^{\gamma}$.
\end{lemma}

\begin{proof}
  According to Lemma~\ref{lem:Ti-bound}, the series converges in $L^\infty$. Let us analyze its regularity. Consider $\Delta_{j} T f^{\pi} =
  \sum_{i} \Delta_{j} T_{i} f^{\pi}$ and split the sum into two contributions,
  $\Delta_{j} T f^{\pi} = \Delta_{j} T_{\leqslant j+1}  f^{\pi} + \Delta_{j}
  T_{>j+1}  f^{\pi}$, where $T_{\leqslant j+1}  f^{\pi} = \sum_{i \leqslant
  j+1} T_{i}  f^{\pi}$ and $T_{>j+1}  f^{\pi} =T f^{\pi}  -T_{\leqslant j+1}
  f^{\pi}$. For the second term we have
  \[ \| \Delta_{j} T_{>j+1}  f^{\pi} \|_{L^{\infty}} \leqslant \sum_{i>j+1} \|
     \Delta_{j} T_{i} f^{\pi} \|_{L^{\infty}} \lesssim \sum_{i>j+1} \| T_{i}
     f^{\pi} \|_{L^{\infty}} \lesssim 2^{-j \gamma} . \]
  For the first one we proceed as follows. Note that $T_{\leqslant j+1}
  f^{\pi} ( x ) = P f^{\pi} (K_{\leqslant j+1,x} ) - \Pi_{x} f^{\pi}_{x} ( K_{\leqslant j+1,x} )$, so that using $K_{j} \ast K_{\leqslant
  j+1} =K_{j}$ we get
  \begin{align*}
     \Delta_{j} T_{\leqslant j+1}  f^{\pi}(x) & = P f^{\pi} ( K_{j,x} ) - \int \mathd y K_{j,x} ( y ) \Pi_{y} f^{\pi}_{y} ( K_{\leqslant j+1,y} ) \\
     & = P f^{\pi} ( K_{j,x} )  - \Pi_{x} f^{\pi}_{x} ( K_{j,x} ) - \int \mathd y K_{j,x} ( y ) \Pi_{y} ( f^{\pi}_{y} - \Gamma_{y,x} f^{\pi}_{x} ) ( K_{\leqslant j+1,y} ) \\
     & =T_{j} f^{\pi} ( x ) - \int \mathd y K_{j,x} ( y ) \Pi_{y} ( f^{\pi}_{y} - \Gamma_{y,x} f^{\pi}_{x} ) ( K_{\leqslant j+1,y} ),
  \end{align*}
  where in the last line we have used the definition of $T_{j} f^{\pi}$. Now
  \[
     | \Pi_{y} ( f^{\pi}_{y} - \Gamma_{y,x} f^{\pi}_{x} ) ( K_{\leqslant j+1,y} ) | \lesssim \sum_{\beta < \gamma} | y-x |^{\gamma - \beta} 2^{-j \beta},
  \]
  so that $\| \Delta_{j} T f^{\pi} -T_{j} f^{\pi} \|_{L^{\infty}} \lesssim 2^{-j
  \gamma}$. By Lemma~\ref{lem:Ti-bound} this implies that $\| \Delta_{j} T f^{\pi} \|_{L^{\infty}}
  \lesssim 2^{-j \gamma}$ and thus the proof is complete.
\end{proof}

Finally we are able to recover (under stronger assumptions and in the setting of Euclidean scaling) the reconstruction theorem \cite{Hairer2013a}, Theorem~3.10, one of the main results of the theory of regularity structures:

\begin{theorem}
  The reconstruction operator $\mathcal R$ exists for all $\gamma \in \R \setminus \{0\}$. If
  $\gamma >0$ we have $\mathcal R=P-T$ while if $\gamma <0$ we can take $\mathcal R=P$.
\end{theorem}

\begin{proof}
  In case $\gamma >0$ set $\mathcal R f^{\pi} =P f^{\pi} -T f^{\pi}$ and observe
  that
  \begin{align*}
     \mathcal R f^{\pi} ( K_{< i,x} ) - \Pi_{x} f^{\pi}_{x} ( K_{< i,x} ) & = P f^{\pi} ( K_{< i,x} ) - \Pi_{x} f^{\pi}_{x} ( K_{< i,x} ) -T f^{\pi} ( K_{< i,x} ) \\
     & =T f^{\pi} ( x ) - \sum_{j \geqslant i} T_j f^\pi(x) -T f^{\pi} ( K_{< i,x} ) \\
     & = \sum_{j\geqslant i} ( \Delta_{j} T f^{\pi} ( x ) -T_{j} f^{\pi} ( x) ).
  \end{align*}
  With the bounds of Lemma~\ref{lem:Ti-bound} and Lemma~\ref{lem:T-hoelder} we conclude that
  \[
     | \mathcal R f^{\pi} ( K_{< i,x} ) - \Pi_{x} f^{\pi}_{x} ( K_{< i,x} ) | \lesssim 2^{-i \gamma},
  \]
  which implies that $\mathcal R$ is the reconstruction operator. If $\gamma <0$, just set $\mathcal R=P$ and observe that
  \[
     | \mathcal R f^{\pi} ( K_{< i,x} ) - \Pi_{x} f^{\pi}_{x} ( K_{<i,x} ) | \lesssim \sum_{j < i} | T_{j} f^{\pi} ( x ) | \lesssim  \sum_{j < i} 2^{-j \gamma} \lesssim 2^{-i \gamma},
  \]
  which shows that also in this case $\mathcal R$ is an admissible reconstruction operator.
\end{proof}

For $\gamma > 0$, we could say that a distribution $f$ is \emph{paracontrolled} by $\Pi$ if there exist
$f^{\pi} \in \mathcal{D}^{\gamma}( \mathcal{T}, \Gamma )$ and $f^{\sharp} \in \mathcal{C}^{\gamma}$
such that
\[
   f=P ( f^{\pi} , \Pi ) +f^{\sharp};
\]
in that case we write $f \in \mathcal{Q}^{\gamma}$. In particular, every
modelled distribution is a paracontrolled distribution since the
reconstruction operator $\mathcal{R}$ delivers a map
\[
   f^{\pi} \in \mathcal{D}^{\gamma}( \mathcal{T}, \Gamma ) \longmapsto \mathcal{R}f^{\pi} =P (f^{\pi} , \Pi ) -T f^{\pi} \in \mathcal{Q}^{\gamma}.
\]
Moreover, every paracontrolled distribution can be decomposed into ``slices'', each of which has its natural regularity. More precisely, let us write $\tau^\alpha$ for the component of $\tau \in T$ in $T_\alpha$, for $\alpha < \gamma$. Then the distribution $P(f^\pi, \Pi)$ is given as
\begin{align*}
   P(f^\pi, \Pi) & = \sum_{i \geqslant 0} P_i f^\pi =\sum_{i \geqslant 0} \int \mathd zK_{< i-1,x} ( z ) \Pi_{z} f^{\pi}_{z} ( K_{i,x} ) \\
   & = \sum_{\alpha < \gamma} \Big( \sum_{i \geqslant 0} \int \mathd zK_{< i-1,x} ( z ) \Pi_{x} (\Gamma_{x,z} f^{\pi}_z)^\alpha ( K_{i,x} ) \Big).
\end{align*}
Now
\[
   \|\Gamma_{x,z} f^{\pi}_z \|_\alpha \lesssim \sum_{\beta: \alpha \leqslant \beta < \gamma} |x-z|^{\beta-\alpha} \|f^\pi_x\|_\beta \lesssim 1 + |x-z|^{\gamma-\alpha},
\]
and Lemma~\ref{lem:Pi acting on Schwartz} shows that $|\Pi_x \tau^\alpha(K_{i,x} )| \lesssim 2^{-i\alpha} \|\tau\|_\alpha$ for all $\tau \in T$, $i \geqslant -1$. Combining these estimates with the fact that $\int \mathd zK_{< i-1,x} ( z ) \Pi_{z} f^{\pi,\alpha}_{z} ( K_{i,x} )$ is spectrally supported in an annulus $2^i \CA$, we deduce that
\[
   \sum_{i \geqslant 0} \int \mathd zK_{< i-1,x} ( z ) \Pi_{x} (\Gamma_{x,z} f^{\pi}_z)^\alpha ( K_{i,x} ) \in \CC^\alpha.
\]
In particular, if $r = |\inf A|$, then every paracontrolled distribution is in $\CC^{-r}$.

Note also
that the paraproduct vanishes on constant and polynomial components of the model. Indeed, if $\tau$ is such that $\Pi_{x} \tau ( y ) = ( y-x )^{\mu}$ for some $\mu \in \N^d$, then $P ( \cdot , \tau ) =0$ since $( \Pi_{x} \tau ) ( K_{i,x}) =0$ for any $i \geqslant 0$.

\appendix

\section{Besov spaces and paraproducts}\label{sec:para}

\subsection{Littlewood-Paley theory and Besov spaces}\label{sec:littlewood}

In the following, we describe the concepts from Littlewood--Paley theory which are necessary for our analysis, and we recall the definition and some properties of Besov spaces. For a general introduction to Littlewood--Paley theory, Besov spaces, and paraproducts, we refer to the nice book of Bahouri, Chemin, and Danchin {\cite{Bahouri2011}}.

Littlewood--Paley theory allows for an efficient way of characterizing the regularity of functions and distributions. It relies on the decomposition of an arbitrary distribution into a series of smooth functions whose Fourier transforms have localized support.

Let $\chi \nocomma , \rho \in \CD$ be nonnegative radial functions on $\mathbb{R}^{d}$, such that
\begin{enumerateroman}
  \item the support of $\chi$ is contained in a ball and the support of $\rho$ is contained in an annulus;
  
  \item $\chi ( z ) + \sum_{j \geqslant 0} \rho ( 2^{-j} z ) =1$ for all $z \in \mathbb{R}^{d}$;
  
  \item \label{def:dyadic partition3}$\tmop{supp} ( \chi ) \cap \tmop{supp} (\rho ( 2^{-j} \cdummy ) ) = \emptyset$ for $j \geqslant 1$ and $\tmop{supp}( \rho ( 2^{-i} \cdummy ) ) \cap \tmop{supp} ( \rho ( 2^{-j} \cdummy ) ) =  \emptyset$ for $| i-j | >1$.
\end{enumerateroman}
We call such $( \chi , \rho )$ \tmtextit{dyadic partition of unity}, and we frequently employ the notation
\[
   \rho_{-1} = \chi \qquad \text{and} \qquad \rho_{j} = \rho (2^{-j} \cdummy ) \text{ for } j \geqslant 0.
\]
For the existence of dyadic partitions of unity see {\cite{Bahouri2011}}, Proposition 2.10. The Littlewood--Paley blocks are now defined as
\[
   \Delta_{-1} u= \CF^{-1} \left( \chi \CF u \right) = \CF^{-1} \left( \rho_{-1} \CF u \right) \quad \text{and} \quad \Delta_{j} u= \CF^{-1} \left( \rho_{j} \CF u \right) \text{ for }  j \geqslant 0. 
\]
Then $\Delta_{j} u=K_{j} \ast u$, where $K_{j} = \CF^{-1} \rho_{j}$, and in particular all $\Delta_{j} u$, $j \geqslant -1$, are smooth functions. We also use the notation
\[
   S_{j} u= \sum_{i \leqslant j-1} \Delta_{i} u.
\]
It is easy to see that $u =  \sum_{j \geqslant -1} \Delta_{j} u= \lim_{j \rightarrow \infty} S_{j} u$ for every $u \in \CS'$.

For $\alpha \in \mathbb{R}$, the H{\"o}lder-Besov space $\CC^{\alpha}$ is given by $\CC^{\alpha} =B^{\alpha}_{\infty , \infty} (\mathbb{R}^{d},\mathbb{R}^{n} )$, where for $p,q \in [1, \infty ]$ we define
\[
   B^{\alpha}_{p,q} (\mathbb{R}^{d} ,\mathbb{R}^{n} ) = \bigg\{ u \in \CS'(\mathbb{R}^{d} ,\mathbb{R}^{n} ):\|u\|_{B^{\alpha}_{p,q}} = \bigg(\sum_{j \geqslant -1} (2^{j \alpha} \| \Delta_{j} u\|_{L^{p}} )^{q} \bigg)^{1/q} < \infty \bigg\},
\]
with the usual interpretation as $\ell^{\infty}$ norm in case $q= \infty$. The $\lVert \cdummy \rVert_{L^{p}}$ norm is taken with respect to Lebesgue measure on $\mathbb{R}^{d}$. While the norm $\lVert \cdummy \rVert_{B^{\alpha}_{p,q}}$ depends on the dyadic partition of unity $( \chi , \rho )$, the space $B^{\alpha}_{p,q}$ does not, and any other dyadic partition of unity corresponds to an equivalent norm. We write $\lVert \cdummy \rVert_{\alpha}$ instead of $\lVert \cdummy \rVert_{B^{\alpha}_{\infty , \infty}}$.

If $\alpha \in (0, \infty ) \backslash \mathbb{N}$, then $\CC^{\alpha}$ is the space of $\lfloor \alpha \rfloor$ times differentiable functions, whose partial derivatives up to order $\lfloor \alpha \rfloor$ are bounded, and whose partial derivatives of order $\lfloor \alpha \rfloor$ are ($\alpha - \lfloor \alpha \rfloor$)-H{\"o}lder continuous (see p. 99 of {\cite{Bahouri2011}}). Note however that for $k \in \mathbb{N}$ the H{\"o}lder-Besov space $\CC^{k}$ is strictly larger than $C^{k}_{b}$.

We will use without comment that $\lVert \cdummy \rVert_{\alpha} \leqslant \lVert \cdummy \rVert_{\beta}$ for $\alpha \leqslant \beta$, that $\lVert \cdummy \rVert_{L^{\infty}} \lesssim \lVert \cdummy \rVert_{\alpha}$ for $\alpha >0$, and that $\lVert \cdummy \rVert_{\alpha} \lesssim \lVert \cdummy \rVert_{L^{\infty}}$ for $\alpha \leqslant 0$. We will also use that $\|S_{j} u\|_{L^{\infty}} \lesssim 2^{j \alpha} \|u\|_{\alpha}$ for $\alpha <0$ and $u \in \CC^{\alpha}$.

We denote by $\CC^{\alpha}_{\text{\tmop{loc}}}$ the set of all distributions $u$ such that $\varphi u \in \CC^{\alpha}$ for all $\varphi \in \CD$. If the difference $\varphi(u_n - u)$ converges to $0$ in $\CC^\alpha$ for all $\varphi \in \CD$, then we say that $(u_n)$ converges to $u$ in $\CC^\alpha_{\tmop{loc}}$.

The following Bernstein inequalities are tremendously useful when dealing with functions with compactly supported Fourier transform.

\begin{lemma}[Lemma 2.1 of {\cite{Bahouri2011}}]\label{lemma:Bernstein}
   Let $\CA$ be an annulus and let $\CB$ be a ball. For any $k \in \mathbb{N}$, $\lambda >0$, and $1 \leqslant p \leqslant q \leqslant \infty$ we have that
   \begin{enumeratenumeric}
      \item if $u \in L^{p} ( \mathbb{R}^{d} )$ is such that $\tmop{supp} ( \CF u ) \subseteq \lambda \CB$, then
         \[
            \max_{\mu \in \mathbb{N}^{d} : | \mu | = k} \| \partial^{\mu} u \|_{L^{q}} \lesssim_{k} \lambda^{k+d \left( \frac{1}{p} - \frac{1}{q} \right)} \| u \|_{L^{p}};
         \]
      \item if $u \in L^{p} ( \mathbb{R}^{d} )$ is such that $\tmop{supp} (\CF u ) \subseteq \lambda \CA$, then
         \[
            \lambda^{k} \| u \|_{L^{p}} \lesssim_{k} \max_{\mu \in \mathbb{N}^{d}: | \mu | = k} \| \partial^{\mu} u \|_{L^{p}} .
         \]
  \end{enumeratenumeric}
\end{lemma}

For example, it is a simple consequence of the Bernstein inequalities that $\| \mathD^k u \|_{\alpha-k} \lesssim \| u \|_\alpha$ for all $\alpha \in \R$ and $k \in \N$.

We point out that everything above and everything that follows can (and will) be applied to distributions on the torus. More precisely, let $\CD'(\mathbb{T}^{d} )$ be the space of distributions on $\mathbb{T}^{d}$. Any $u \in \CD' (\mathbb{T}^{d} )$ can be interpreted as a periodic tempered distribution on $\mathbb{R}^{d}$, with frequency spectrum contained in $\mathbb{Z}^{d}$ -- and vice versa. For details see~\cite{Schmeisser1987}, Chapter 3.2. In particular, $\Delta_{j} u$ is a periodic smooth function, and therefore $\| \Delta_{j} u\|_{L^{\infty}} = \| \Delta_{j} u\|_{L^{\infty} (\mathbb{T}^{d} )}$. In other words, we can define
\[
    \CC^\alpha(\mathbb{T}^d) = \{ u \in \CC^\alpha: u \text{ is } (2\pi)-\text{periodic}\}
\]
for $\alpha \in \R$. However, for $p \neq \infty$ this definition is not very useful, because no nontrivial periodic function is in $L^p$ for $p < \infty$. Therefore, general Besov spaces on the torus are defined as
\[
   B^\alpha_{p,q}(\mathbb{T}^d) = \biggl\{ u \in \mathcal{D}' (\mathbb{T}^d) : \lVert u\rVert_{B^{\alpha}_{p, q}(\mathbb{T}^d)} = \biggl( \sum_{j \geqslant - 1} (2^{j \alpha} \lVert \Delta_j u \rVert_{L^p(\mathbb{T}^d)})^q\biggr)^{1/q} < \infty \biggr\},
\]
where we set
\[
   \Delta_j u =  (2\pi)^{-d} \sum_{k \in \Z^d} e^{\imath \langle k, x\rangle} \rho_j(k) (\CF_{\mathbb{T}^d}{u})(k) = \CF_{\mathbb{T}^d}^{-1} (\rho_j \CF_{\mathbb{T}^d}{u}),
\]
and where $\CF_{\mathbb{T}^d}$ and $\CF_{\mathbb{T}^d}^{-1}$ denote Fourier transform and inverse Fourier transform on the torus. The two definitions are compatible: we have $\CC^\alpha(\mathbb{T}^d) = B^\alpha_{\infty, \infty}(\mathbb{T}^d)$. 
Strictly speaking we will not work with $B^\alpha_{p,q}(\mathbb{T}^d)$ for $(p,q) \neq (\infty, \infty)$. But we will need the Besov embedding theorem on the torus.
\begin{lemma}\label{l:besov embedding torus}
   Let $1\leqslant p_1 \leqslant p_2 \leqslant \infty$ and $1\leqslant q_1 \leqslant q_2 \leqslant \infty$, and let $\alpha \in \R$. Then $B^{\alpha}_{p_1,q_1}(\mathbb{T}^d)$ is continuously embedded in $B^{\alpha - d(1/p_1 - 1/p_2)}_{p_2,q_2}(\mathbb{T}^d)$, and $B^{\alpha}_{p_1,q_1}(\mathbb{R}^d)$ is continuously embedded in $B^{\alpha - d(1/p_1 - 1/p_2)}_{p_2,q_2}(\mathbb{R}^d)$.
\end{lemma}
For the embedding theorem on $\R^d$ see~\cite{Bahouri2011}, Proposition~2.71. The result on the torus can be shown using the same arguments, see for example~\cite{Chemin2006}. In both cases, the proof is based on the Bernstein inequalities, Lemma~\ref{lemma:Bernstein}.

The following characterization of Besov regularity for functions which can be decomposed into pieces that are well localized in Fourier space will be useful below.

\begin{lemma}\label{lem: Besov regularity of series}(Lemmas~2.69 and~2.84 of~{\cite{Bahouri2011}})
   \begin{enumeratenumeric}
      \item Let $\CA$ be an annulus, let $\alpha \in \mathbb{R}$, and let $(u_{j})$ be a sequence of smooth functions such that $\CF u_{j}$ has its support in $2^{j} \CA$, and such that $\| u_{j} \|_{L^{\infty}} \lesssim 2^{-j \alpha}$ for all $j$. Then
         \[
            u= \sum_{j \geqslant -1} u_{j} \in \CC^{\alpha} \qquad  \tmop{and} \qquad \| u \|_{\alpha} \lesssim \sup_{j \geqslant -1} \{ 2^{j \alpha} \| u_{j} \|_{L^{\infty}} \} .
         \]
      \item Let $\CB$ be a ball, let $\alpha >0$, and let $( u_{j} )$ be a sequence of smooth functions such that $\CF u_{j}$ has its support in $2^{j} \CB$, and such that $\| u_{j} \|_{L^{\infty}} \lesssim 2^{-j \alpha}$ for all $j$. Then
         \[
            u= \sum_{j \geqslant -1} u_{j} \in \CC^{\alpha} \qquad \tmop{and} \qquad \| u \|_{\alpha} \lesssim \sup_{j \geqslant -1} \{ 2^{j \alpha} \| u_{j} \|_{L^{\infty}} \} .
         \]
   \end{enumeratenumeric}
\end{lemma}

\begin{proof}
  It $\CF u_{j}$ is supported in $2^{j} \CA$, then $\Delta_{i} u_{j} \neq 0$ only for $i \sim j$. Hence, we obtain
  \[
     \| \Delta_{i} u \|_{L^{\infty}} \leqslant \sum_{j:j \sim i} \| \Delta_{i} u_{j} \|_{L^{\infty}} \leqslant \sup_{k \geqslant -1} \{ 2^{k \alpha} \| u_{k} \|_{L^{\infty}} \} \sum_{j:j \sim i} 2^{-j  \alpha} \simeq \sup_{k \geqslant -1} \{ 2^{k \alpha} \| u_{k} \|_{L^{\infty}} \} 2^{-i \alpha}.
  \]
  If $\CF u_{j}$ is supported in $2^{j} \CB$, then $\Delta_{i} u_{j} \neq 0$ only for $i \lesssim j$. Therefore,
   \[
      \| \Delta_{i} u \|_{L^{\infty}} \leqslant \sum_{j:j \gtrsim i} \| \Delta_{i} u_{j} \|_{L^{\infty}} \leqslant \sup_{k \geqslant -1} \{ 2^{k \alpha} \| u_{k} \|_{L^{\infty}} \} \sum_{j:j \gtrsim i} 2^{-j  \alpha}  \lesssim \sup_{k \geqslant -1} \{ 2^{k \alpha} \| u_{k} \|_{L^{\infty}} \} 2^{-i \alpha},
   \]
   using $\alpha >0$ in the last step.
\end{proof}

\subsection{Linear operators acting on Besov spaces}

Here we discuss the action of some important linear operators on Besov spaces. We start with the rescaling of the spatial variable:

\begin{lemma}\label{lemma:scaling}
  For $\lambda >0$ and $u \in \CS'$ we define the scaling transformation $\Lambda_{\lambda} u ( \cdot ) =u ( \lambda \cdot )$. Then
  \[
     \| \Lambda_{\lambda} u \|_{\alpha} \lesssim \max\{1, \lambda^{\alpha} \} \| u \|_{\alpha}
  \]
  for all $\alpha \in \mathbb{R} \setminus \{ 0 \}$ and all $u \in \CC^{\alpha}$.
\end{lemma}

\begin{proof}
  Let $u \in \CC^{\alpha}$ and let $\Lambda_{\lambda} u ( x ) =u ( \lambda x )$ for some $\lambda >0$. Note that $\Lambda_{\lambda} \mathD =  \lambda^{-1} \mathD \Lambda_{\lambda}$, and therefore $\Lambda_{\lambda} \Delta_{j} u= \Lambda_{\lambda} \rho ( 2^{-j} \mathD ) u= \rho ( 2^{-j} \lambda^{-1} \mathD ) \Lambda_{\lambda} u$, which implies that the Fourier transform of $\Lambda_{\lambda} \Delta_{j} u$ is supported in the annulus $\lambda 2^{j} \CA$ (where $\CA$ is the annulus in which $\rho$ is supported). In particular, if $k \geqslant 0$, we have $\Delta_{k} \Lambda_{\lambda} \Delta_{j} u \neq 0$ only if $2^{k} \sim \lambda 2^{j}$. Thus, there exist $a,b>0$ such that
  \begin{align*}
     \| \Delta_{k} \Lambda_{\lambda} u \|_{L^{\infty}} & \lesssim \sum_{j:a2^{k} \leqslant \lambda 2^{j} \leqslant b2^{k}} \| \Delta_{k} \Lambda_{\lambda} \Delta_{j} u \|_{L^{\infty}} \lesssim \sum_{j:a2^{k} \leqslant \lambda 2^{j} \leqslant b2^{k}} \| \Delta_{j} u \|_{L^{\infty}} \\
     & \lesssim \| u \|_{\alpha} \sum_{j:a2^{k} \leqslant \lambda 2^{j} \leqslant b2^{k}} 2^{- \alpha j} \lesssim \| u \|_{\alpha} \lambda^{\alpha} 2^{- \alpha k} 
  \end{align*}
  for all $k \geqslant 0$. For $k=-1$ we can simply bound
  \[
     \| \Delta_{-1} \Lambda_{\lambda} u \|_{L^{\infty}} \lesssim \sum_{j: \lambda 2^{j} \lesssim 1} \| \Delta_{k} \Lambda_{\lambda} \Delta_{j} u \|_{L^{\infty}} \lesssim \| u \|_{\alpha} \sum_{j: \lambda 2^{j} \lesssim 1} 2^{- \alpha j} \lesssim \| u \|_{\alpha} \max\{1, \lambda^{\alpha} \}.
  \]
\end{proof}

Next, we are concerned with the action of Fourier multipliers on Besov spaces.

\begin{lemma}\label{lemma:fourier multiplier smoothing}
  Let $\varphi$ be a continuous function, such that $\varphi$ is infinitely differentiable everywhere except possibly at 0, and such that $\varphi$ and all its partial derivatives decay faster than any rational function at infinity. Assume also that $\CF \varphi \in L^{1}$. Then
  \[
     \| \varphi ( \varepsilon \mathD ) u \|_{\alpha + \delta} \lesssim \varepsilon^{- \delta} \| u \|_{\alpha} \hspace{2em} \tmop{and} \hspace{2em} \| \varphi ( \varepsilon \mathD ) u \|_{\delta}  \lesssim \varepsilon^{- \delta} \| u \|_{L^{\infty}} .
  \]
  for all $\varepsilon \in (0,1]$, $\delta \geqslant 0$, $\alpha \in \mathbb{R}$, and $u \in \CS'$.
\end{lemma}

\begin{proof}
  Let $\psi \in \CD$ with support in an annulus be such that $\psi \rho = \rho$, where $( \chi , \rho )$ is our dyadic partition of unity. Then we have for $j \geqslant 0$ that
  \[
     \varphi ( \varepsilon \mathD ) \Delta_{j} u= \left[ \CF^{-1} ( \varphi (\varepsilon \cdummy ) \psi ( 2^{-j} \cdummy ) ) \right] \ast \Delta_{j}  u,
  \]
  and therefore Young's inequality implies
  \begin{align*}
     \| \varphi ( \varepsilon \mathD ) \Delta_{j} u \|_{L^{\infty}} & \lesssim \left\| \CF^{-1} ( \varphi ( \varepsilon \cdummy ) \psi ( 2^{-j} \cdummy ) ) \right\|_{L^{1}} 2^{-j \alpha} \| u \|_{\alpha} \\
     & = \left\| \CF^{-1} ( \varphi ( 2^{j} \varepsilon \cdummy ) \psi ) \right\|_{L^{1}} 2^{-j \alpha} \| u \|_{\alpha}.
  \end{align*}
  Hence, it suffices to show that $\left\| \CF^{-1} ( \varphi ( 2^{j} \varepsilon \cdummy ) \psi ) \right\|_{L^{1}} \lesssim \varepsilon^{- \delta} 2^{-j \delta}$. But
  \begin{align*}
     \left\| \CF^{-1} ( \varphi ( 2^{j} \varepsilon \cdummy ) \psi ) \right\|_{L^{1}} & \lesssim \left\| ( 1+ | \cdot |^{2} )^{d} \CF^{-1} ( \varphi ( 2^{j} \varepsilon \cdummy ) \psi ) \right\|_{L^{\infty}}  \\
     & \lesssim \left\| \CF^{-1} ( ( 1+ \Delta )^{d} ( \varphi ( 2^{j} \varepsilon \cdummy ) \psi ) ) \right\|_{L^{\infty}} \\
     & \lesssim \| ( 1+ \Delta )^{d} ( \varphi ( 2^{j} \varepsilon \cdummy ) \psi ) \|_{L^{1}} \\
     & \lesssim ( 1+2^{j} \varepsilon )^{2d} \max_{\mu \in \mathbb{N}^{d} : | \mu | \leqslant 2d} \| \partial^{\mu} \varphi ( 2^{j}  \varepsilon \cdot ) \|_{L^{\infty} ( \tmop{supp} ( \psi ) )} .
  \end{align*}
  By assumption, $\varphi$ is smooth away from 0, and $\varphi$ and all its partial derivatives decay faster than any rational function at infinity. Thus, we get
  \[
     \sup_{|\mu| \leqslant 2d} \sup_{x \geqslant 1} ( 1+ | x | )^{\delta +2d} | \partial^{\mu} \varphi ( x ) | \lesssim 1.
  \]
  Since $\tmop{supp} ( \psi )$ is bounded away from $0$, there exists a minimal $j_{0} \in \mathbb{N}$, such that $2^{j_{0}} \varepsilon | x | \geqslant 1$ for all $x \in \tmop{supp} ( \psi )$, and therefore
  \[
     \left\| \CF^{-1} ( \varphi ( 2^{j} \varepsilon \cdummy ) \psi ) \right\|_{L^{1}} \lesssim ( 1+2^{j} \varepsilon )^{2d} ( 1+2^{j}\varepsilon )^{- \delta -2d} = ( 1+2^{j} \varepsilon )^{- \delta} \leqslant 2^{-j \delta} \varepsilon^{- \delta}
  \]
  for all $j \geqslant j_{0}$. On the other side, we get for $j \leqslant j_{0}$
  \begin{align*}
     \| \varphi ( \varepsilon \mathD ) \Delta_{j} u \|_{L^{\infty}} &\lesssim \| \CF^{-1} ( \varphi ( \varepsilon \cdummy ) ) \|_{L^{1}} \| \Delta_{j} u \|_{L^{\infty}} \lesssim 2^{-j \alpha} \| u \|_{\alpha} = ( \varepsilon 2^{j} )^{\delta} \varepsilon^{- \delta} 2^{-j ( \alpha +  \delta )} \| u \|_{\alpha} \\
     & \leqslant ( \varepsilon 2^{j_{0}} )^{\delta} \varepsilon^{- \delta} 2^{-j ( \alpha + \delta )} \| u \|_{\alpha} \lesssim \varepsilon^{- \delta} 2^{-j ( \alpha + \delta )} \| u \|_{\alpha},
  \end{align*}
  where we used that $\delta \geqslant 0$. The estimate for $u \in L^{\infty}$ follows from the same arguments.
\end{proof}

\begin{remark}
   If the support of $\CF u$ has a ``hole'' at 0, that is if there exists a ball $\CB$ centered at 0 such that $\CF u$ is supported outside of $\CB$, then the estimates of Lemma~\ref{lemma:fourier multiplier smoothing} hold uniformly in $\varepsilon > 0$ and not just for $\varepsilon \in (0,1]$. This is an immediate consequence of the previous proof.
\end{remark}

As an application, we derive the smoothing properties of the heat kernel generated by the fractional Laplacian.

\begin{lemma}\label{lemma:heat flow smoothing}
   Let $\sigma \in ( 0,1 ]$, let $- ( - \Delta )^{\sigma}$ be the fractional Laplacian with periodic boundary conditions on $\mathbb{T}^{d} \nosymbol$, and let $( P_{t} )_{t \geqslant 0}$ be the semigroup generated by $- ( - \Delta )^{\sigma}$. Then for all $T>0$, $t \in (0,T]$, $\alpha \in \R$, $\delta \geqslant 0$, and $u \in \CS'$ we have
  \[
     \|P_{t} u\|_{\alpha + \delta} \lesssim_T t^{- \delta / ( 2 \sigma )} \|u\|_{\alpha} \hspace{2em} \tmop{and} \hspace{2em} \| P_{t} v \|_{\delta} \lesssim_T t^{- \delta / ( 2 \sigma )} \| v \|_{L^{\infty}} .
  \]
  If $\CF u$ is supported outside of a ball centered at 0, then these estimates are uniform in $t > 0$ and not just in $t \in (0,T]$.
\end{lemma}

\begin{proof}
  The semigroup is given by $P_{t} = \varphi ( t^{1/ ( 2 \sigma )} \mathD )$ with $\varphi ( z ) =e^{- | z |^{2 \sigma}}$. Now $\varphi$ and its derivatives decay faster than any rational function at $\infty$. For $\sigma \leqslant 1$, $\CF \varphi$ is the density of a symmetric $2 \sigma$-stable random variable, and therefore in $L^{1}$. For $\sigma > 1$ it is easily shown that $(1+|\cdot|^{d+1}) \CF \varphi$ is bounded, and therefore in $L^1$. Thus, the estimates follow from Lemma~\ref{lemma:fourier multiplier smoothing}.
\end{proof}


\begin{lemma}\label{lemma:semigroup hoelder in time}
   Let $\sigma$ and $(P_t)_{t \geqslant 0}$ be as in Lemma~\ref{lemma:heat flow smoothing}. Let $\alpha \in \mathbb{R}$, $\beta \in (0,1)$, and let $u \in \CC^{\alpha}$. Then we have for all $t \geqslant 0$
  \[
     \| ( P_{t} - \tmop{Id} ) u \|_{L^{\infty}} \lesssim t^{\beta / ( 2 \sigma )} \| u \|_{\beta}.
  \]
\end{lemma}

\begin{proof}
  For the uniform estimate of $( P_{t} - \tmop{Id} ) u$, we write $P_{t} - \tmop{Id}$ as convolution operator: if $\varphi ( z ) =e^{- | z |^{2\sigma}}$ and $K ( x ) =\mathcal{F}^{-1} \varphi$, then
  \begin{align*}
     | ( P_{t} - \tmop{Id} ) u ( x ) | & = \left| t^{-d/ ( 2 \sigma )} \int K \left( \frac{x-y}{t^{1/ ( 2 \sigma )}} \right) ( u ( y ) -u ( x ) ) \mathd y \right| \\
     & \lesssim t^{-d/ ( 2 \sigma )} \int K \left( \frac{x-y}{t^{1/ ( 2 \sigma )}} \right) | y-x |^{\beta} \| u \|_{\beta} \mathd y \lesssim t^{\beta / ( 2 \sigma )} \| u \|_{\beta},
  \end{align*}
  where we identified $\CC^{\beta}$ with the space of H{\"o}lder continuous functions.
\end{proof}

Based on Lemma~\ref{lemma:heat flow smoothing} and Lemma~\ref{lemma:semigroup hoelder in time}, we derive the following Schauder estimates:

\begin{lemma}\label{lemma:schauder}
  Let $\sigma$ and $(P_t)_{t \geqslant 0}$ be as in Lemma~\ref{lemma:heat flow smoothing}. Assume that $v \in C_{T} \CC^{\beta}$ for some $\beta \in \mathbb{R}$ and $T>0$. Letting $V ( t ) = \int_{0}^{t} P_{t-s} v ( s ) \mathd s$, we have
  \begin{equation}\label{eq:schauder-heat}
     t^{\gamma} \| V ( t ) \|_{\beta +2 \sigma} \lesssim \sup_{s \in [ 0,t ]} ( s^{\gamma} \| v ( s ) \|_{\beta} )
  \end{equation}
  for all $\gamma \in [ 0,1 )$ and all $t \in [ 0,T ]$. If $ \beta \in (-2\sigma,0)$, then we also have
  \begin{equation}\label{eq:holder-est-heat}
     \| V \|_{C_{T}^{( \beta +2\sigma ) / ( 2 \sigma )} L^{\infty}} \lesssim \sup_{s \in [ 0,t ]} \| v ( s ) \|_{\beta} .
  \end{equation}
\end{lemma}

\begin{proof}
  Consider $\Delta_{q} V$ for some $q \geqslant 0$ and let $\delta \in [ 0,t/2 ]$. We decompose the integral into two parts:
  \[
     \Delta_{q} V ( t ) = \int_{0}^{t} P_{t-s} ( \Delta_{q} v ) ( s ) \mathd s= \int_{0}^{\delta} P_{s} ( \Delta_{q} v ) ( t-s ) \mathd s+  \int_{\delta}^{t} P_{s} ( \Delta_{q} v ) ( t-s ) \mathd s. 
  \]
  Letting $M= \sup_{s \in [ 0,t ]} ( s^{\gamma} \| v ( s ) \|_{\beta} )$, we estimate the first term by
  \begin{align*}
     \left\| \int_{0}^{\delta} P_{s} ( \Delta_{q} v ) ( t-s ) \mathd s \right\|_{L^{\infty}} & \leqslant \int_{0}^{\delta} 2^{-q \beta} \| v ( t-s ) \|_{\beta} \mathd s \leqslant 2^{-q \beta} M \int_{0}^{\delta} ( t-s)^{- \gamma} \mathd s \\
     & =M2^{-q \beta} t^{1- \gamma} \int_{0}^{\delta /t} \frac{\mathd s}{( 1-s )^{\gamma}} \lesssim M2^{-q \beta} t^{- \gamma} \delta ,
  \end{align*}
  using $| 1- ( 1- \delta /t )^{1- \gamma} | \lesssim \delta /t$ in the last step. On the other side, we can use Lemma~\ref{lemma:heat flow smoothing} to
  estimate the second term for $\varepsilon >0$ by
  \begin{align*}
     \left\| \int_{\delta}^{t} P_{s} ( \Delta_{q} v ) ( t-s ) \mathd s \right\|_{L^{\infty}} & \lesssim \int_{\delta}^{t} s^{-1- \varepsilon} 2^{-q ( \beta +2 \sigma ( 1+ \varepsilon ) )} \| v ( t-s ) \|_{\beta} \mathd s \\
     & \lesssim M2^{-q ( \beta +2 \sigma ( 1+ \varepsilon ) )} \int_{\delta}^{t} \frac{\mathd s}{s^{1+ \varepsilon} ( t-s )^{\gamma}} \\
     & =M2^{-q ( \beta +2 \sigma ( 1+ \varepsilon ) )} t^{- \varepsilon - \gamma} \int_{\delta /t}^{1} \frac{\mathd s}{s^{1+ \varepsilon} ( 1-s )^{\gamma}} \\
     & \lesssim M2^{-q ( \beta +2 \sigma ( 1+ \varepsilon ) )} t^{- \gamma} \delta^{- \varepsilon} =M2^{-q ( \beta +2 \sigma )} ( 2^{q2 \sigma} \delta  )^{- \varepsilon} t^{- \gamma} .
  \end{align*}
  If $2^{-q2 \sigma} \leqslant t/2$, we can take $\delta =2^{-q2 \sigma}$ to
  obtain $\| \Delta_{q} V ( t ) \|_{L^{\infty}} \lesssim Mt^{- \gamma} 2^{-q (\beta +2 \sigma )}$. If $2^{-q2 \sigma} >t/2$, we have $\| \Delta_{q} V ( t ) \|_{L^{\infty}} \leqslant M2^{-q \beta} t^{1- \gamma} \lesssim Mt^{- \gamma} 2^{-q ( \beta +2 \sigma )}$, and the first claim follows.
  
  As for the second claim, note that for $0 \leqslant s<t \leqslant T$ we have
  \[
      V ( t ) -V ( s ) = ( P_{t-s} - \tmop{Id} ) V ( s ) + \int_{s}^{t} P_{t-r} v ( r ) \mathd r,
  \]
  and therefore we can apply Lemma~\ref{lemma:semigroup hoelder in time} to obtain
  \begin{align*}
     \| V ( t ) -V ( s ) \|_{L^{\infty}} & \lesssim \| ( P_{t-s} - \tmop{Id} ) V( s ) \|_{L^{\infty}} + \int_{s}^{t} \| P_{t-r} v ( r ) \|_{L^{\infty}} \mathd r \\
     & \lesssim | t-s |^{( \beta +2\sigma ) / ( 2 \sigma )} \| V ( s ) \|_{\beta +2\sigma} \\
     &\quad +  \int_{s}^{t} \| v ( r ) \|_{\beta} \mathd r \lesssim_{T} | t-s |^{( \beta +2\sigma ) / ( 2 \sigma )} \sup_{r \in [ 0,t ]} \| v ( r ) \|_{\beta} ,
  \end{align*}
  where we used that $( \beta +2\sigma )/2\sigma \in ( 0,1 )$ and that $| t-s | \leqslant T$. This yields the second claim.
\end{proof}

When dealing with \textsc{rde}s, the convolution with the (fractional) heat kernel has a natural correspondence in the integral map.

\begin{lemma}\label{lemma:integral}
  Let $u \in \CC^{\alpha -1} ( \mathbb{R} )$ for some $\alpha \in ( 0,1 )$. Then there exists a unique $U \in \CC^{\alpha}_{\tmop{loc}} ( \mathbb{R} )$ such that $\mathD U=u$ and $U ( 0
  ) =0$. This antiderivative $U$ satisfies
  \begin{equation}\label{eq:integral estimate 1}
     | U ( t ) -U ( s ) | \lesssim | t-s |^{\alpha} \| u \|_{\alpha -1}
  \end{equation}
  for all $s,t \in \mathbb{R}$ with $| s-t | \leqslant 1$.We will use the notation $U ( t ) = \int_{0}^{t} u ( s ) \mathd s$ to denote this map, which is an extension of the usual definite integral. If the support of $u$ is contained in $[ -T,T ]$ for some $T>0$, then $U \in \CC^{\alpha}$ and
  \[
     \| U \|_{\alpha} \lesssim T \| u \|_{\alpha -1}.
  \]
\end{lemma}

\begin{proof}
  The second statement about compactly supported $u$ follows from the first statement by identifying $\CC^{\alpha}$ with the space of bounded H{\"o}lder continuous functions.
  
  As for the first statement, we define
  \[
     U ( t ) = \sum_{j \geqslant -1} \int_{0}^{t} \Delta_{j} u ( s ) \mathd s.
  \]
  If we can show~{\eqref{eq:integral estimate 1}}, then $U$ is indeed in $\CC^{\alpha}_{\tmop{loc}}$ and therefore in particular in $\CS'$. Since the derivative $\mathD$ is a continuous operator on $\CS'$, we then conclude that $\mathD U= \sum_{j} \Delta_{j} u=u$. Let therefore $s,t \in \mathbb{R}$ with $| s-t | \leqslant 1$. We have
  \[
     \left| \int_{s}^{t} \Delta_{j} u ( r ) \mathd r \right| \leqslant 2^{j (1- \alpha )} \| u \|_{\alpha -1} | t-s | .
  \]
  If $j \geqslant 0$, then $\Delta_{j} u= \mathD \mathD^{-1} ( \Delta_{j} u)$, where $\mathD^{-1}$ is the Fourier multiplier with symbol $1/ ( \iota  z )$, and therefore
  \[
     \left| \int_{s}^{t} \Delta_{j} u ( r ) \mathd r \right| = | \mathD^{-1} \Delta_{j} u ( t ) - \mathD^{-1} \Delta_{j} u ( s ) | \lesssim 2^{-j} \| \Delta_{j} u \|_{L^{\infty}} \lesssim 2^{-j \alpha} \| u \|_{\alpha -1} ,
  \]
  where we used the Bernstein inequality, Lemma~\ref{lemma:Bernstein}. If $j_{0}$ is such that $2^{-j_{0}} \leqslant | t-s | <2^{-j_{0} +1}$, then we use the first estimate for $j \leqslant j_{0}$ and the second estimate for $j>j_{0}$, and obtain
  \begin{align*}
     | U ( t ) -U ( s ) | &\leqslant \sum_{j \geqslant -1} \left| \int_{s}^{t} \Delta_{j} u ( r ) \mathd r \right| \lesssim \sum_{j \leqslant j_{0}} 2^{j ( 1- \alpha )} \| u \|_{\alpha -1} | t-s | + \sum_{j>j_{0}} 2^{-j \alpha} \| u \|_{\alpha -1} \\
     & \lesssim ( 2^{j_{0} ( 1- \alpha )} | t-s | +2^{-j_{0} \alpha} ) \| u \|_{\alpha -1} \simeq | t-s |^{\alpha} \| u \|_{\alpha -1} .
  \end{align*}
  Uniqueness is easy since every distribution with zero derivative is a constant function.
\end{proof}

\section{More commutator estimates}\label{sec:convolution commutators}

When applying the scaling argument to solve equations, we need to control the resonant product of the rescaled data. This can be done by relying on the following commutator estimate.

\begin{lemma}\label{lem:scaling commutator}
   Let $\alpha, \beta \in \R$ and $f, g \in \CS$. Then we have uniformly in $\lambda \in (0,1]$
   \[
      \| \Lambda_\lambda(f\reso g) - (\Lambda_\lambda f) \reso (\Lambda_\lambda g) \|_{\alpha+\beta} \lesssim \max\{\lambda^{\alpha+\beta}, 1\} \| f \|_\alpha \|g\|_\beta,
   \]
   and thus $ \Lambda_\lambda(\cdot\reso \cdot) - (\Lambda_\lambda \cdot) \reso (\Lambda_\lambda \cdot)$ extends to a bounded bilinear operator from $\CC^\alpha \times \CC^\beta$ to $\CC^{\alpha+\beta}$.
\end{lemma}
\begin{proof}
   We have $\Lambda_\lambda \Delta_j = \Lambda_\lambda \rho_j(\mathD) = \rho_j(\lambda^{-1} \mathD) \Lambda_\lambda$ for all $j \geqslant -1$. Let $k \in \N$ and $\lambda' \in (1/2,1]$ be such that $\lambda = \lambda' 2^{-k}$. Then
   \begin{align}\label{eq:scaling commutator pr1} \nonumber
      \Lambda_\lambda (f \reso g) & = \sum_{\substack{|i-j| \leqslant 1 \\ i,j \leqslant k}} \Lambda_\lambda ( \Delta_i f \Delta_j g) \\
      &\quad + \sum_{\substack{|i-j| \leqslant 1 \\ i,j > k}} \rho(2^{-i+k} \lambda'^{-1} \mathD) \Lambda_\lambda f \rho(2^{-j+k} \lambda'^{-1} \mathD) \Lambda_\lambda g.
   \end{align}
   The first sum is spectrally supported in a ball centered at zero (which does not depend on $k$ or $\lambda$), and therefore
   \[
      \bigg\|  \sum_{\substack{|i-j| \leqslant 1 \\ i,j \leqslant k}} \Lambda_\lambda ( \Delta_i f \Delta_j g) \bigg\|_{\alpha + \beta} \lesssim \sum_{\substack{|i-j| \leqslant 1 \\ i,j \leqslant k}} 2^{-i \alpha -j\beta} \|f\|_\alpha \|g\|_\beta \lesssim \max\{\lambda^{\alpha+\beta},1\} \|f\|_\alpha \|g\|_\beta.
   \]
   The second sum is the resonant paraproduct $(\Lambda_\lambda f \, \widetilde{\reso} \Lambda_\lambda g)$ with respect to the dyadic partition of unity $(\chi(\lambda'^{-1} \cdot), \rho(\lambda'^{-1}\cdot))$, except that the sum only starts in $i,j =1$. By Lemma~\ref{lem:paraproduct difference is bounded operator} we can therefore bound
   \[
      \bigg\| \sum_{\substack{|i-j| \leqslant 1 \\ i,j > k}} \rho(2^{-i+k} \lambda'^{-1} \mathD) \Lambda_\lambda f \rho(2^{-j+k} \lambda'^{-1} \mathD) \Lambda_\lambda g - (\Lambda_\lambda f) \reso (\Lambda_\lambda g) \bigg\|_{\alpha+\beta} \lesssim \|f \|_\alpha \|g\|_\beta.
   \]
\end{proof}

Next, we prove that it is possible to exchange paraproduct and time integration, at the price of introducing a smoother correction term:

\begin{lemma}\label{lemma:controlled old and new}
  Let $\alpha , \beta \in ( 0,1 )$ with $\alpha + \beta <1$. Let $u \in \CC^\alpha(\R, \R^{d \times n})$ and $v \in \CC^{\beta}(\R, \R^n)$. Then
  \[
     \left| \int_{s}^{t} ( u \lpara \partial_{t} v ) ( r ) \mathd r-u ( s ) ( v ( t ) -v ( s ) ) \right| \lesssim | t-s |^{\alpha + \beta} \| u  \|_{\alpha} \| v \|_{\beta} ,
  \]
  for all $s,t \in \mathbb{R}$ with $| t-s | \leqslant 1$, where we write $\int_{s}^{t} f ( r ) \mathd r= \int_{0}^{t} f ( r ) \mathd r- \int_{0}^{s} f ( r ) \mathd r$.
\end{lemma}

\begin{proof}
  Fix $s,t \in \mathbb{R}$ with $| s-t | \leqslant 1$. We can rewrite
  \[
     \int_{s}^{t} ( u \lpara \partial_{t} v ) ( r ) \mathd r-u ( s ) ( v ( t ) -v ( s ) ) = \sum_{j} \int_{s}^{t} [ S_{j-1} u ( r ) -u ( s ) ] \partial_{r} \Delta_{j} v ( r ) \mathd r.
  \]
  We will use two different estimates, one for large $j$ and one for small $j$. First note that
  \begin{align*}
     \left| \int_{s}^{t} [ S_{j-1} u ( r ) -u ( s ) ] \partial_{r} \Delta_{j} v ( r ) \mathd r \right| & \leqslant \left| \int_{s}^{t} [ S_{j-1} u ( r ) -S_{j-1} u ( s ) ]  \partial_{r} \Delta_{j} v ( r ) \mathd r \right| \\
     &\quad + \left| \int_{s}^{t} [ S_{j-1} u ( s ) -u ( s ) ] \partial_{r} \Delta_{j} v ( r ) \mathd r  \right| .
  \end{align*}
  Now $| S_{j-1} u ( r ) - S_{j-1} u ( s ) | \lesssim | r-s |^\alpha \| u \|_{\alpha}$, 
  and therefore
  \begin{align}\label{eq:controlled old and new pr1} \nonumber
     &\left| \int_{s}^{t} [ S_{j-1} u ( r ) -u ( s ) ] \partial_{r} \Delta_{j} v ( r ) \mathd r \right| \\ \nonumber
     &\hspace{50pt} \lesssim \left( \int_{s}^{t}  | r-s |^\alpha 2^{j ( 1- \beta )} \mathd r+ \int_{s}^{t} 2^{-j \alpha} 2^{j ( 1- \beta )} \mathd r \right)\! \| u \|_{\alpha} \| v \|_{\beta} \\
    &\hspace{50pt} \lesssim ( 2^{j ( 1- \beta)} | t-s |^{1+\alpha} +2^{j ( 1- \alpha - \beta )} | t-s | ) \| u \|_{\alpha} \| v \|_{\beta} .
  \end{align}
  On the other side, it follows from integration by parts that
  \begin{align}\label{eq:controlled old and new pr2} \nonumber
     &\left| \int_{s}^{t} [ S_{j-1} u ( r ) - u ( s ) ] \partial_{r}  \Delta_{j} v ( r ) \mathd r \right| \\ \nonumber
     &\hspace{15pt} \leqslant \left| \int_{s}^{t} [ S_{j-1} u ( r ) -S_{j-1} u ( s ) ] \partial_{r}  \Delta_{j} v ( r ) \mathd r \right|  + \left| \int_{s}^{t} [ S_{j-1} u ( s ) - u ( s ) ] \partial_{r}  \Delta_{j} v ( r ) \mathd r \right|\\ \nonumber
     &\hspace{15pt} \leqslant | (S_{j-1} u ( t ) - S_{j-1} u(s))  \Delta_{j} v ( t ) | + \left| \int_{s}^{t} \partial_{r} S_{j-1} u ( r )    \Delta_{j} v ( r ) \mathd r \right|\\ \nonumber
     &\hspace{15pt}\quad + |(S_{j-1} u ( s ) - u ( s ))(\Delta_j v(t) - \Delta_j v(s))| \\
     &\hspace{15pt}\lesssim \big( |t-s|^\alpha 2^{-j\beta} + |t-s|^{\alpha+\beta-\varepsilon} 2^{-j\varepsilon} + 2^{-j(\alpha+\beta)}\big)  \|u\|_\alpha \|v\|_\beta,
  \end{align}
  for all $\varepsilon \in [0,\alpha+\beta)$, where for the middle term we applied Lemma~\ref{lemma:integral}, which gives us
  \begin{align*}
      \left| \int_{s}^{t} \partial_{r} S_{j-1} u ( r ) \Delta_{j} v ( r ) \mathd r \right| & \lesssim |t-s|^{\alpha+\beta-\varepsilon} \|\partial_{r} S_{j-1} u ( r ) \Delta_{j} v ( r )\|_{\alpha+\beta-\varepsilon-1}\\
      & \lesssim |t-s|^{\alpha+\beta-\varepsilon} 2^{j(\alpha+\beta-\varepsilon-1)} \|\partial_{r} S_{j-1} u ( r ) \Delta_{j} v ( r )\|_{L^\infty} \\
      & \lesssim |t-s|^{\alpha+\beta-\varepsilon} 2^{-j\varepsilon} \|u\|_\alpha \|v\|_\beta.
  \end{align*}
%
%
  Let now $j_{0} \in \mathbb{N}$ be such that $2^{-j_{0}} \leqslant | t-s | <2^{-j_{0} +1}$. We use estimate~{\eqref{eq:controlled old and new pr1}} for
  $j \leqslant j_{0}$ and~{\eqref{eq:controlled old and new pr2}} for
  $j>j_{0}$ to obtain
  \begin{align*}
     & \left| \int_{s}^{t} ( u \lpara \partial_{t} v ) ( r ) \mathd r-u ( s ) ( v ( t ) -v ( s ) ) \right| \\
     &\hspace{50pt} \lesssim \sum_{j \leqslant j_{0}} ( 2^{j ( 1 - \beta )} | t-s |^{1+\alpha} +2^{j ( 1- \alpha - \beta )} | t-s | ) \| u \|_{\alpha} \| v \|_{\beta} \\
     &\hspace{50pt}\qquad + \sum_{j>j_{0}} ( |t-s|^\alpha 2^{-j\beta} + | t-s |^{\alpha + \beta - \varepsilon} 2^{-j \varepsilon} +2^{-j (\alpha+\beta)} ) \| u \|_{\alpha} \| v \|_{\beta} \\
     &\hspace{50pt} \simeq \| u \|_{\alpha} \| v \|_{\beta} | t-s |^{\alpha + \beta},
  \end{align*}
  where we used that $\alpha + \beta <1$.
\end{proof}

\section{A modified paralinearization theorem}\label{sec:paralin mod}

When solving singular PDEs with general nonlinearity, it is often useful to take the paracontrolled structure of the solution into account in the paralinearization theorem, as this allows us to obtain better bounds. Here we prove the result that we needed when solving the parabolic Anderson model.

\begin{lemma}\label{lem:mod paralin}
   Let $\alpha \in (0,1)$ and $\beta \in (0,\alpha]$ be such that $\alpha + \beta > 1$. Let $f \in \CC^\alpha$, $g \in \CC^{\alpha+\beta}$, and $F \in C^{3}_b$. Then
   \begin{equation}\label{eq:mod paralin 2}
      \| F(f+g) - F'(f+g)\lpara (f+g) \|_{\alpha+\beta} \lesssim \| F\|_{C^3_b} (1 + \|f \|_\alpha^{1+\beta/\alpha} + \| g \|_{L^\infty}^2)(1 + \| g \|_{\alpha+\beta}).
   \end{equation}
\end{lemma}
\begin{proof}
   Since $\|F'(f+g) \lpara g \|_{\alpha+\beta} \lesssim \| F\|_{C^1_b} \| g \|_{\alpha+\beta}$, it suffices to control $F(f+g) - F'(f+g) \lpara f$. We use the same decomposition as in the proof of Lemma~\ref{lemma:paralinearization}:
   \[
      F(f+g) - F'(f+g)\lpara f = \sum_{i \geqslant -1} [\Delta_i F(f+g) - S_{i-1} F'(f+g) \Delta_i f] = \sum_{i \geqslant -1} u_i
   \]
   with
  \[
     u_i ( x ) = \int K_{i} ( x-y ) K_{<i-1} ( x-z ) [F ( f ( y) + g(y) ) - F'(f(z) + g(z)) f(y)] \mathd y \mathd z
  \]
  and since $K_i(x-y)$ integrates to zero, we can replace the term in the square brackets by
  \[
     \{F ( f ( y ) + g(y) ) - F ( f (z) + g(y)) - F'(f(z)+g(z)) (f(y) - f(z))\} + F(f(z) + g(y)).
  \]
  Applying a first order Taylor expansion and using the fact that $g \in \CC^{\alpha+\beta}$ is Lipschitz continuous, the first term can be bounded by
  \begin{align*}
     &|F ( f ( y ) + g(y) ) - F ( f (z) + g(y)) - F'(f(z)+g(z)) (f(y) - f(z))| \\
     &\hspace{50pt}\lesssim \|F\|_{C^{1+\beta/\alpha}_b} \|f \|_\alpha  |z-y|^\alpha (\|f \|_\alpha^{\beta/\alpha} + \| g \|_{\alpha+\beta}^{\beta/\alpha}) ( |z-y|^{\beta/\alpha} + |x-y|^{\beta}).
  \end{align*}
  This leads to
  \begin{align}\label{eq:mod paralin pr1} \nonumber
     |u_i ( x )| & \lesssim \|F\|_{C^{1+\beta/\alpha}_b} \|f \|_\alpha (\|f \|_\alpha^{\beta/\alpha} + \| g \|_{\alpha+\beta}^{\beta/\alpha}) 2^{-i(\alpha+\beta)} \\
     &\quad + \Big| \int K_{i} ( x-y ) K_{<i-1} ( x-z ) F(f(z) + g(y)) \mathd y \mathd z \Big|.
  \end{align}
  To estimate the remaining integral, note that
  \[
     \Big| \int K_i(x-y) F(f(z) + g(y)) \mathd y \Big| \leqslant \| y \mapsto F(f(z) + g(y)) \|_{\alpha+\beta} 2^{-i(\alpha+\beta)}.
  \]
  Since the $C^3_b$ norm of $F(f(z) + \cdot)$ is bounded by $\|F \|_{C^3_b}$, we can apply Theorem~2.87 of \cite{Bahouri2011} to obtain that
  \[
     \| y \mapsto F(f(z) + g(y)) \|_{\alpha+\beta} \lesssim \|F \|_{C^3_b} (1 + \|g\|_{L^\infty}^2) (1 + \|g\|_{\alpha+\beta}),
  \]
  which yields~\eqref{eq:mod paralin 2}. Since \cite{Bahouri2011} deals with a more general situation, there the estimate is stated in a weaker form: it is only shown that
  \[
     \| y \mapsto F(f(z) + g(y)) \|_{\alpha+\beta} \le C(F,\|g\|_{L^\infty},\alpha+\beta)(1 + \|g\|_{\alpha+\beta}).
  \]
  But by reducing the proof to our special case we get the claimed form of $C(F,\|g\|_{L^\infty},\alpha+\beta)$.
\end{proof}

\paragraph{Acknowledgments.} 
During an Oberwolfach workshop in the summer of 2012, M. Hairer discussed with one of us (M.G.) his approach to extend rough path theory and  we would like to thank M. Hairer for suggesting the application to the two-dimensional non-linear parabolic Anderson model discussed in this paper.

The main part of the research was carried out while N.P. was employed by Humboldt-Universit\"at zu Berlin. M.G. is supported by a Junior fellowship of the Institut Universitaire de France (IUF) and  by the ANR Project ECRU (ANR-09-BLAN-0114-01). N.P. is supported by the Fondation Sciences Math\'ematiques de Paris (FSMP) and by a public grant overseen by the French National Research Agency (ANR) as part of the ``Investissements d'Avenir'' program (reference: ANR-10-LABX-0098).

\bibliographystyle{alpha}

\begin{thebibliography}{FGGR12}

\bibitem[BCD11]{Bahouri2011}
Hajer Bahouri, Jean-Yves Chemin, and Raphael Danchin.
\newblock {\em {Fourier analysis and nonlinear partial differential
  equations}}.
\newblock Springer, 2011.

\bibitem[BGN13]{Brzezniak2010}
Zdzis{\l}aw Brze{\'z}niak, Massimiliano Gubinelli, and Misha Neklyudov.
\newblock {Global evolution of random vortex filament equation}.
\newblock {\em Nonlinearity}, 26(9):2499, 2013.

\bibitem[BGR05]{Bessaih2005}
Hakima Bessaih, Massimiliano Gubinelli, and Francesco Russo.
\newblock {The evolution of a random vortex filament}.
\newblock {\em Ann. Probab.}, 33(5):1825--1855, 2005.

\bibitem[Bon81]{Bony1981}
Jean-Michel Bony.
\newblock {Calcul symbolique et propagation des singularites pour les
  {\'e}quations aux d{\'e}riv{\'e}es partielles non lin{\'e}aires}.
\newblock {\em Ann. Sci. {\'E}c. Norm. Sup{\'e}r. (4)}, 14:209--246, 1981.

\bibitem[CC13]{Catellier2013}
R{\'e}mi Catellier and Khalil Chouk.
\newblock Paracontrolled distributions and the 3-dimensional stochastic
  quantization equation.
\newblock {\em arXiv preprint arXiv:1310.6869}, 2013.

\bibitem[CF09]{Caruana2009}
Michael Caruana and Peter Friz.
\newblock {Partial differential equations driven by rough paths}.
\newblock {\em J. Differential Equations}, 247(1):140--173, 2009.

\bibitem[CFO11]{Caruana2011}
Michael Caruana, Peter~K. Friz, and Harald Oberhauser.
\newblock {A (rough) pathwise approach to a class of non-linear stochastic
  partial differential equations}.
\newblock {\em Ann. Inst. H. Poincar\'e Anal. Non Lin\'eaire}, 28(1):27--46,
  2011.

\bibitem[CG06]{Chemin2006}
Jean-Yves Chemin and Isabelle Gallagher.
\newblock {On the global wellposedness of the 3-{D} {N}avier-{S}tokes equations
  with large initial data}.
\newblock {\em Ann. Sci. {\'E}cole Norm. Sup. (4)}, 39(4):679--698, 2006.

\bibitem[CG14]{Chouk2013}
Khalil Chouk and Massimiliano Gubinelli.
\newblock {Rough sheets}.
\newblock {\em arXiv preprint arXiv:1406.7748}, 2014.

\bibitem[CGP15]{Chouk2015}
Khalil Chouk, Jan Gairing, and Nicolas Perkowski.
\newblock An invariance principle for the two-dimensional parabolic Anderson model with small potential.
\newblock in preparation, 2015.

\bibitem[CM94]{Carmona1994}
Ren\'e A. Carmona and S.A. Molchanov.
\newblock {\em {Parabolic Anderson problem and intermittency}}.
\newblock American Mathematical Society, 1994.

\bibitem[DF12]{Diehl2012}
Joscha Diehl and Peter Friz.
\newblock {Backward stochastic differential equations with rough drivers}.
\newblock {\em Ann. Probab.}, 40(4):1715--1758, 2012.

\bibitem[DGT12]{Deya2012}
Aur{\'e}lien Deya, Massimiliano Gubinelli, and Samy Tindel.
\newblock {Non-linear rough heat equations}.
\newblock {\em Probab. Theory Related Fields}, 153(1-2):97--147, 2012.

\bibitem[FGGR12]{Friz2012}
Peter~K Friz, Benjamin Gess, Archil Gulisashvili, and Sebastian Riedel.
\newblock {Spatial rough path lifts of stochastic convolutions}.
\newblock {\em arXiv preprint arXiv:1211.0046}, 2012.

\bibitem[FO11]{Friz2011b}
Peter Friz and Harald Oberhauser.
\newblock {On the splitting-up method for rough (partial) differential
  equations}.
\newblock {\em J. Differential Equations}, 251(2):316--338, 2011.

\bibitem[FV10]{Friz2010}
Peter Friz and Nicolas Victoir.
\newblock {\em {Multidimensional stochastic processes as rough paths. Theory
  and applications}}.
\newblock Cambridge University Press, 2010.

\bibitem[GIP14]{Gubinelli2013}
Massimiliano Gubinelli, Peter Imkeller, and Nicolas Perkowski.
\newblock A {F}ourier approach to pathwise stochastic integration.
\newblock {\em arXiv preprint arXiv:1410.4006}, 2014.

\bibitem[GLT06]{Gubinelli2006}
Massimiliano Gubinelli, Antoine Lejay, and Samy Tindel.
\newblock {Young integrals and {SPDE}s}.
\newblock {\em Potential Anal.}, 25(4):307--326, 2006.

\bibitem[GP15]{Gubinelli-KPZ-2013}
Massimiliano Gubinelli and Nicolas Perkowski.
\newblock {KPZ} reloaded.
\newblock in preparation, 2015.

\bibitem[GP15a]{Gubinelli2015EBP}
Massimiliano Gubinelli and Nicolas Perkowski.
\newblock Lectures on singular stochastic {PDE}s.
\newblock arXiv preprint arXiv:1502.00157, 2015.

\bibitem[Gub04]{Gubinelli2004}
Massimiliano Gubinelli.
\newblock {Controlling rough paths}.
\newblock {\em J. Funct. Anal.}, 216(1):86--140, 2004.

\bibitem[Gub12]{Gubinelli2012a}
Massimiliano Gubinelli.
\newblock {Rough solutions for the periodic {K}orteweg--de {V}ries equation}.
\newblock {\em Commun. Pure Appl. Anal.}, 11(2):709--733, 2012.

\bibitem[Hai11]{Hairer2011}
Martin Hairer.
\newblock {Rough stochastic {PDE}s}.
\newblock {\em Comm. Pure Appl. Math.}, 64(11):1547--1585, 2011.

\bibitem[Hai13]{Hairer2013b}
Martin Hairer.
\newblock {Solving the KPZ equation}.
\newblock {\em Ann. Math.}, 178(2):559--664, 2013.

\bibitem[Hai14]{Hairer2013a}
Martin Hairer.
\newblock {A theory of regularity structures}.
\newblock {\em Invent. Math.}, 198(2):269--504, 2014.

\bibitem[HMW14]{Hairer2012}
Martin Hairer, Jan Maas, and Hendrik Weber.
\newblock {Approximating rough stochastic PDEs}.
\newblock {\em Comm. Pure Appl. Math.}, 67(5):776--870, 2014.

\bibitem[Hu2002]{Hu2002}
Yaozhong Hu.
\newblock {Chaos expansion of heat equations with white noise potentials}.
\newblock {\em Potential Anal.}, 16(1):45--66, 2002.

\bibitem[HW13]{Hairer2013}
Martin Hairer and Hendrik Weber.
\newblock {Rough {B}urgers-like equations with multiplicative noise}.
\newblock {\em Probab. Theory Related Fields}, 155(1-2):71--126, 2013.

\bibitem[Jan97]{Janson1997}
Svante Janson.
\newblock {\em {Gaussian {H}ilbert spaces}}, volume 129 of {\em {Cambridge
  Tracts in Mathematics}}.
\newblock Cambridge University Press, Cambridge, 1997.

\bibitem[K\"on15]{Koenig2015}
Wolfgang K\"onig.
\newblock {\em {The parabolic Anderson model}}, in preparation.
\newblock available at \url{http://www.wias-berlin.de/people/koenig/www/PAMsurveyBook.pdf}, 2015.

\bibitem[KPZ86]{Kardar1986}
Mehran Kardar, Giorgio Parisi, and Yi-Cheng Zhang.
\newblock {Dynamic scaling of growing interfaces}.
\newblock {\em Physical Review Letters}, 56(9):889--892, 1986.

\bibitem[LCL07]{Lyons2007}
Terry~J. Lyons, Michael Caruana, and Thierry L{\'e}vy.
\newblock {\em {Differential equations driven by rough paths}}, volume 1908 of
  {\em {Lecture Notes in Mathematics}}.
\newblock Springer, Berlin, 2007.

\bibitem[LQ02]{Lyons2002}
Terry Lyons and Zhongmin Qian.
\newblock {\em {System control and rough paths}}.
\newblock Oxford University Press, 2002.

\bibitem[Lyo98]{Lyons1998}
Terry~J. Lyons.
\newblock {Differential equations driven by rough signals}.
\newblock {\em Rev. Mat. Iberoam.}, 14(2):215--310, 1998.

\bibitem[NT11]{Nualart2011}
David Nualart and Samy Tindel.
\newblock {A construction of the rough path above fractional {B}rownian motion
  using {V}olterra's representation}.
\newblock {\em Ann. Probab.}, 39(3):1061--1096, 2011.

\bibitem[Per14]{Perkowski2014Thesis}
Nicolas Perkowski.
\newblock {\em Studies of Robustness in Stochastic Analysis and Mathematical
  Finance}.
\newblock PhD thesis, Humboldt-Universit{\"a}t zu Berlin, 2014.

\bibitem[ST87]{Schmeisser1987}
Hans-J{\"u}rgen Schmeisser and Hans Triebel.
\newblock {\em {Topics in {F}ourier analysis and function spaces}}, volume~42.
\newblock Akademische Verlagsgesellschaft Geest \& Portig K.-G., Leipzig, 1987.

\bibitem[Tei11]{Teichmann2011}
Josef Teichmann.
\newblock {Another approach to some rough and stochastic partial differential
  equations}.
\newblock {\em Stoch. Dyn.}, 11(2-3):535--550, 2011.

\bibitem[Tri06]{Triebel2006}
Hans Triebel.
\newblock {\em {Theory of function spaces. {III}}}, volume 100 of {\em
  {Monographs in Mathematics}}.
\newblock Birkh{\"a}user Verlag, Basel, 2006.

\bibitem[Unt10a]{Unterberger2010a}
J{\'e}r{\'e}mie Unterberger.
\newblock {A rough path over multidimensional fractional Brownian motion with
  arbitrary Hurst index by Fourier normal ordering}.
\newblock {\em Stochastic Processes and their Applications}, 120(8):1444--1472,
  2010.

\bibitem[Unt10b]{Unterberger2010}
J{\'e}r{\'e}mie Unterberger.
\newblock {H{\"o}lder-Continuous Rough Paths by Fourier Normal Ordering}.
\newblock {\em Comm. Math. Phys.}, 298(1):1--36, 2010.

\end{thebibliography}

\end{document}